\documentclass{amsart}

\usepackage{amsthm,amsmath,amssymb,enumitem,caption}
\usepackage{stmaryrd}
\usepackage{graphicx,hyperref,fontenc,inputenc}

\usepackage{array,float}

\usepackage[all,cmtip]{xy}
\usepackage[justification=centering,width=\linewidth,]{caption}
\usepackage{float}
\usepackage{wrapfig}

\usepackage{comment}
\specialcomment{solution}{
\begin{small} \underline{Solution}.
}{
\end{small}
}
\specialcomment{hint}{
\begin{small} \textbf{\underline{Hint}}.
}{
\end{small}
}
\excludecomment{solution}

\restylefloat{table}

\restylefloat{figure}
\numberwithin{figure}{subsection}
\numberwithin{table}{subsection}

\numberwithin{equation}{subsection}

\usepackage{multirow}
\usepackage{varwidth} 
\newcolumntype{L}{>{\centering\arraybackslash}m{1.7cm}}

\newcommand{\Z}{\mathbb Z}
\newcommand{\Q}{\mathbb Q}
\newcommand{\R}{\mathbb R}
\newcommand{\C}{\mathbb C}

\newcommand{\F}{\mathbb F}
\newcommand{\Ps}{\mathbb{P}}
\newcommand{\Pone}{\Ps^1}
\newcommand{\A}{\mathbb{A}}
\newcommand{\Aone}{\A^1}

\newcommand{\Zp}{\Z_p}
\newcommand{\Cp}{\C_p}
\newcommand{\Qp}{\Q_p}

\newcommand{\Qpbar}{\overline{\Q}_p}
\newcommand{\Zpbar}{\overline{\Z}_p}
\newcommand{\Fpbar}{\overline{\F}_p}

\newcommand{\Rplus}{\R_{>0}}

\newcommand{\Spa}{\operatorname{Spa}}

\newcommand{\supp}{\operatorname{supp}}
\newcommand{\Spec}{\operatorname{Spec}}
\newcommand{\maxSpec}{\operatorname{max-Spec}}
\newcommand{\Frac}{\operatorname{Frac}}
\newcommand{\Spv}{\operatorname{Spv}}
\newcommand{\Cont}{\operatorname{Cont}}
\newcommand{\disc}{\operatorname{disc}}
\newcommand{\img}{\operatorname{img}}
\newcommand{\Gauss}{\text{Gauss}}

\newcommand{\id}{\mathrm{id}}

\newcommand{\val}[1]{|#1|}
\newcommand{\bval}{\val{-}}
\newcommand{\norm}[1]{{|\!|#1|\!|}}
\newcommand{\bnorm}{{\norm{-}}}
\newcommand{\triv}{\mathrm{triv}}
\newcommand{\ord}{\operatorname{ord}}

\newcommand{\lr}[1]{\langle{#1}\rangle}

\newcommand{\im}{\operatorname{im}}

\newcommand{\calO}{\mathcal O}
\newcommand{\OCp}{\calO_{\Cp}}

\newcommand{\frakm}{\mathfrak m}
\newcommand{\frakp}{\mathfrak p}
\newcommand{\frakq}{\mathfrak q}
\newcommand{\frakQ}{\mathfrak Q}
\newcommand{\fraka}{\mathfrak a}

\newcommand{\term}[1]{\textbf{#1}}

\theoremstyle{plain}
\newtheorem{theorem}[subsubsection]{Theorem}
\newtheorem{lemma}[subsubsection]{Lemma}

\newtheorem{proposition}[subsubsection]{Proposition}
\newtheorem{definition}[subsubsection]{Definition}
\theoremstyle{definition}

\newtheorem{warning}[subsubsection]{Warning}
\newtheorem{example}[subsubsection]{Example}
\newtheorem{exercise}{Exercise}[subsection]

\begin{document}

\title{Huber rings and valuation spectra}

\author{John Bergdall}
\address{University of Arkansas\\Department of Mathematical Sciences\\850 W Dickson Street, Fayetteville, AR.\\
USA.}
\email{bergdall@uark.edu}




\maketitle

\setcounter{tocdepth}{1}
\tableofcontents

\section*{Introduction}

These notes expand on a four-hour lecture course given in Heidelberg in March 2023, as part of the ``Spring School on non-Archimedean Geometry and Eigenvarieties''. They are designed for graduate students and other learners. We introduce Huber rings and valuation theory alongside frequent examples. The notes are largely self-contained, though many details are given in exercises found following each lecture.

\subsection*{Context}

Hensel developed the $p$-adic numbers and their analysis in the waning years of the 19th century. Tate's theory of rigid analytic spaces dates to the 1960's \cite{Tate-RigidAnalytic}. The $p$-adic numbers allow for number theory modeled on power series expansions. Tate's theory models analytic geometry over the $p$-adic numbers, building spaces such as discs, annuli, and more, along with robust definitions of their rings of analytic functions. These models are applied to study problems in both geometry and number theory. One original motivation was uniformizing $p$-adic elliptic curves with split multiplicative reduction, now called Tate curves, via rigid analytic maps, in analogy with complex uniformization of (all) elliptic curves over the complex numbers.

Tate develops rigid analytic spaces using a class of rings, now called affinoid algebras, and their maximal ideal spectra. He equips these spaces with presheaves of functions, which he proves are actually sheaves. More precisely, the sheaves are sheaves only for a so-called Grothendieck topology. This major caveat is responsible for significant challenges in learning and using Tate's theory.

Three new theories were developed starting in the late 1980's:
\begin{enumerate}[label=(\roman*)]
\item Raynaud's formal models \cite{Bosch-RigidGeometry1,Bosch-RigidGeometry2,Bosch-RigidGeometry3,Bosch-RigidGeometry4}.
\item Berkovich's analytic spaces \cite{Berkovich-SpectralTheory,Bekovich-EtaleCohomology}.
\item Huber's adic spaces \cite{Huber-ContinuousValuations,Huber-Generalization,Huber-EtaleCohomology}.
\end{enumerate}

Conrad's four lectures at the 2007 Arizona Winter School \cite{Conrad-RigidGeometry} focused on Tate's theory, along with the work of Bosch--L\"utkebohmert--Raynaud and Berkovich. Our lectures in Heidelberg, and those of our colleagues H\"ubner \cite{Hubner-AdicSpaces}, Johansson \cite{Johansson-Sheaves} and Heuer \cite{Heuer-PerfectoidLectures}, discuss Huber's theory and its applications.

Why write these notes now? First, interest in adic spaces has exploded since they became a pillar for Scholze's  perfectoid spaces \cite{Scholze-Perfectoid,ScholzeWeinstein-Berkeley}. For instance, the 2017 Arizona Winter School was dedicated to perfectoid spaces, and those lectures necessarily included only a rapid introduction to adic spaces \cite{AWS-Perfectoid}. Second, eigenvarieties are the Spring School's second topic. These are traditionally developed as rigid analytic spaces, in Tate's style, by Hida, Coleman and Mazur, Buzzard, and many more. Recent works \cite{AndreattaIovitaPilloni-Halo,JohanssonNewton-Extended,Gulotta-Thesis}, however, extend eigenvarieties to characteristic $p$ local fields. All those works require Huber's perspective.

\subsection*{Motivation}
The first two lectures focus on spaces of valuations, on which Huber's theory is based. As background, we recall Tate's spaces and how to shift into a valuative mindset. We work over $\Cp$, the $p$-adic complex numbers. The field $\Cp$ is complete for the $p$-adic norm $\bnorm_p$ {\em and} algebraically closed, which makes geometry more clear.

The fundamental ring in Tate's theory is the Tate algebra $\Cp\langle w \rangle$. It is defined as the ring of series
\begin{equation*}
f = a_0 + a_1 w + a_2w^2 + \dotsb \in \Cp\llbracket w \rrbracket   
\end{equation*}
with the property that
\begin{equation}\label{eqn:limit-equation}
\lim_{i\rightarrow \infty} \norm{a_i}_p = 0.
\end{equation}

Tate models the closed unit disc over $\Cp$ as maximal ideals in $\Cp\lr w$. To explain, if $f \in \Cp\lr w$ and $\norm{\alpha}_p\leq 1$, then $f(\alpha)$ converges by \eqref{eqn:limit-equation}. The evaluation map
\begin{align}\label{eqn:evaluation-map}
\Cp\lr w &\xrightarrow{f \mapsto f(\alpha)} \Cp
\end{align}
therefore exists and has kernel $\lr{w-\alpha}\in \maxSpec(\Cp\lr{w})$. The \term{Weierstrass preparation theorem} implies all maximal ideals arise this way. This gives a bijection
\begin{equation*}
\{\alpha \in \Cp \mid \norm{\alpha}_p\leq 1\} \longleftrightarrow \maxSpec(\Cp\lr w).
\end{equation*}

Maximal ideals also detect inequalities needed for analytic geometry. For instance, $\norm{\alpha}_p \leq \frac{1}{p}$ if and only if $w-\alpha$ generates a maximal ideal in the larger ring
\begin{align*}
\Cp\lr{\frac{w}{p}}  &= \{b_0 + b_1\frac{w}{p} + b_2(\frac{w}{p})^2 + \dotsb \mid \lim_{i\rightarrow \infty} \norm{b_i}_p = 0\}\\
&= \{a_0 + a_1w + a_2w^2 + \dotsb \mid \lim_{i\rightarrow \infty} \norm{a_i}_pp^{-i} = 0\}.
\end{align*}
The ring $\Cp\lr{\frac{w}{p}} \cong \Cp\lr{w,v}/\langle pw-v\rangle$ is an example of a $\Cp$-affinoid algebra. Rigid spaces are glued from affinoid spaces, which are the maximal ideal spectra of affinoid algebras, in analogy to how schemes are glued from the prime ideal spectral of rings.

This model for geometry faces a major technical issue. It is too easily disconnected. In Tate's theory, 
a disc of radius $p^{-s} < 1$ is given by
\begin{equation*}
\{\norm{\alpha}_p \leq p^{-s} \} \leftrightarrow \maxSpec(\Cp\lr{\frac{w}{p^s}}).
\end{equation*}
These cover the open unit disc $\{\norm{\alpha}_p < 1\}$. The boundary of the unit disc is
\begin{equation*}\label{pageref:tate-decomposition}
\{\norm{\alpha}_p = 1\} \leftrightarrow \maxSpec(\Cp\lr{w,w^{-1}}) = \maxSpec(\Cp\lr{w,v}/\lr{wv-1}).
\end{equation*} 
Thus, the closed unit disc decomposes
\begin{equation}\label{eqn:tates-problem}
\{\norm{\alpha}_p \leq 1\} = \{\norm{\alpha}_p = 1\} \cup \bigcup_{s > 0}  \{\norm{\alpha}_p \leq p^{-s} \},
\end{equation}
with each piece being affinoid. 

Why is this disconnection an issue? A na\"ive sheaf theory would produce $\Cp\lr w$ as the ring of functions on $\{\norm{\alpha}_p \leq 1\}$. The disconnection \eqref{eqn:tates-problem} would then say that a series $\Cp\lr w$ can theoretically be defined by prescribing an analytic series on the disc's ``boundary'' $\{\norm{\alpha}_p = 1\}$ and, independently, a compatible collection of series on disc $\{\norm{\alpha}_p \leq p^{-s} \}$ with $s > 0$. So, a na\"ive sheaf theory would allow a single series that identically vanishes on the interior of the disc and not on the boundary. But, such a series is disallowed by the Weierstrass preparation theorem. Tate's solution was to use the language of Grothendieck topologies to disallow coverings such as \eqref{eqn:tates-problem}.

Huber proposes a different model of $p$-adic geometry. (Berkovich's approach has the same origin.)  Returning to the evaluation maps \eqref{eqn:evaluation-map}, if $\norm{\alpha}_p \leq 1$, we define a (semi-)norm $\bval_\alpha$ on $\Cp\langle w \rangle$ by 
\begin{equation}\label{eqn:motivation-valuation-alpha}
\val{f}_\alpha = \norm{f(\alpha)}_p.
\end{equation}
The maximal ideal $\lr{w-\alpha}$ is equal to the \term{support} of $\bval_\alpha$, which is the set of  $f$ such that $\val{f}_\alpha = 0$. So, Tate's ``points'' are recovered directly from these norms. However, there are more norms on $\Cp\lr{w}$. For instance, there is the \term{Gauss norm}
\begin{equation}\label{eqn:gauss-norm}
\norm{a_0+a_1w + a_2w^2 + \dotsb}_{\Gauss} = \max_{i\geq 0} \norm{a_i}_p.
\end{equation}
This norm is special because it is a norm that gives $\Cp\lr w$ the structure of a complete topological ring. It is of a different nature than $\bval_\alpha$ in the sense its support is $\{0\}$.

Huber's theory goes even further. It considers more general objects called \term{valuations}. These satisfy the axioms of norms, except their target is not always the non-negative real numbers. Huber models rigid analytic geometry on spaces of {\em continuous} valuations on topological rings, such as $\Cp\lr w$.

\subsection*{Lecture contents}

The first two lectures address Huber rings and their continuous valuation spectra. By the end of the second lecture, we explain:
\begin{enumerate}[label=(\roman*)]
\item Huber's continuous valuation spectra and how localization theory detects subdiscs such as $\{\norm{\alpha}_p\leq \frac{1}{p}\}$.\label{enum-part:inequalities-good}
\item Huber's natural model for the closed unit disc will contain at least one point is that {\em not} naturally part of the ``covering'' seen in  \eqref{eqn:tates-problem}.\label{enum-part:missing-points-found}
\end{enumerate}
Point \ref{enum-part:inequalities-good} supports comparing Huber's theory with Tate's. Point \ref{enum-part:missing-points-found} suggests Huber's theory treats coverings and analytic functions differently than Tate's.

H\"ubner's initial lecture on adic spectra \cite{Hubner-AdicSpaces} comes between our second and third lectures. The reader should read that first and then come back to our third lecture, where we detail further constructions with adic spectra.

Our fourth lecture is independent from those of our colleagues. In it, we analyze the closed unit disc, focusing on a systematic explanation of the point referred to in \ref{enum-part:missing-points-found} above.

\subsection*{References}

Huber rings and their valuation theory are the topic of Huber's papers \cite{Huber-ContinuousValuations, Huber-Generalization}. Other published references include Scholze's seminal work \cite{Scholze-Perfectoid} and Weinstein's lectures at the 2017 Arizona Winter School \cite{AWS-Perfectoid}. In addition, many have learned with the help of unpublished notes of Wedhorn \cite{Wedhorn-AdicSpaces}, Conrad \cite{Conrad-PerfectoidSeminarNotes}, and Morel \cite{Morel-AdicSpaceNotes}. These resources are all more ambitious than ours. We hope our lectures, in fact, introduce and complement more advanced texts.

\subsection*{Acknowledgements}
We thank Eugen Hellmann, Judith Ludwig, Sujatha Ramdorai, and Otmar Venjakob for their organizing efforts and invitation to speak at the Spring School. We also thank the Algebra Seminar at the University of Arkansas, who listened to practice lectures. Lance Miller, in particular, made significant comments that greatly improved these notes. Our Spring School colleagues Ben Heuer, Katharina H\"ubner, and Christian Johansson were remarkably helpful. They suggested devoting space to the adic Nullstellensatz (Theorem \ref{theorem:Aplus-recovery}), which adds weight to this text. Johansson and Ludwig also helped in key ways, during the School, with the algebraic arguments in Section \ref{section:some-points}. H\r{a}vard Damm-Johnsen and Milan Mal\v{c}i\'c did excellent work as course assistants. Finally, an anonymous reviewer, Zachary Feng, Damm-Johnsen, Kalyani Kansal, Ken Lee, and Gautier Ponsinet are thanked for comments during lectures or via email that improved these notes.

Financial support for travel to Heidelberg was partially provided by the University of Arkansa's Office of the Provost and Office for Faculty Affairs. While writing these notes, the author's work was also supported by NSF grant DMS-2302284 and a Simons Travel Support for Mathematicians Gift.

\section{Huber rings}\label{section:huber}

The primary goal of this lecture is defining Huber rings and giving their initial properties. We emphasize examples and material on bounded-ness and localization.

\subsection{Definition}

\begin{definition}\label{definition:huber-ring}
A \term{Huber ring} is a topological ring $A$ that contains an open subring $A_0$ for which there is an ideal $I \subseteq A_0$ such that:
\begin{enumerate}[label=(\alph*)]
\item The topology on $A_0$ (and thus on $A$) is the $I$-adic topology.\label{enum-part:huber-defn-adic}
\item The ideal $I$ is finitely generated.\label{enum-part:huber-defn-finitely-generated}
\end{enumerate}
\end{definition}

Huber rings are called f-adic rings in the original work \cite{Huber-ContinuousValuations}.  The ``-adic'' refers to condition \ref{enum-part:huber-defn-adic} in Definition \ref{definition:huber-ring}, while the ``f-'' recognizes the finiteness assumption in part \ref{enum-part:huber-defn-finitely-generated}. The shift toward the term ``Huber ring'' follows the introduction of perfectoid spaces and derivative works. The main benefit of the new name is that the ``f'' in f-adic cannot be confused with an italicized ``$f$'', which frequently represents a mathematical object, such as a function.

In \ref{enum-part:huber-defn-adic}, saying $A$ has the $I$-adic topology means a subset $U \ni 0$ is open if and only if $I^n \subseteq U$ for some $n$. The open sets around non-zero elements are determined by translation, since $A$ is a topological ring. Note, however, that $I \subseteq A$ is an open additive subgroup in $A$. It has the structure of an ideal over $A_0$, only.

We call $A_0$ a \term{ring of definition} and $I \subseteq A_0$ an \term{ideal of definition}. The pair $(A_0,I)$ is a \term{pair of definition}. Pairs of definition are auxiliary structures. They exist but need not be specified. Also, there is choice involved. For instance, if $(A_0,I)$ is a pair of definition, then so is $(A_0,I^n)$ for any $n \geq 1$.

\subsection{Examples}\label{sub:examples}

\begin{example}\label{example:ring}
Let $A$ be any ring and $I$ any finitely generated ideal in $A$. We equip $A$ with the $I$-adic topology to make $A$ a topological ring. Since $I$ is finitely generated, $A$ is a Huber ring with pair of definition $(A,I)$.

As a specific case, let $A = \Zp\llbracket w \rrbracket   $. Unlike Tate's affinoid algebras, this is a power series ring without any convergence conditions. It is noetherian and local, with maximal ideal $\frakm = \langle p,w \rangle$. The $\frakm$-adic topology on $A$ turns $A$ into a Huber ring.
\end{example}

\begin{example}\label{example:padic-numbers} The field of $p$-adic numbers $\Qp$ is not covered by Example \ref{example:ring}. Recall $\Qp$ is a topological field with its topology defined by the $p$-adic norm $\bnorm_p$. A neighborhood basis of zero is given by the open balls $\{\norm{\alpha}_p \leq p^{-n}\}$. Therefore, 
\begin{equation*}
\Zp  = \{\alpha \in \Qp \mid \norm{\alpha}_p \leq 1\}
\end{equation*}
is open in $\Qp$. Moreover, for $n \geq 0$ and $\alpha \in \Zp$, we have
\begin{equation*}
\alpha \in p^n \Zp \iff \norm{\alpha}_p \leq p^{-n}.
\end{equation*}
Therefore, the topology defined by $\bnorm_p$ on $\Zp$ is the same as the topology induced by the (principal) ideal $p\Zp$. We have shown $\Qp$ is a Huber ring by identifying $(\Zp,p\Zp)$ as a pair of definition. It is even a Tate ring. See Section \ref{subsec:tate-rings}.
\end{example}

\begin{example}\label{example:nonarchimedean-field}
A \term{non-Archimedean field} is a complete topological field $K$ whose topology is defined by a non-trivial non-Archimedean norm 
\begin{equation*}
\bnorm: K \rightarrow \R_{\geq 0}.
\end{equation*}
Such $K$ are Huber rings, just as for $K = \Qp$. First, the subring
\begin{equation*}
A_0 = \calO_K := \{\alpha \in K \mid \norm \alpha \leq 1\}
\end{equation*}
is open. Second, since $\bnorm$ is non-trivial, there exists $\alpha \in K^\times$ such that $\norm{\alpha} \neq 1$. Define $\varpi = \alpha^{\pm 1}$, making sure $\norm{\varpi} < 1$. Since $\norm{\varpi^n} \rightarrow 0$ as $n\rightarrow +\infty$, the norm topology on $A_0$ is equivalent to the topology defined by the ideal $\varpi A_0 \subseteq A_0$. There is a wide range of choices for $\varpi$. Each such choice is called a \term{pseudo-uniformizer} for $K$. Examples of non-Archimedean fields include perfectoid fields discussed in Heuer's lectures \cite{Heuer-PerfectoidLectures}.

The field $K = \Cp$ is important to keep in mind. By definition, $\Cp$ is the completion of an algebraic closure $\Qpbar$ of $\Qp$ with respect to the $p$-adic norm. The ring $A_0 = \OCp$ is a local ring whose norm topology is defined by the principal ideal $p\OCp$. This is {\em not} the topology defined by the maximal ideal
\begin{equation*}
\frakm_{\OCp} = \{x \in \OCp \mid \norm{x}_p < 1\}.
\end{equation*}
Indeed, since $\Cp$ is algebraically closed we have $\frakm_{\OCp}^2 = \frakm_{\OCp}$. Therefore, the powers of $\frakm_{\OCp}$ do not shrink toward zero as the powers of $p\OCp$ do.
\end{example}

\begin{example}\label{example:one-variable-tate-algebra}
Let $K$ be a non-Archimedean field with norm $\bnorm$. Define
\begin{equation*}
A = K\langle w \rangle = \left\{ \sum_{i=0}^\infty a_i w^i \in K\llbracket w \rrbracket    \mid \lim_{i\rightarrow \infty} \norm{a_i} = 0\right\}.
\end{equation*}
This is called a Tate algebra in one variable, in honor of its role in Tate's rigid analytic geometry. It is a topological ring, with topology defined by a norm $\bnorm_{\Gauss}$
\begin{equation*}
\norm{a_0+a_1w + a_2w^2 + \dotsb}_{\Gauss} = \max_{i\geq 0} \norm{a_i}
\end{equation*}
that we first encountered in \eqref{eqn:gauss-norm}. This norm is called the \term{Gauss norm}, presumably because proofs that $\bnorm_{\Gauss}$ is multiplicative  resemble proofs of Gauss's lemma on irreducibility of integer polynomials over the rational field. The topological ring $A$ is a Huber ring. A ring of definition is
\begin{align*}
A_0 = \calO_K\lr w &= \left\{\sum_{i=0}^{\infty} a_i w^i \in A \mid a_i \in \calO_K \text{ for all $i$}\right\}\\
&= \{f \in A \mid \norm{f}_{\Gauss} \leq 1\}.
\end{align*}
An ideal of definition is 
\begin{equation*}
\varpi A_0 = \left\{\sum_{i=0}^{\infty} a_i w^i \in A \mid a_i \in \varpi \calO_K \text{ for all $i$}\right\} = \{f \in A\mid \norm{f}_{\Gauss} < \norm{\varpi} \}.
\end{equation*}
Here, $\varpi$ is a pseudo-uniformizer for $K$ as in Example \ref{example:nonarchimedean-field}. If $\norm{-}$ is a discrete norm, then the inequality in the prior displayed equation becomes $\norm{f}_{\Gauss} < 1$.
\end{example}

\begin{warning}\label{warning:topological-ring}
In these examples, we have been careful to define a topological ring first and  assert the Huber ring property second. More accurately, in each case we first defined a ring $A$ with a topology. Second, we left as Exercise \ref{exercise:topological-ring-check} to check each $A$ was actually a topological ring. And, third, we identified a pair of definition $(A_0,I)$. 

Often you will find Huber rings $A$ defined only using a pair of definition $(A_0,I)$. After all, taking a topological ring $A_0$ and declaring $A_0\subseteq A$ to be open is a valid way to describe a topology on a ring $A$. It just may not define a topological ring! This issue arises in practice, with rational localizations. So, we propose Exercises \ref{exercise:subring-Iadic}-\ref{exercise:bad-topological-rings} as early work with topological rings.
\end{warning}

\subsection{Example:\ Tate rings}\label{subsec:tate-rings}
Tate rings are important enough that they get their own subsection. A \term{Tate ring} is a Huber ring $A$ in which there exist a {\em unit} $\varpi$ such that $\varpi^n \rightarrow 0$ as $n\rightarrow +\infty$. Such a $\varpi$ is called a \term{pseudo-uniformizer}, borrowing from Example \ref{example:nonarchimedean-field}.

Let us examine the structure of a Tate ring $A$. Choose a pair of definition $(A_0,I)$. Since $\varpi^n \rightarrow 0$ as $n\rightarrow +\infty$, there exists an $n\geq 1$ such that $\varpi^n \in I$. The element $\varpi^n$ is still a pseudo-uniformizer. So, without loss of generality $\varpi \in I \subseteq A_0$. Now we claim:
\begin{enumerate}[label=(\alph*)]
\item The topology on $A_0$ is the $\varpi A_0$-adic topology.\label{enum-part:tate-ring-varpi-adic-topology}
\item Algebraically, $A = A_0[\frac{1}{\varpi}]$.\label{enum-part:tate-ring-varpi-adic-algebra}
\end{enumerate}
Thus, any choice of a pseudo-uniformizer $\varpi$ (in a given ring of definition) simultaneously controls the topological and algebraic structure of a Tate ring. For proof, $\varpi A_0$ is an open $A_0$-ideal, since it is a multiplicative translate of the open subset $A_0 \subseteq A$. But also $\varpi \in I$ and so $\varpi A_0 \subseteq I$. This proves \ref{enum-part:tate-ring-varpi-adic-topology}. For \ref{enum-part:tate-ring-varpi-adic-algebra}, consider any $f \in A$. Since multiplication by $f$ is continuous on $A$, there exists $n$ such that $f\varpi^n A_0 \subseteq A_0$, and therefore $f \in A_0[\frac{1}{\varpi}]$.

\subsection{Rings of definition}

Huber rings are topological rings. Pairs of definition $(A_0,I)$ are auxiliary data. The goal of this section is illustrating the flexibility of this data.

Let $A$ be a topological ring. A subset $X \subseteq A$ is called \term{bounded} if for any open neighborhood $U\ni 0$, there exists an open neighborhood $V \ni 0$ such that $XV \subseteq U$. Note that $X$, $U$, and $V$ are {\em a priori} just subsets. The condition $XV \subseteq U$ means that if $x \in X$ and $v \in V$ then $xv \in U$. If $U$ is an additive subgroup, which is often the case in the context of Huber rings, then the condition $XV \subseteq U$ is equivalent to $X \cdot V \subseteq U$ where $X \cdot V$ is the abelian group generated by $XV$.\footnote{The typographical difference between $XV$ and $X \cdot V$ can cause problems while scanning. The ``$X\cdot V$''-notation is used by Huber \cite{Huber-ContinuousValuations}. It has stuck in many references. Beware!}

The next proposition classifies rings of definition from a topological perspective. It is repeatedly used in analyzing more refined structures in Section \ref{subsec:bounded-conditions}.

\begin{proposition}\label{proposition:ring-of-definition-vs-open-bounded}
If $A$ is a Huber ring, then a subring $A_0 \subseteq A$ is a ring of definition if and only if $A_0$ is open and bounded.
\end{proposition}
\begin{proof}
Let $A$ be a Huber ring and $A_0$ a ring of definition. First, $A_0$ is open by definition. Second, basic open neighborhoods of zero have the form $U = I^n$ with $I \subseteq A_0$ an ideal. Since $A_0 I^n \subseteq I^n$, we see that $A_0$ is bounded (taking $V = I^n$ as well).

We now argue the converse. Assume $A_0$ is open and bounded. Choose any pair of definition $(B_0,J)$ for $A$. Since $A_0$ is open, it contains a power $J^n$ of $J$. The pair $(B_0,J^n)$ is also a pair of definition for $A$. Replacing $J$ by $J^n$, we assume that $J \subseteq A_0$.

Now choose $f_1,\dotsc,f_d \in J$ that generate $J$ as a $B_0$-ideal. Define
\begin{equation*}
I = \sum_{i=1}^d A_0 f_i \subseteq A_0.
\end{equation*} 
We claim $(A_0,I)$ is a pair of definition for $A$. Since $I$ is a finitely generated $A_0$-ideal, we focus on showing the $I$-adic topology is the $J$-adic topology on $A$.

First, we show $J^2 \subseteq I$. Recall $J \subseteq A_0$.  Then,
\begin{equation*}
J^2 = J\cdot (\sum_{i=1}^d B_0 f_i) = \sum_{i=1}^d Jf_i \subseteq \sum_{i=1}^d A_0 f_i = I.
\end{equation*}
Second, we show $I^n \subseteq J$ for some $n$. This is where we use that $A_0$ is bounded. Indeed, since $A_0$ is bounded and $J$ is an ideal, we may choose $n$ such that $A_0 \cdot J^n \subseteq J$. The $A_0$-ideal $I^n$ is additively generated by elements $a f_1^{n_1} \dotsb f_d^{n_d}$ where $n_1+\dotsb+n_d = n$ and $a \in A_0$. Since $f_1^{n_1}\dotsb f_d^{n_d} \in J^n$, we  conclude that $I^n \subseteq A_0 \cdot J^n \subseteq J$, as claimed.
\end{proof}

\subsection{Bounded conditions}\label{subsec:bounded-conditions}
Let $A$ be a topological ring. This subsection introduces power-bounded elements $A^{\circ}$ and topologically nilpotent elements $A^{\circ\circ}$.

\subsubsection{Power-bounded elements}
We say $f \in A$ is \term{power-bounded} if its powers
\begin{equation*}
\{1,f,f^2,f^3,\dotsc\}
\end{equation*}
form a bounded subset in $A$. It is traditional to use $A^{\circ}$ as notation:
\begin{equation*}
A^{\circ} = \{f \in A \mid f \text{ is power-bounded}\}.
\end{equation*}
For instance, if $A = \Cp\lr w$, then $w \in A^{\circ}$, and in fact $A^{\circ} = \OCp\lr w$. (See Exercise \ref{exercise:power-bounded}.)

Note that $f \in A^{\circ}$ demands that for any open subset $U \ni 0$, there is an open subset $V\ni 0$ for which $V \subseteq U$, and $fV \subseteq U$, and $f^2V \subseteq U$, and so on. A finite number of these containments can be arranged, since multiplication by $f$ is continuous on $A$, but the condition that $f \in A^{\circ}$ is more strict. 

It is clear that $A^{\circ}$ is closed for multiplication. It follows from Exercise \ref{exercise:bounded-exercises} that $A^{\circ}$ is in fact a subring of $A$. If $A$ is a Huber ring, it even equals the union of all rings of definition. The main step in the proof is the next lemma.

\begin{lemma}\label{lemma:ring-generated-is-definition}
Let $A$ be a Huber ring. If $A_0$ is a ring of definition and $f \in A^{\circ}$, then $A_0[f]$ is a ring of definition.
\end{lemma}
\begin{proof}
Since $A$ is a Huber ring, we only need to show $A_0[f]$ is open and bounded, by Proposition \ref{proposition:ring-of-definition-vs-open-bounded}. Since $A_0 \subseteq A_0[f]$ already, the open-ness is clear. What about boundededness? Let $I \subseteq A_0$ be an ideal of definition. Since $f \in A^{\circ}$, there exists $n$ such that $f^m I^n \subseteq I$ for all $m\geq 0$. Since $I$ and $I^n$ are $A_0$-ideals, we see that $A_0[f] I^n \subseteq I$. Then, $A_0[f]$ is bounded by Exercise \ref{exercise:bounded-exercises}.
\end{proof}

\begin{proposition}\label{proposition:power-bounded-union-rings-of-definition}
If $A$ is a Huber ring, then
\begin{equation*}
A^{\circ} = \bigcup_{\substack{A_0 \subseteq A\\\textrm{\rm ring of def.}}} A_0.
\end{equation*} 
So, $A^\circ$ is an open subring in $A$. It is also integrally closed in $A$.
\end{proposition}
\begin{proof}
Let $A_0$ a ring of definition. Since $A_0$ is bounded and closed under exponentiation, we see $A_0 \subseteq A^{\circ}$. Conversely, if $f \in A^{\circ}$ and $A_0$ is any ring of definition, then $f \in A_0[f]$. By Lemma \ref{lemma:ring-generated-is-definition}, $A_0[f]$ is a ring of definition.

Having shown the equality in the proposition, we have that $A^{\circ}$ is open, and we already indicated why it is a subring. The primary observation to show $A^{\circ}$ is integrally closed is that the proof so far implies that if $f \in A$ is integral over $A^{\circ}$, then $f$ is integral over some ring of definition $A_0$. Given this, one checks that the finite $A_0$-algebra $A_0[f]$ is open and bounded. See Exercise \ref{exercise:power-bounded-integrally-closed}.
\end{proof}

Proposition \ref{proposition:power-bounded-union-rings-of-definition} {\em does not} claim $A^{\circ}$ is bounded. See Exercise \ref{exercise:power-bounded-not-rod}.

\subsubsection{Topologically nilpotent elements}

An element $f \in A$ is called \term{topologically nilpotent} if for all open neighborhoods $U \ni 0$, we have $f^n \in U$ for $n \gg 0$. The sufficiently large ``$\gg$'' depends on $f$ and $U$. As an example, a pseudo-uniformizer in a Tate ring is topologically nilpotent. The formal notation is
\begin{equation*}
A^{\circ\circ} = \{f \in A \mid f \text{ is topologically nilpotent}\}.
\end{equation*}
Assume $A$ is a Huber ring and $A_0$ is a ring of definition. If $f \in A^{\circ\circ}$ is topologically nilpotent, then the powers of $f$ are in the bounded union $A_0 \cup \{1,\dotsc,f^N\}$ for some $N$. Therefore, when $A$ is a Huber ring we have $A^{\circ\circ} \subseteq A^{\circ}$. The analogue of Proposition \ref{proposition:power-bounded-union-rings-of-definition} is the following result.

\begin{proposition}\label{proposition:topologically-nilpotent}
If $A$ is a Huber ring, then
\begin{equation*}
A^{\circ\circ} = \bigcup_{\substack{I \subseteq A\\\textrm{\rm ideal of def.}}} I.
\end{equation*}
Moreover, $A^{\circ\circ}$ is a radical $A^{\circ}$-ideal.
\end{proposition}

\begin{proof}
We show topologically nilpotent elements lie in ideals of definition. (The other containment is straightforward.) Suppose $f \in A^{\circ\circ}$. We just explained that $A^{\circ\circ} \subseteq A^{\circ}$, and so we may choose, by Proposition \ref{proposition:power-bounded-union-rings-of-definition}, a pair of definition $(A_0,J)$ such that $f \in A_0$. Since $J$ is open, there exists $n$ such that $f^n \in J$. Now define $I = J + A_0f \subseteq A_0$. Then, $I$ is a finitely generated $A_0$-ideal because $J$ is. We claim $(A_0,I)$ is a pair of definition. First, $I$ is open since it contains $J$. Second, $f^n \in J$ and $fJ \subseteq J$, since $f \in A_0$. Therefore, $I^n \subseteq J$. We have shown the $I$-adic and $J$-adic topologies coincide, finishing the claim.

As in the case of power-bounded elements, we leave the auxiliary claim that $A^{\circ\circ}$ is a radical $A^{\circ}$-ideal as Exercise \ref{exercise:top-nilp-ideal-in-Acirc}. (A hint is provided.)
\end{proof}

Note, the proof makes clear that the notation ``$I \subseteq A$ ideal of def.'' additionally indexes over all rings rings of definition $A_0$, not just the ideals of definition inside some fixed $A_0$.

\subsection{Rational localization}\label{subsection:rational-localization}

We end our introduction to Huber rings by constructing \term{rational localizations}. These localizations are to Huber rings and adic spaces what ring-theoretic localizations are to all rings and schemes. That is, they are used to construct affine subspaces of adic spaces. The reader is invited to later meditate on this analogy in the context of the proof of the adic Nullstellensatz (Theorem \ref{theorem:Aplus-recovery}).

Given a Huber ring $A$, a rational localization $A(\frac{g_1,\dotsc,g_r}{s})$ is another Huber ring depending on elements $g_1,\dotsc,g_r,s \in A$. Localizations appear at the start of  \cite{Huber-Generalization}, where adic spaces are defined. See also Section \ref{subsubsec:rational-localization-huber-pairs} and H\"ubner's lectures \cite{Hubner-AdicSpaces}. You cannot localize with respect to all possible choices of elements. The criterion is that the $A$-ideal $\fraka = As + Ag_1 + \dotsb + Ag_r \subseteq A$ is {\em open}. Before we explain, here are examples:
\begin{enumerate}[label=(\roman*)]
\item Suppose $(A_0,I)$ is a pair of definition and $f_1,\dotsc,f_d \in I$ are $A_0$-generators. Then, $\{g_1,\dotsc,g_r\} = \{f_1,\dotsc,f_d\}$ is valid with any $s$, since $I \subseteq \fraka$.
\item Let $A$ be a Tate ring. Then, the only open $A$-ideal is $\fraka = A$. Therefore, the condition on $g_1,\dotsc,g_r,s$ is that they generate $A$. See Exercise \ref{exercise:ideal-tate}.\label{enum-part:tate-ring-open-observation}
\item Suppose $A = \Zp\llbracket w \rrbracket   $ with the $\lr{p,w}$-adic topology. Then $\{g,s\} = \{w,p\}$ is a valid choice, but $\{g,s\} = \{p,p\}$ is not. See Exercise \ref{exercise:bad-localization}.
\item Localization on $\Cp\lr w$ is related to  $\Cp\lr{\frac{w}{p}}$ in Section  \ref{subsec:final-rational-localization}.
\end{enumerate}

The topological ring $A(\frac{g_1,\dotsc,g_r}{s})$ is defined in two steps, first algebraically and second topologically. We fix a pair of definition $(A_0,I)$. This choice is ultimately immaterial, by the universal property in Theorem \ref{theorem:rational-localization}. 
\begin{enumerate}[label=(RL-\arabic*)]
\item The underlying ring is $A(\frac{g_1,\dotsc,g_r}{s}) = A[\frac{1}{s}]$.\label{enum-part:RL-ring-def}
\item Let $A'=A(\frac{g_1,\dotsc,g_r}{s})$ and $A_0' = A_0[\frac{g_1}{s},\dotsc,\frac{g_r}{s}] \subseteq A'$. We make $A_0'$ a topological ring by giving it the $IA_0'$-adic topology. We then give $A'$ the unique topology where $A_0' \subseteq A'$ is open.
\label{enum-part:RL-top-def}
\end{enumerate}
By \ref{enum-part:RL-ring-def} and \ref{enum-part:RL-top-def}, we have a ring with a topology. But remember now Warning \ref{warning:topological-ring}! We must justify that in fact we have defined a topological ring.

\begin{lemma}\label{lemma:localization-topological}
Let $A$ be a Huber ring. Assume that $g_1,\dotsc,g_r,s \in A$ generate an open $A$-ideal. Then $A(\frac{g_1,\dotsc,g_r}{s})$ is a topological ring.
\end{lemma}

The proof of Lemma \ref{lemma:localization-topological} uses that $I$ is finitely generated over $A_0$, which we have not {\em really} used until now. The only result where we explicitly used the property was Proposition \ref{proposition:ring-of-definition-vs-open-bounded}. However, the issue there is preserving the finitely generated property while switching ideals of definition. The same result holds assuming only that $A$ satisfies Definition \ref{definition:huber-ring}\ref{enum-part:huber-defn-adic}. (We thank Kalyani Kansal for this observation.)
\begin{proof}[Proof of Lemma \ref{lemma:localization-topological}]
The first step is a general simplification. Then, we write out the argument only in the case that $A$ is a Tate ring. The main reason is to decrease the notations and to generate a proof that is simpler to recall.

For notation, define $A' = A(\frac{g_1,\dotsc,g_r}{s})$, $A_0' = A_0[\frac{g_1}{s},\dotsc,\frac{g_r}{s}]$ and $I' = IA_0'$. We give $A_0'$ the $I'$-adic topology and declare $A_0' \subseteq A'$ open. If $I$ is replaced by a different ideal of definition in $A_0$, the topology on $A'$ does not change.

The topology on $A'$ is built by declaring a topological ring $A_0' \subseteq A'$ to be open. By Exercise \ref{exercise:subring-Iadic}, we must only show multiplication by $f'$ is continuous on $A'$, for all $f' \in A'$. Since $I'$ is generated by $I$ over $A_0'$, the explicit claim is that for all $f' \in A'$, there exists an $n$ such that $f' I^n \subseteq A_0'$. If $f ' = f \in A$, this is clear since $A_0$ is a topological ring for the $I$-adic topology. A general element of $A'$ is $f' = f/s^N$ for some $N$. Therefore, only the case $f'=\frac{1}{s}$ is significant. Thus, we want $\frac{1}{s}I^n \subseteq A_0'$ for some $n$.

We now assume that $A$ is a Tate ring. Choose a pseudo-uniformizer $\varpi \in A$ that belongs to $A_0$. By Section \ref{subsec:tate-rings}, we know $A = A_0[\frac{1}{\varpi}]$ and we may assume $I = \varpi A_0$. As mentioned in \ref{enum-part:tate-ring-open-observation} prior to the lemma, since $A$ is a Tate ring, we are assuming that $g_1,\dotsc,g_r,s$ generate the {\em unit} ideal. So, $\varpi$ may be expressed as
\begin{equation}\label{eqn:varpi-equation}
\varpi = a_0s + a_1 g_1 + \dotsb + a_rg_r \;\;\;\; (a_i \in A).
\end{equation}
Since $A = A_0[\frac{1}{\varpi}]$, there is a positive integer $n$ such that $a_i \varpi^{n-1} \in A_0$ for all $i$. By \eqref{eqn:varpi-equation}, we then have that
\begin{equation}\label{eqn:A0-subopen}
\varpi^{n} \in A_0s + A_0g_1 + \dotsb + A_0g_r.
\end{equation}
And we are done now because
\begin{equation*}
\frac{1}{s}I^{n} = \frac{1}{s}\varpi^{n}A_0 \subseteq A_0[\frac{g_1}{s},\dotsc,\frac{g_r}{s}] = A_0'
\end{equation*}

In general, the open-ness of $As + Ag_1 + \dotsb + Ag_r$ leads to expressions similar to \eqref{eqn:varpi-equation}, with $\varpi$ is replaced by any one of a finite number of $A_0$-generators of an ideal of definition $I$. The conclusion, analogous to \eqref{eqn:A0-subopen}, is that $A_0s + A_0g_1 + \dotsb + A_0g_r$ is open. Fill in the details and finish the argument as Exercise \ref{exercise:local-topological-ring}. 
\end{proof}

\begin{theorem}[Rational localization]\label{theorem:rational-localization}
Let $A$ be a Huber ring and assume that $g_1,\dotsc,g_r,s \in A$ generate an open ideal of $A$. 
\begin{enumerate}[label=(\alph*)]
\item For $(A_0,I)$ a fixed pair of definition, $A(\frac{g_1,\dotsc,g_r}{s})$ is a Huber ring.\label{enum-part-theorem:RL-huber}
\item The natural map $A \rightarrow A(\frac{g_1,\dotsc,g_r}{s})$ is initial among continuous morphisms $A \rightarrow B$, for $B$ a Huber ring, for which the image of $s$ is invertible and the image of each $g_i/s$ is power-bounded.\label{enum-part-theorem:RL-univ-property}
\end{enumerate}
\end{theorem}
\begin{proof}
We proved in Lemma \ref{lemma:localization-topological} that  $A'=A(\frac{g_1,\dotsc,g_r}{s})$ is a topological ring. By construction it has a ring of definition $A_0'=A_0[\frac{g_1}{s},\dotsc,\frac{g_r}{s}]$ with finitely generated ideal of definition $IA_0'$. This proves that $A'$ is a Huber ring, possibly depending on the choice of $(A_0,I)$. That choice disappears once we prove the universal property \ref{enum-part-theorem:RL-univ-property}, since the property itself makes no reference to $(A_0,I)$.

For \ref{enum-part-theorem:RL-univ-property}, the  localization map $\iota: A \rightarrow A'=A[\frac{1}{s}]$ is continuous, since $I \subseteq \iota^{-1}(IA_0')$. Moreover, $s$ is a unit in $A'$ and $\frac{g_j}{s}$ is power-bounded, since it lies in the ring of definition $A_0'$ of $A'$ (recall Proposition \ref{proposition:power-bounded-union-rings-of-definition}). Suppose $\varphi: A \rightarrow B$ is given as in \ref{enum-part-theorem:RL-univ-property}. Since $\varphi(s)$ is invertible in $B$, there is a natural factorization 
\begin{equation*}
\xymatrix{
A \ar[r]^-{\varphi} \ar[d]_{\iota} & B\\
A' \ar@{.>}[ur]_-{\psi}
}
\end{equation*}
 at the level of $A$-algebras. We must check $\psi$ is continuous. Fix an open neighborhood $U \ni 0$ in $B$, which we assume is an additive subgroup. Since $\frac{\varphi(g_i)}{\varphi(s)}$ is power-bounded in $B$ for all $i$, there exists an open neighborhood $V \ni 0$ for which
\begin{equation*}
\psi(\frac{g_i}{s})^mV = (\frac{\varphi(g_i)}{\varphi(s)})^m V \subseteq U
\end{equation*} 
for all $m \geq 0$ and all $i$. Since $\varphi$ is continuous, $I^n \subseteq \varphi^{-1}(V)$ for some $n$. Since $(I')^n = I^n A_0'$ is spanned by elements of the form $f(\frac{g_i}{s})^m \in A'$ for $f \in I^n$ and $U$ is an additive subgroup, it follows that $(I')^n \subseteq \psi^{-1}(U)$.
\end{proof}

We make two complementary remarks. First, the proof clarifies that the universal property is valid as long as $B$  possesses a neighborhood basis of zero consisting of additive subgroups ($B$ is a ``non-Archimedean ring''). Second, some authors assume in Theorem \ref{theorem:rational-localization} that $g_1,\dotsc,g_r$ generate an open ideal. There is practically no difference, since the definitions \ref{enum-part:RL-ring-def} and \ref{enum-part:RL-top-def} make $A(\frac{g_1,\dotsc,g_r}{s}) = A(\frac{g_1,\dotsc,g_r,s}{s})$.

\subsection{Example:\ Localizing $\Cp\lr w$}\label{subsec:final-rational-localization}
The final section of this lecture connects rational localization on $\Cp\lr w$ to the $\Cp$-affinoid algebra $\Cp\lr{\frac{w}{p}}$. The discussion is simplified by focusing first on $\Cp[w]$.

Consider $A = \Cp[w]$ as a topological ring with the topology induced from the Gauss norm. That is, if $f = a_0 + a_1w + \dotsb + a_nw^n$, then
\begin{equation*}
\norm{f}_{\Gauss} = \max_{i\geq 0} \norm{a_i}_p.
\end{equation*}
Just like $\Cp\lr w$, we see $A$ is a Tate ring with pseudo-uniformizer $p \in A$. A ring of definition is $A_0 = \OCp[w]$ and an ideal of definition is $p\OCp[w]$.

Now localize with $g = w$ and $s = p$. The hypothesis of Theorem \ref{theorem:rational-localization} is satisfied because $s$ is a unit in $A$. As a ring,
\begin{equation*}
A(\frac{w}{p}) = \Cp[w][\frac{1}{p}] = \Cp[w].
\end{equation*}
So, $A = A(\frac{w}{p})$, still. But the topology is new! The original topology has a basis around zero given by $p^n\OCp[w]$. The new topology has a basis around zero given by
\begin{equation}\label{eqn:open-balls-localization}
p^n A_0[\frac{w}{p}] = p^n\OCp[w,\frac{w}{p}] = p^n\OCp[\frac{w}{p}].
\end{equation}
Here is a concrete difference. In the Gauss norm topology,  $w$ {\em is} power bounded but {\em not} topologically nilpotent since $w^n \not\in p\OCp[w]$ for any $n$. Yet, in $A(\frac{w}{p})$ we have
\begin{equation*}
w^n = p^n (\frac{w}{p})^n \in p^nA_0[\frac{w}{p}]
\end{equation*}
Therefore, $w$ {\em is} topologically nilpotent in $A(\frac{w}{p})$.

We chose $A$ to be the polynomial ring so that \eqref{eqn:open-balls-localization} was most clear. The connection to affinoid algebras is via completion. The main point is that $\Cp\lr w$ is the completion of $\Cp[w]$ for the Gauss norm. As Exercise \ref{exercise:norm-localization-check}, the reader can check that the topology on $A(\frac{w}{p})=\Cp[w]$ is the {\em same} as the topology induced by the norm
\begin{equation*}
\val{f}_{\frac{1}{p}} := \max_{i\geq 0} \norm{a_i}_p p^{-i},
\end{equation*}
and $\Cp\lr{\frac{w}{p}}$ is the completion of $\Cp[w]$ for $\bval_{\frac{1}{p}}$. (The completions here are all with respect to norms. We revisit more general completions of Huber rings in Section \ref{subsubsec:completion-huber-pairs}.)

\subsection*{Section \thesection~ Exercises}

\begin{exercise}\label{exercise:topological-ring-check}
Prove that the rings with topology described in Examples \ref{example:ring}-\ref{example:one-variable-tate-algebra} are all topological rings.
\end{exercise}

\begin{exercise}\label{exercise:subring-Iadic}
Let $A$ be a ring and $A_0 \subseteq A$ a subring and $I$ an $A_0$-ideal. Consider $A_0$ as a topological ring with the $I$-adic topology as in Example \ref{example:ring}. Define a topology on $A$ by declaring $A_0 \subseteq A$ to be open. Show that the following are equivalent:
\begin{enumerate}[label=(\roman*)]
\item $A$ is a topological ring.
\item For all $f \in A$, multiplication by $f$ is continuous on $A$.
\item For all $f \in A$, there exists an $n \gg 0$ such that $fI^n \subseteq A_0$.
\end{enumerate}
\end{exercise}

\begin{solution}
See \cite[pg. 274]{Bourbaki-GeneralTopology-Ch1-4}.
\end{solution}

\begin{exercise}\label{exercise:bad-topological-rings}
These examples have the same flavor (with similar solutions). In each, we give a ring $A$ containing a topological ring $A_0$. The exercise is to check that declaring $A_0 \subseteq A$ to be open will {\em not} make $A$ into a topological ring.
\begin{enumerate}[label=(\alph*)]
\item Let $A_0=\OCp$ with the $\frakm$-adic topology. Show that declaring $\OCp \subseteq \Cp$ to be open will not make $\Cp$ a topological field.
\item Consider $\Zp\llbracket w \rrbracket   $ with the $\langle p,w\rangle$-adic topology. Show that if one declares $\Zp\llbracket w \rrbracket    \subseteq \Zp\llbracket w \rrbracket   [\frac{1}{p}]$ is open, then $\Zp\llbracket w \rrbracket   [\frac{1}{p}]$ will not be a topological ring.
\end{enumerate}
\end{exercise}

\begin{solution}
For each question we use Exercise \ref{exercise:subring-Iadic}.
\begin{enumerate}[label=(\alph*)]
\item Consider multiplication by $\frac{1}{3} : \Q \rightarrow \Q$. Then $\frac{1}{3}2^n \Z \not\subseteq \Z$ for any $n \geq 0$.
\item Consider multiplication by $\frac{1}{p} : \Cp \rightarrow \Cp$. It is clear that $\frac{1}{p} \frakm_{\OCp} \not\subseteq \OCp$ since $\sqrt{p} \in \frakm_{\OCp}$ but $\frac{1}{\sqrt{p}} \not\in \OCp$. Since $\frakm_{\OCp}^n = \frakm_{\OCp}$ for all $n > 0$, we see multiplication by $\frac{1}{p}$ will not be continuous.
\item Let $\frakm = \langle p,w\rangle$ be the maximal ideal of $A_0$. Then, for each $n$ we have $w^n \in \frakm^n$ but $\frac{1}{p}w^n \not\in A_0$. Therefore multiplication by $\frac{1}{p}$ will not be continuous on $A = \Zp\llbracket w \rrbracket   [\frac{1}{p}]$. 
\end{enumerate}
\end{solution}

\begin{exercise}\label{exercise:power-bounded}
Let $A = \Cp\lr w$.

\begin{enumerate}[label=(\alph*)]
\item Show that $w$ is power-bounded but not topologically nilpotent.
\item Show that, in fact, $A^{\circ} = \OCp\lr w$, and
\begin{equation*}
A^{\circ\circ} = \frakm_{\OCp} \lr w = \left\{\sum_{i=0}^{\infty} a_i w^i \in A \mid a_i \in \frakm_{\OCp} \text{ for all $i$}\right\}.
\end{equation*}
\end{enumerate}
\end{exercise}

\begin{exercise}\label{exercise:bounded-exercises}
Let $A$ be a topological ring and $X$ and $Y$ subsets of $A$.
\begin{enumerate}[label=(\alph*)]
\item Show that if $X \subseteq Y$ and $Y$ is bounded, then $X$ is bounded.\label{exercise-part:containment}
\item Show that if $A$ is a Huber ring and $X$ and $Y$ are bounded, then so is $X \cdot Y$.\label{exercise-part:subgroup-generated-bounded}
\item Suppose $A$ is a Huber ring. Show that $X\subseteq A$ is bounded if and only if for any ideal of definition $I$, there exists an $n \geq 1$ such that $XI^n \subseteq I$.\label{exercise-part:huber-arbitrary}
\end{enumerate}
\end{exercise}

\begin{solution}
\begin{enumerate}[label=(\alph*)]
\item Given $U$ if $V$ witnesses $YV \subseteq U$ then $XV \subseteq U$ also.
\item Let $U$ be an open {\em subgroup}. Since $Y$ is bounded, there exists an open subgroup $V$ such that $xv \in U$ for all $v \in V$ and $x \in X$. Similarly, there exists $W$ an open subgroup sch that $yw \in V$ for all $y \in Y$ and $w \in W$. Thus if $\sum x_i y_i \in X\cdot Y$ then
\begin{equation*}
(\sum x_i y_i)W \subseteq \sum x_i V \subseteq U,
\end{equation*}
since these are all subgroups.
\item If $X$ is bounded, then for any pair of definition $(A_0,I)$ there exists an $n$ such that $XI^n \subseteq I$. Conversely, suppose that we have the pair $(A_0,I)$. It suffices to show that $XV \subseteq J$ where $J$ is any ideal of definition. Choose $d \gg 0$ such that $I^d \subseteq J$. Then,
\begin{equation*}
X I^{nd} = X I^n I^n \dotsb I^n \subseteq I^d \subseteq J.
\end{equation*}
Thus we can take $V = I^{nd}$.
\end{enumerate}
\end{solution}

\begin{exercise}\label{exercise:rings-of-definition-filtered}
Show that if $A$ is a Huber ring and $A_0$ and $A_0'$ are rings of definition, then there exists a ring of definition $A_0''$ containing both.
\end{exercise}

\begin{solution}
Let $A_0''$ be
\begin{equation*}
A_0'' = A_0 \cdot A_0' = \{a_1a_1' + \dotsb + a_m a_m' \mid a_j \in A_0 \text{ and } a_j' \in A_0'\}.
\end{equation*}
This is a subring by the distributive law. It is open since it contains $A_0$, which is open. Each of $A_0$ and $A_0'$ are bounded, and so $A_0''$ is bounded by Exercise \ref{exercise:bounded-exercises}\ref{exercise-part:subgroup-generated-bounded}. Finally, we have shown $A_0''$ is open and bounded and therefore a ring of definition by Proposition \ref{proposition:ring-of-definition-vs-open-bounded}.
\end{solution}

\begin{exercise}\label{exercise:power-bounded-integrally-closed}
Show that if $A$ is a Huber ring, then $A^{\circ}$ is integrally closed in $A$.
\end{exercise}
\begin{hint}
A sketch is given at the end of the proof of Proposition \ref{proposition:power-bounded-union-rings-of-definition}.
\end{hint}

\begin{solution}
Let $f \in A$ integral over $A^{\circ}$. So, there is some relation 
\begin{equation*}
f^n + a_{n-1}f^{n-1} + \dotsb + a_0 = 0
\end{equation*}
with $a_j \in A^{\circ}$. By Proposition \ref{proposition:power-bounded-union-rings-of-definition} and Exercise \ref{exercise:rings-of-definition-filtered}, there is a single $A_0$ such that $a_j \in A_0$ for all $j$. Therefore $f$ is integral over $A_0$. We claim $A_0[f]$ is a ring of definition. It is clear that $A_0[f]$ is open, since $A_0 \subseteq A_0[f]$. To show that it is bounded, we note that $A_0[f] = A_0 + A_0 f + \dotsb + A_0 f^{n-1}$, by the integral relation. Since each $A_0f^j$ is bounded, we see $A_0[f]$ is bounded. Therefore $A_0[f]$ is a ring of definition by Proposition \ref{proposition:ring-of-definition-vs-open-bounded}. We conclude $f \in A^{\circ}$ by Proposition \ref{proposition:power-bounded-union-rings-of-definition}.
\end{solution}

\begin{exercise}\label{exercise:top-nilp-ideal-in-Acirc}
Let $A$ be a Huber ring. 
\begin{enumerate}[label=(\alph*)]
\item Show that $f \in A$ is topologically nilpotent if and only if there exists an ideal of definition $I \subseteq A$ and an integer $n \geq 1$ such that $f^n \in I$.\label{exercise-part:top-nilp-ideal-in-Acirc-old-lemma}
\item Show that $A^{\circ\circ} \subseteq A^{\circ}$ is a radical $A^{\circ}$-ideal.\label{exercise-part:top-nilp-ideal-in-Acirc-radical-ideal}
\end{enumerate}
\end{exercise}

\begin{hint}
In \ref{exercise-part:top-nilp-ideal-in-Acirc-old-lemma}, there exists $n$ such that $f^n \in I$ and $m$ such that $f^j I^m \subseteq I$ for all $j\leq n-1$. Then $f^N \in I$ for $N \geq n(m+1)$. For \ref{exercise-part:top-nilp-ideal-in-Acirc-radical-ideal}, the tricky part of the ideal property is showing $A^{\circ\circ}$ is closed for addition. For that, it may be helpful to note that if $f,g$ are topologically nilpotent, then they lie in a common ring of definition $A_0$ (Exercise \ref{exercise:rings-of-definition-filtered}).
\end{hint}

\begin{solution}
\begin{enumerate}[label=(\alph*)]
\item The step with content is showing that if $f^n \in I$, then $f$ is topologically nilpotent. We learned this trick from the proof of \cite[Proposition 8.3.5]{Conrad-PerfectoidSeminarNotes}.

Suppose $f \in A$ and $I$ is an ideal of definition and $n \geq 1$ and $f^n \in I$. We need to show $f$ is topologically nilpotent we just need to show $f^N \in I$ for all $N \geq n_0$ for some $n_0$.  Since $A$ is a topological ring, we may choose an integer $m\geq 1$ such that $f^j I^m \subseteq I$ for $j=0,1,\dotsc,n-1$. Then, for $N \geq n(m+1)$ we write
\begin{equation*}
N = nq + j \;\;\;\;\; (m < q,\; 0 \leq j < n).
\end{equation*}
Then
\begin{equation*}
f^N = f^{nq+j} = f^j f^{nm} f^{n(q-m)}. 
\end{equation*}
Since $q-m > 0$ we have $f^{n(q-m)} \in I$ and certainly $f^{nm} \in I^m$. Thus, we see that
\begin{equation*}
f^N \in f^j I^m I \subseteq I.
\end{equation*}
Therefore, we have shown $f^N \in I$ for $N \geq n_0$ where $n_0 = n(m+1)$. The proof is complete.
\item First suppose that $f$ is topologically nilpotent and $g$ is power-bounded. Given $U$, there exists a $V$ such that $g^m V \subseteq U$ for all $m\geq 1$. Since $f$ is topologically nilpotent we see that $f^m \in V$ for $m \gg 0$. Therefore, for $m \gg 0$ we have $(fg)^m \in U$. Since $U$ is arbitrary, we have $fg \in A^{\circ\circ}$.

Now suppose $f,g$ are topologically nilpotent. By Exercise \ref{exercise:rings-of-definition-filtered} we can assume $f,g \in A_0$ for a fixed ring of definition. Let $I$ be an ideal of definition for $A_0$. Since $f,g$ are topologically nilpotent there exists some $m$ such that $f^m$ and $g^m$ lie in $I$. But then
\begin{equation*}
(f+g)^{2m} \in A_0 f^m + A_0 g^m \in I
\end{equation*}
as well. Therefore, by Lemma \ref{lemma:topologically-nilpotent-short-check} we have that $f+g$ is topologically nilpotent.
\end{enumerate}
\end{solution}

\begin{exercise}\label{exercise:power-bounded-not-rod}
Let $A = \Qp[\varepsilon] = \Qp \oplus \Qp\varepsilon$ with $\varepsilon^2 = 0$ and the $p$-adic topology induced on each factor. Show that $A$ is a Huber ring and
\begin{equation*}
A^{\circ} = \Zp \oplus \Qp\varepsilon.
\end{equation*}
Conclude that $A^{\circ}$ need not be a ring of definition.
\end{exercise}

\begin{solution}
The key point is that if $a + b \varepsilon \in A$ then $(a+b)^n = a^n + na^{n-1}b \varepsilon$.
\end{solution}

\begin{exercise}\label{exercise:ideal-tate}
Let $A$ be a Tate ring.
\begin{enumerate}[label=(\alph*)]
\item Show that any ideal of definition for $A$ contains a pseudo-uniformizer.\label{exercise-part:ideal-tate-every-ideal}
\item Show that if $\fraka$ is an {\em open} ideal of $A$, then $\fraka = A$. \label{exercise-part:ideal-tate-open-ideal}
\end{enumerate}
\end{exercise}

\begin{solution}
Part (a). The way I defined it, a pseudo-uniformizer $\varpi$ is a unit for which $\varpi \in J$ for some ideal of definition $J$. Given any $I$ there is some power $J^d$ such that $J^d \subseteq I$. Then $\varpi^d \in I$ and $\varpi^d$ is also a unit. Therefore $I$ contains a pseudo-uniformizer.

Part (b). If $\fraka$ is open then $\fraka \supseteq I$ for some ideal of definition $I$. By part (a), there exists a pseudo-uniformizer in $I$, which is then a unit in $A$. Therefore $\fraka$ contains a unit.
\end{solution}

\begin{exercise}\label{exercise:bad-localization}
Let $A = \Zp\llbracket w \rrbracket   $ with the $\frakm$-adic topology where $\frakm = \langle w,p\rangle$. This is a Huber ring as in Example \ref{example:ring}. 
\begin{enumerate}[label=(\alph*)]
\item Show that $g_1 = p$ and $s = p$ is invalid for the hypotheses in Theorem \ref{theorem:rational-localization}.
\item Try to define $A(\frac{p}{p})$ as in \ref{enum-part:RL-ring-def} and \ref{enum-part:RL-top-def}. Confirm that you {\em do not} get a topological ring.
\item Re-affirm directly that your objection disappears for the ring $A(\frac{w}{p})$.
\end{enumerate}
\end{exercise}
\begin{solution}
Part (a). The element $f_1$ does not generate an open ideal. Indeed, $\lr{w,p}^n \subseteq p\Zp\llbracket w \rrbracket   $ for any $n$.

Part (b). The construction gives you $A' = A(\frac{p}{p}) = A[\frac{1}{p}] = \Zp\llbracket w \rrbracket   [\frac{1}{p}]$ with neighborhoods of zero given by $\lr{w,p}^n\Zp\llbracket w \rrbracket   $ Then, multiplication by $\frac{1}{p}$ on $A'$ fails to be continuous because $\frac{1}{p}\lr{w,p}^n\Zp\llbracket w \rrbracket    \not\subseteq \Zp\llbracket w \rrbracket   $ for any $n$.

Part (c). The objection disappears because $A'$ is the same ring but the open sets becomes $\lr{w,p}^n\Zp\llbracket w \rrbracket   [\frac{w}{p}]$ and now $\frac{1}{p}\lr{w,p}\Zp\llbracket w \rrbracket    \subseteq \Zp\llbracket w \rrbracket   [\frac{w}{p}]$.
\end{solution}

\begin{exercise}\label{exercise:local-topological-ring}
Let $A$ be a Huber ring. Suppose $g_1,\dots,g_r,s \in A$ generate an open ideal. Show that $A(\frac{g_1,\dotsc,g_r}{s})$ is a topological ring, as promised in Lemma \ref{lemma:localization-topological}.
\end{exercise}
\begin{hint}
A hint is given at the end of the proof of Lemma \ref{lemma:localization-topological} in the text.
\end{hint}

\begin{solution}
Let $A$ be a Huber ring. We assume $\fraka = As + Ag_1 + \dotsb +  Ag_r$ is open in $A$. The generalization of the prior argument is that \eqref{eqn:A0-subopen} becomes the statement that actually the $A_0$-submodule $A_0s + A_0g_1 + \dotsb + A_0g_r$ is also open in $A$. To prove this, we simplify notation and replace $I$ by some fixed power in order to assume $I \subseteq \fraka$ (generalizing \eqref{eqn:varpi-equation} above). Fix a finite list $f_1,\dotsc,f_d$ of $A_0$-generators of $I$. Since $I \subseteq \fraka$ we can find $a_{ik} \in A$ such that
\begin{equation}\label{eqn:gk-eqn}
f_k = a_{0k}s + a_{1k}g_1 + \dotsb + a_{rk}g_r
\end{equation}
for all $1 \leq k \leq d$. There are only a finite number of $a_{ik}$ and so there exists a single power $m\geq 0$ such that $a_{ik}I^m \in A_0$ for all $i,k$. This shows $I^{m+1}$ is contained in $A_0s + A_0g_1 + \dotsb + A_0 g_r$. Now we have from \eqref{eqn:gk-eqn} that
\begin{equation*}
\frac{1}{s}I^{m+1} \subseteq A_0[\frac{g_1}{s},\dotsc,\frac{g_r}{s}].
\end{equation*}
By \eqref{eqn:key-1/s-containment}, we have shown $A'$ is a topological ring.
\end{solution}

\begin{exercise}\label{exercise:norm-localization-check}
Show that the topology on $\Qp[w](\frac{w}{p})$ in Section \ref{subsec:final-rational-localization} is the topology induced on $\Qp[w]$ given by the norm
\begin{equation*}
\val{a_0+a_1w + \dotsb + a_dw^d} = \max_i \norm{a_i}_p p^{-i}.
\end{equation*}
\end{exercise}
\begin{solution}
The topology is given by neighborhoods by $J^n = p^n \Zp[\frac{w}{p}]$. If we take $f = a_0 + a_1w + \dotsb + a_d w^d$ then we can re-write this as
\begin{equation*}
f = b_0 + b_1\frac{w}{p} + \dotsb + b_d(\frac{w}{p})^d
\end{equation*}
where $b_i = p^i a_i$. Therefore 
\begin{equation*}
f \in J^n \iff b_i \in p^n\Zp \iff \norm{a_i}_p p^{-i} \leq p^{-n} \;\;\;\; (\text{for all $i$}).
\end{equation*}
\end{solution}

\section{Valuation theory}\label{section:valuation}

The second lecture introduces valuation theory. The primary goal is discussing continuous valuations on topological rings. One highlight is a simple criterion (Proposition \ref{proposition:contiunity-criteria}) for a valuation on a Huber ring to be continuous. The final discussion will focus on the continuous valuation spectrum for the Tate algebra $\Cp\lr w$.

\subsection{Definition}

The symbol $\Gamma$ refers to a \term{totally ordered abelian group}. We write the group operation multiplicatively. By definition, $\Gamma$ is an abelian group with an order relation $\leq$ satisfying the following axioms.
\begin{enumerate}[label=(\alph*)]
\item For all $\gamma_1,\gamma_2 \in \Gamma$, either $\gamma_1 \leq \gamma_2$  or $\gamma_2 \leq \gamma_1$, and both occur if and only if $\gamma_1 = \gamma_2$ (``totally'' ordered).
\item If $\gamma \in \Gamma$, then $\gamma' \mapsto \gamma \gamma'$ preserves $\leq$.
\end{enumerate}
From the first axiom, it makes sense to write $\gamma_1 < \gamma_2$ provided $\gamma_1 \leq \gamma_2$ and $\gamma_1 \neq \gamma_2$. We always extend $\Gamma$ to the totally ordered monoid $\Gamma \cup \{0\}$. The multiplicative structure is given by $0 \cdot \gamma = 0 = \gamma \cdot 0$, for all $\gamma \in \Gamma$. The order extends by $0 < \gamma$ for all $\gamma \in \Gamma$.

\begin{definition}\label{definition:valuation}
Let $A$ be a ring. A \term{valuation} on $A$ is a function $\bval : A \rightarrow \Gamma\cup \{0\}$ such that $\val{0} = 0$ and $\val{1} = 1$, and for all $f,g \in A$ we have
\begin{align*}
\val{fg} &= \val{f}\val{g}  & \text{($\bval$ is multiplicative);}\\
\val{f+g} &\leq \max\{\val{f},\val{g}\}  & \text{(the ultrametric triangle inequality).}
\end{align*}
\end{definition}
Note, we are using valuation as a term generalizing an ultrametric norm, as opposed to something like the $p$-adic valuation $v_p : \Qp \rightarrow \Z \cup \{\infty\}$. This is a common choice of language in this research area. In If we were to restrict to valuations valued in $\Rplus$, we might prefer the terminology semi-norm. (``Semi-'' because we allow for $|f| = 0$ even if $f \neq 0$.) The terminology ``valuation'' seems to be preferred because $\Gamma$ may in fact not embed into $\Rplus$. Finally, confusion may arise since we often consider a ring $A$ topologized by a norm. To avoid confusion, we reserve $\bval$ for a valuation on $A$, while $\bnorm$ will denote a fixed norm on $A$.

\subsection{Examples}\label{subsection:valuation-examples}

\begin{example}
Let $A$ be an integral domain. The \term{trivial valuation} $\bval_{\triv}$ is 
\begin{equation*}
\val{a}_{\triv} = \begin{cases}
1 & \text{if $a \neq 0$};\\
0 & \text{if $a = 0$}.
\end{cases}
\end{equation*}
Therefore, if $A$ is any ring and $\frakp \subseteq A$ is a prime ideal, we have the trivial valuation $A \twoheadrightarrow A/\frakp \xrightarrow{\bval_{\triv}} \{0,1\}$ modulo $\frakp$.
\end{example}

\begin{example}
Let $A = \Q$. Some valuations on $A$ are given by:
\begin{itemize}
\item The trivial valuation $\bval=\bval_{\triv}$.
\item The valuation $\bval = \bnorm_{\infty}$ given by the Archimedean norm $\Q \xrightarrow{\bnorm_\infty} \R_{\geq 0}$.
\item The valuation $\bval = \bnorm_\ell$ given by an $\ell$-adic norm $\Q \xrightarrow{\bnorm_\ell} \R_{\geq 0}$ for a prime $\ell$.
\end{itemize}
Ostrowski's theorem states these are the only valuations on $\Q$, up to equivalence. (See Section \ref{subsec:valuation-spectra} below for ``equivalence'').
\end{example}

\begin{example}\label{example:valations-on-p-adics}
Let $A = \Qp$. We already know about $\bval_{\triv}$ and  $\bnorm_p$. But, there are more. A theorem sometimes attributed to Chevalley states that if $L/K$ is a field extension, then a valuation on $K$ extends to a valuation on $L$. See \cite[Chapter VI, \S 3, no. 3, Proposition 5]{Bourbaki-CommutativeAlgebra}, for instance. So, each valuation on $\Q$ extends to a valuation on $\Qp$ (in many ways). Note, this algebraic phenomenon requires enlarging the target group $\Gamma$. On $\Qp$, the valuations $\bval = \bval_{\triv}$ and $\bval = \bnorm_p$ are distinguished (up to equivalence) as the ones that take value in a cyclic group. The $p$-adic norm is distinguished as the only one that is continuous for the $p$-adic topology on $\Qp$. (Compare with Exercise \ref{exercise:Qp-valuation-check}, after reading a bit further.)
\end{example}

\begin{example}\label{example:tate-algebra}
Let $A = \Cp\lr w$ be the one-variable Tate algebra over $\Cp$. For $\alpha \in \OCp$, we have a valuation
\begin{equation*}
f \overset{\bval_\alpha}{\longmapsto} \norm{f(\alpha)}_p 
\end{equation*}
on $A$. These valuations were considered in our motivation (page \pageref{eqn:motivation-valuation-alpha}). There is also the valuation $\bval_1 := \bnorm_{\Gauss}$ that defines the topology on $A$. It is given by
\begin{equation*}
f = a_0 + a_1w + a_2w^2 + \dotsb \rightsquigarrow  \val{f}_{1} = \max_{i\geq 0} \norm{a_i}_p.
\end{equation*}
We switch notation to $\bval_1$, so that we can easily change ``$1$'' to another real number $r$ such that  $0 < r \leq 1$. That is, we define
\begin{equation*}
\val{f}_r = \max_{i\geq 0}  \norm{a_i}_pr^i.
\end{equation*}
Each $\bval_\alpha$ is semi-norm, while each $\bval_r$ is a norm. A more exotic example is given in Section \ref{subsec:exotic-Gamma} below.

\subsection{Continuous valuations}\label{subsection:continuous-valuations}
In the examples above, $\Qp$ and $\Cp\lr w$ are topological rings. Huber rings are also topological.  So, let us explain continuity for valuations. Assume $A \xrightarrow{\bval} \Gamma \cup \{0\}$ is a valuation. The \term{value group} $\Gamma_{\bval}$  is the subgroup of $\Gamma$ generated by the {\em non-zero} $\val{f}$ for $f \in A$. If $A=K$ is a field, then $\Gamma_{\bval} = |K^\times|$. In general, it is the {\em smallest}  subgroup of $\Gamma$ through which $\bval$ factors.

Now suppose $A$ is a topological ring. We say that $\bval$ is a \term{continuous valuation} on $A$ if for all $\gamma \in \Gamma_{\bval}$, the subset
\begin{equation}\label{eqn:Ugamma}
U_\gamma = \{f \in A \mid \val f < \gamma \} \subseteq A
\end{equation}
is open in $A$. (The $U_\gamma$ are open {\em subgroups} even.) In the examples above:
\begin{enumerate}[label=(\roman*)]
\item Let $A$ be a topological ring and $\frakp$ a prime ideal. The trivial valuation modulo $\frakp$ is continuous if and only if $\frakp \subseteq A$ is open.
\item The $p$-adic norm is continuous on $\Qp$. The trivial norm is not.
\item The valuations $\bval_{\alpha}$ and $\bval_r$ on $\Cp\lr w$ are continuous.
\end{enumerate}
What might a reasonable discontinuous valuation look like? Consider $\Cp[w] \subseteq \Cp\lr w$ with the  topology induced from the Gauss norm. Then,
\begin{equation*}
\val{a_0+a_1w + \dotsb + a_nw^n}_r = \max_{i=0,\dotsc,n} \norm{a_i}_pr^i
\end{equation*}
is a valuation on $\Cp[w]$ for all $0 \leq r < \infty$. However, it is continuous if and only if $r \leq 1$. See Exercise \ref{exercise:r-norm-polynomials}.

\subsection{An exotic example}\label{subsec:exotic-Gamma}
As promised in Example \ref{example:tate-algebra}, we give an extended example of a valuation on $\Cp\lr w$ that is not a (semi-)norm. Define  $\Gamma = \R_{> 0} \times \R_{>0}$ with the ``read left to right'' ordering. The technical term is \term{lexicographic}.  In symbols,
\begin{equation*}
a < c \implies (a,b) < (c,d) \;\;\;\;\; \text{(any $b,d$)},
\end{equation*}
and
\begin{equation*}
e < b \implies (a,e) < (a,b).
\end{equation*}
The right-hand portion of Figure \ref{fig:exotic-valuation-order} illustrates the order relation on $\Gamma$.

\begin{figure}[htp]
\centering
\includegraphics{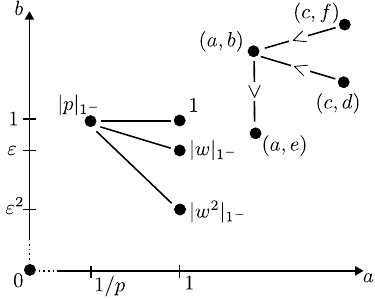}
\caption{A visualization of $\Gamma \cup \{0\}$ for $\Gamma = \R_{> 0} \times \R_{>0}$. The right-hand portion of the figure illustrates the generic order relation. The left-hand portion illustrates the position of $p$ versus $w$ under the valuation $|-|_{1^-}$.}
\label{fig:exotic-valuation-order}
\end{figure}

Now pick a real number $0 < \varepsilon < 1$.  We define $\bval_{1^-}$ on $\Cp\lr w$ by
\begin{equation*}
\val{a_0 + a_1w + a_2w^2 + \dotsb}_{1^-} = \max_{i \geq 0} (\norm{a_i}_p, \varepsilon^i) \in \Gamma.
\end{equation*}
What does $\bval_{1^-}$ measure? Write $\val{f}_{1^-} = (a,b)$. Then, $a \geq \norm{a_i}_p$ for all $i$ by definition of the lexicographic order. Since $a = \norm{a_i}_p$ for some $i$ as well, we see 
\begin{equation}\label{eqn:1-minus-factorize}
a = \max_{i \geq 0} \norm{a_i}_p = \val{f}_{1}.
\end{equation}
So, the first coordinate of $\val{f}_{1^-}$ is the Gauss norm of $f$. What about the second? Suppose that $i_1 < i_2 < \dotsb < i_n$ are the indices where $\norm{a_{i_j}}_p = \val{f}_{1}$. Then, in $\Gamma$, we have
\begin{equation*}
(\norm{a_{i_n}}_p,\varepsilon^{i_n}) < \dotsb < (\norm{a_{i_2}}_p,\varepsilon^{i_2}) < (\norm{a_{i_1}}_p,\varepsilon^{i_1}) = \val{f}_{1^-}.
\end{equation*}
Thus, $\val{f}_{1^-}$ records, in the second coordinate, the index of least degree whose coefficient realizes the Gauss norm. We propose Exercise \ref{exercise:1minus-motivation} to interpret $\bval_{1^-}$ as combining the Gauss norm with the order of vanishing of polynomials modulo $p$.

The calculation shows $\bval_{1^-}$ is continuous on $\Cp\lr w$. Indeed, if $a,b > 0$ are real numbers, we have shown that within $\Cp\lr w$ we have
\begin{equation*}
U_{a}^{\bval_1} \subseteq U_{(a,b)}^{\bval_{1^-}}.
\end{equation*}
(The superscripts indicate the valuation to which the ``$U_\gamma$''-notation is being applied.) Each $U_a^{\bval_1}$ is open, since $\bval_1$ defines the topology on $\Cp\lr w$. Thus each $U_{(a,b)}^{\bval_{1^-}}$ is open as well.

Finally, $\bval_{1^-}$ is written with a ``$1^-$'' to promote intuition that the coordinate $w$ is measured infinitesimally below 1. Indeed, $\val{p}_{1^-} = (1/p,1)$ and $\val{w}_{1^-} = (1,\varepsilon)$. Therefore,
\begin{equation}\label{eqn:p-versus-w-1minus}
\val{p}_{1^-} < \val{w}^n_{1^-} < 1 \;\;\;\;\;\; \text{(for all $n$)}.
\end{equation}
 It is as if $w$ as squeezed between $1/\sqrt[n]{p}$ and $1$ for all $n$. See Figure \ref{fig:exotic-valuation-order} again. We will give another sense in which $\bval_{1^-}$ is near to 1 in Section \ref{subsec:valuation-final-example}.
\end{example}

\subsection{Valuation spectra}\label{subsec:valuation-spectra}
Valuation spectra are defined as equivalence classes of valuations. Let $A$ be any ring with valuations $\bval_1$ and $\bval_2$. We say that $\bval_1$ is \term{equivalent} to $\bval_2$ if for all $f,g \in A$ we have
\begin{equation*}
\val{f}_1 \leq \val{g}_1 \iff \val{f}_2 \leq \val{g}_2.
\end{equation*}
Note that $\bval_1$ and $\bval_2$ may take values in different abstract groups. When they are equivalent, there is an {\em isomorphism} of value groups turning one valuation into the other. This and another interpretation of equivalence are suggested as Exercise \ref{exercise:valuation-equivalence}. 

One impact of equivalence for valuation is that it turns seemingly true statements into actually true ones. For instance, ``there is only one continuous valuation on $\Qp$'' is only true if it is understood up to equivalence. (See Exercise \ref{exercise:Qp-valuation-check}.) After all, $\bnorm_p^2$ and $\bnorm_p$ are distinct continuous valuations. Another example of equivalence is that $\bval_{1^-}$ depends on the choice of the parameter $\varepsilon$ only up to equivalence.

Now assume $A$ is a topological ring. Its \term{continuous valuation spectrum} is
\begin{equation*}
\Cont(A) = \{\text{continuous valuations on $A$}\}/\text{(valuation equivalence)}.
\end{equation*}
We suggest as Exercise \ref{exercise:continuity-equivalent-valuations} showing that whether or not a valuation $\bval$ is continuous depends only on $\bval$ up to equivalence. Therefore, there is a natural inclusion $\Cont(A) \subseteq \Cont(A_{\disc})$,
where $A_{\disc}$ is $A$ with the discrete topology. The larger space
\begin{equation*}
\Spv(A) := \Cont(A_{\disc})
\end{equation*}
is called the \term{valuation spectrum}. No continuity qualification is imposed on $\Spv(A)$. For a non-zero ring, the valuation spectrum is always non-empty, since one always has the trivial valuation modulo a prime ideal. The non-emptiness of $\Cont(A)$ is more subtle. See Section \ref{subsection:final-comments} for a related discussion. In the remainder of this subsection and Section \ref{subsec:support-fibers}, we make formal constructions on $\Spv(A)$. The continuous valuations return in Section \ref{subsection:continuous-huber}. 

We will use $x$ to denote an element of $\Spv(A)$. Let $\bval$ be a choice of representative for the class $x$. If $f \in A$, we define notation
\begin{equation}\label{eqn:notation-valuation-point}
\val{f(x)} := \val{f}.
\end{equation}
Since the target group of $\bval$ is not well-defined, \eqref{eqn:notation-valuation-point} has no clear meaning. One route, mentioned above, is solving Exercise \ref{exercise:valuation-equivalence} to see $\val{f}$ is well-defined up to ordered group isomorphism on value groups. The route we will take is to play more loosely and agree to only use the notation \eqref{eqn:notation-valuation-point} in situations where only the class $x$, and not the choice of $\bval$, matters.

For instance, if $x \in \Spv(A)$, then its \term{support} is defined to be
\begin{equation*}
\supp(x) = \{f \in A \mid \val{f(x)} = 0\}.
\end{equation*}
The support depends only on the equivalence class $x$ because ``$\val{f} = 0$'' is the same as $``\val{f} \leq \val{0}$''. It is a prime ideal, and if $A$ is topological and $x \in \Cont(A)$, then it is closed. Prove these facts as Exercise \ref{exercise:supports-closed}.

The notation \eqref{eqn:notation-valuation-point} is also used in equipping $\Spv(A)$ with a topology. For $g,s \in A$ we define
\begin{equation}\label{eqn:subbasic-rational-subset}
U(\frac{g}{s}) = \{x \in \Spv(A) \mid \val{g(x)} \leq \val{s(x)} \neq 0\}.
\end{equation}
Weinstein observes that this  ``blends  features of the Zariski topology on schemes with the topology on rigid spaces'' (\cite[p. 6]{AWS-Perfectoid}). Indeed, $x \in U(\frac{g}{s})$ implies both that $s \not\in \frakp = \supp(x)$ (a Zariski condition) and $\val{g(x)}\leq \val{s(x)}$ (a rigid condition). A basic open set for the topology on $\Spv(A)$ is, by definition, a finite intersection
\begin{align*}
U(\frac{g_1,\dotsc,g_r}{s}) &:= \bigcap_{i=1}^r U(\frac{g_i}{s})\\
&= \{x \in \Spv(A) \mid \val{g_i(x)} \leq \val{s(x)} \neq 0 \text{ for all $i$}\}.
\end{align*}
If $A$ is a topological ring, we equip $\Cont(A) \subseteq \Spv(A)$ with the induced topology by these basic opens.

We have two warnings before going further. First, be careful about cancellation. The set $U(\frac{s}{s})$ {\em has} a condition:
\begin{equation}\label{eqn:U(s/s)}
U(\frac{s}{s}) = \{x \in \Spv(A) \mid \val{s(x)} \neq 0\}.
\end{equation}
Second, we can form $U(\frac{g_1,\dotsc,g_r}{s})$ for any choice of $g_1,\dotsc,g_r,s$. If $A$ is a topological ring then a \term{rational subset} is one of the form
\begin{equation*}
U(\frac{g_1,\dotsc,g_r}{s})
\end{equation*}
where $g_1,\dotsc,g_r,s \in A$ are chosen to generate an open $A$-ideal. This is the same condition required for rational localization of Huber rings. We return to this in Section \ref{subsubsec:rational-localization-huber-pairs}.

\subsection{The support map}\label{subsec:support-fibers}

The function $x \mapsto \supp(x)$ defines a {\em continuous} function
\begin{equation*}
\supp: \Spv(A) \rightarrow \Spec(A).
\end{equation*}
To see this, note that basic open sets in $\Spec(A)$ take the form
\begin{equation*}
D(s) = \{\frakp \in \Spec(A) \mid s \not\in \frakp\}
\end{equation*}
for $s \in A$. Thus, $\supp^{-1}(D(s)) = U(\frac{s}{s})$. 

We next describe the fibers of the support map. Let $\bval$ be  a valuation on $A$ and $\frakp$ its support. If $f \in A$ and $f \not\in \frakp$, then the strong triangle inequality implies $\val{f + g} = \val{f}$ for all $g \in \frakp$. The same holds if $f \in \frakp$, since $\frakp$ is an additive subgroup. This shows $\val{f}$ depends only on $f \bmod \frakp \in A/\frakp$ and $\bval$ factors through $A/\frakp$, on which it defines a valuation with support the zero ideal. In particular, $\bval$ extends to a valuation on the fraction field $\Frac(A/\frakp)$. In summary, there is always a commuting diagram
\begin{equation}\label{eqn:valuation-factorization}
\xymatrixcolsep{4pc}
\xymatrixrowsep{2pc}
\xymatrix{
A \ar[r]^-{\bval} \ar@{>>}[d] & \Gamma \cup \{0\}\\
A/\mathfrak p \ar@{.>}[ur]^-{\bval} \ar[r] \ar@{^{(}->}[r] & \Frac(A/\mathfrak p). \ar@{.>}[u]_-{\bval}
}
\end{equation}
The factorization only depends on $\bval$ up to equivalence, in the sense that \eqref{eqn:valuation-factorization} induces a  well-defined map
\begin{equation}\label{eqn:support-fiber-in-text-homeomorphism}
\Spv(A) \supseteq \supp^{-1}(\frakp) \rightarrow \Spv(\Frac(A/\frakp))
\end{equation}
for each prime ideal $\frakp \in \Spec(A)$. In fact, \eqref{eqn:support-fiber-in-text-homeomorphism} gives a homeomorphism
\begin{equation*}
\supp^{-1}(\frakp) \xrightarrow{\simeq} \Spv(\Frac(A/\frakp)).
\end{equation*} 
The support map therefore allows us to treat $\Spv(A)$ as families of valuations {\em over residue fields} of $\Spec(A)$.

The target of \eqref{eqn:support-fiber-in-text-homeomorphism} is the valuation spectrum of a field. This is a space classically understood through algebra. Let $K$ be any field. For $x \in \Spv(K)$, the subring
\begin{equation*}
A_x = \{\alpha \in K \mid \val{\alpha(x)}\leq 1\} \subseteq K
\end{equation*}
is called the \term{valuation ring} of $x$. Each $A_x$ is a valuation ring in the sense of commutative algebra (see  \cite[Chapter VI, \S 1, no. 2]{Bourbaki-CommutativeAlgebra}). That is, if $\alpha \in K$, then either $\alpha \in A_x$ or $\alpha^{-1} \in A_x$. Any valuation ring is a local ring. In this case, one can directly check the set of non-units in $A_x$ is equal to those $\alpha \in A_x$ such that $\val{\alpha(x)} < 1$, and that this set forms an $A_x$-ideal, written $\frakm_x$. According to Exercises \ref{exercise:valuation-to-valuation-ring}-\ref{exercise:valuation-valuation-ring-bijection}, the association
\begin{equation}\label{eqn:SpvK-valuation-rings}
\Spv(K) \xrightarrow{x\mapsto A_x} \{\text{valuation subrings $A\subseteq K$}\}
\end{equation}
is a bijection. We will use this bijection in our analysis of the closed unit disc in Section \ref{section:some-points}.

Finally, the inclusion $\Cont(A) \subseteq \Spv(A)$ endows $\Cont(A)$ with a topology as well. The support map is, by definition, continuous when restricted to $\Cont(A)$. However, the above analysis is purely algebraic. One of Huber's preliminary results on continuous valuation spectra (\cite[Theorem 3.1]{Huber-ContinuousValuations}) is an analysis of the support fibers in $\Cont(A)$ when $A$ is a Huber ring. We will not study that result. Instead, in the next section we will directly analyze continuity of valuations on Huber rings, rather than analyzing the space $\Cont(A)$ itself.

The reader may want practice manipulating $\Cont(-)$. Exercise \ref{exercise:functorial-Cont(A)} is recommended, as are Exercises \ref{exercise:adic-maps} and \ref{exercise:adic-maps-functoriality}, where the important concept of an \term{adic morphism} of Huber rings is explained.

\subsection{Continuous valuations on Huber rings}\label{subsection:continuous-huber}

We now focus on Huber rings. The primary goal is explaining which valuations on Huber rings are continuous.

Let $\Gamma$ be a totally ordered abelian group. We say $\gamma \in \Gamma\cup \{0\}$ is \term{co-final} in $\Gamma$ if, for all $\delta \in \Gamma$ we have $\gamma^n < \delta$ for $n\gg 0$. This is similar to topological nilpotence. Note that $\gamma = 0$ is automatically co-final. It is also true that every co-final $\gamma$ must be less than $1$. Indeed, if $1 \leq \gamma$ then $1 \leq \gamma^n$ for all $n\geq 0$. If $\Gamma \subseteq \R_{>0}$, then the co-final $\gamma$ are indeed just the $\gamma$ such that $\gamma < 1$.

Now suppose that $\bval : A \rightarrow \Gamma \cup \{0\}$ is a valuation. Recall the value group $\Gamma_{\bval}$ is the smallest subgroup of $\Gamma$ containing the non-zero $\val{f}$ for $f \in A$. For valuations on Huber rings, we have the following continuity criterion.

\begin{proposition}[Continuity criterion]\label{proposition:contiunity-criteria}
Let $A$ be a Huber ring. Let 
\begin{equation*}
\bval: A \rightarrow \Gamma \cup \{0\}
\end{equation*} 
be a valuation and $\Gamma_{\bval}$ be its value group. The following conditions are equivalent:
\begin{enumerate}[label=(\roman*)]
\item $\bval$ is continuous.\label{enum-part:valuation-is-continuous} 
\item If $f$ is topologically nilpotent, then $\val{f}$ is co-final in $\Gamma_{\bval}$.\label{enum-part:top-nil-implies-co-final}
\item  Suppose $(A_0,I)$ is a pair of definition and write $I = A_0 f_1 + \dotsb + A_0 f_d$. 

Then, for each $i$ we have $\val{f_i}$ is co-final in $\Gamma_{\bval}$ and $\val{ff_i} < 1$ for all $f \in A_0$.\label{enum-part:ideal-defn-criterion-continuity}
\end{enumerate}
In particular, with notation as in \ref{enum-part:ideal-defn-criterion-continuity}, suppose $\bval$ is a valuation, each $\val{f_i}$ is co-final in $\Gamma_{\bval}$, and some $\val{f_i}\neq 0$. Then,
\begin{equation*}
\text{$\bval$ is continuous} \iff \val{f} < \frac{1}{\max(\val{f_1},\dotsc,\val{f_d})} \;\;\;\;\text{ for all $f \in A_0$}.
\end{equation*}
\end{proposition}
\begin{proof}
First suppose $\bval$ is continuous. If $\gamma \in \Gamma_{\bval}$, then the set
\begin{equation*}
U_{\gamma} = \{f \in A \mid \val{f} < \gamma\}
\end{equation*}
is open in $A$. So, if $f \in A$ is topologically nilpotent, then $f^n \in U_{\gamma}$ for all $n\gg 0$. In symbols, $\val{f}^n < \gamma$ for all $n \gg 0$. So, $\val{f}$ is co-final in $\Gamma_{\bval}$. This shows \ref{enum-part:valuation-is-continuous} implies \ref{enum-part:top-nil-implies-co-final}.

Now assume that \ref{enum-part:top-nil-implies-co-final} holds. Fix the notation $(A_0,I)$ as in \ref{enum-part:ideal-defn-criterion-continuity}. The elements of $I$ are among the topologically nilpotent elements in $A$. Therefore, \ref{enum-part:top-nil-implies-co-final}  implies each $\val{f_i}$ is co-final in $\Gamma_{\bval}$. In fact, each $g \in I$ has co-final value $\val{g}$, which implies the weaker conclusion that if $g \in I$ then $\val{g} < 1$. Therefore, if $f \in A_0$ and $i=1,\dotsc,d$, then $\val{ff_i} < 1$ because $g = ff_i \in I$. We have shown \ref{enum-part:top-nil-implies-co-final} implies \ref{enum-part:ideal-defn-criterion-continuity}.

Finally, we prove that \ref{enum-part:ideal-defn-criterion-continuity} implies \ref{enum-part:valuation-is-continuous} by a direct argument. Suppose that $\gamma \in \Gamma_{\bval}$. By assumption in \ref{enum-part:ideal-defn-criterion-continuity}, $\val{f_i}$ is co-final in $\Gamma_{\bval}$ for $i=1,\dotsc,d$. There are only $d$-many $f_i$, so there exists $n\geq 0$ such that $\val{f_i}^n < \gamma$ for all $i$ at once. We now claim
\begin{equation}\label{eqn:Ind-in-Udelta}
I^{nd+1} \subseteq U_\gamma.
\end{equation}
If proven, then $U_\gamma$ is open in $A$. Since $\gamma$ was arbitrary, we have proven \ref{enum-part:ideal-defn-criterion-continuity} implies \ref{enum-part:valuation-is-continuous}.

We now show \eqref{eqn:Ind-in-Udelta}. The $A_0$-ideal $I^{nd+1}$ is generated as an abelian group by elements of the form
\begin{equation}\label{eqn:f-g-fiexpansion}
g = f f_1^{m_1}\dotsb f_d^{m_d} \;\;\;\;\;\;\; f \in A_0,\;\; m_1+\dotsb + m_d = nd + 1.
\end{equation}
Since $U_{\gamma}$ is an additive subgroup of $A$, it is enough to show $g \in U_{\gamma}$ for such $g$. Now, in \eqref{eqn:f-g-fiexpansion} we have $d$-many $m$'s and they sum to $nd+1$. So, $m_i \geq n+1$ for some $i$. Since $f \in A_0$ we have $\val{ff_i} < 1$ by \ref{enum-part:ideal-defn-criterion-continuity}. Since $\val{f_j} < 1$ for all $j$, in any case, we can write $g = f'f_i^n$ where $\val{f'} < 1$. Therefore, $\val{g} < \gamma$, completing the proof that \eqref{eqn:Ind-in-Udelta} holds.
\end{proof}

We have two reasons for explaining Proposition \ref{proposition:contiunity-criteria}. The primary reason is that we will use the criterion to check valuations are continuous in Sections \ref{section:Aplus}-\ref{section:some-points}. The secondary reason is that the criterion is only implicitly presented in Huber's paper \cite{Huber-ContinuousValuations}. It is more explicit in the notes by Conrad \cite{Conrad-PerfectoidSeminarNotes}, Morel \cite{Morel-AdicSpaceNotes}, and Wedhorn \cite{Wedhorn-AdicSpaces}. See especially \cite[Corollary 9.3.3]{Conrad-PerfectoidSeminarNotes}. However, in those sources, the criterion is established alongside arguments showing $\Cont(A)$ is a spectral space if $A$ is a Huber ring. In particular, the other references all appeal to the theorem \cite[Theorem 3.1]{Huber-ContinuousValuations} referenced at the end of the prior section. While learning this material, we have found it helpful to have a direct argument toward Proposition \ref{proposition:contiunity-criteria}, which allows for a nearly instant check on whether a valuation is continuous.

\subsection{Example:\ $\Cont(\Cp\lr w)$}\label{subsec:valuation-final-example}
In the final section of our second lecture, we revisit continuous valuations on $\Cp \lr w$. We explain an instance of specialization in $\Cont(\Cp\lr w)$, illustrate the support map, and introduce Huber's model for the closed unit disc.

Let  $A=\Cp\lr w$, which is a topological ring with the topology endowed from the Gauss norm $\bval_1 = \bnorm_{\Gauss}$. The valuations on $A$ listed in Example \ref{example:tate-algebra} are continuous. We re-list them. 
\begin{enumerate}[label=(\roman*)]
\item If $\alpha \in \OCp$, then $\val{f(x_\alpha)} = \norm{f(\alpha)}_p$ defines a point $x_\alpha \in \Cont(A)$. 
\item If $r \leq 1$, the $r$-Gauss norm $\val{f(x_r)} = \val{f}_r$ defines a point $x_r \in \Cont(A)$. 
\item We also have $x_{1^-} \in \Cont(A)$ given by $\val{f(x_{1^-})} = \val{f}_{1^-}$, where on non-zero $f$ we have
\begin{equation*}
\val{a_0 + a_1w + a_2w^2 + \dotsb}_{1^-} = \max_{i\geq 0} (\norm{a_i}_p,\varepsilon^i) \in \R_{>0} \times \R_{>0},
\end{equation*}
for some choice of $0 < \varepsilon < 1$. See Section \ref{subsec:exotic-Gamma}.
\end{enumerate}
We previously argued in \eqref{eqn:1-minus-factorize} in Section \ref{subsec:exotic-Gamma} that $\bval_1$ and $\bval_{1^-}$ are related by a commuting diagram
\begin{equation}\label{eqn:1-minus-versus-gauss}
\xymatrix{
 & A \ar[dl]_-{\bval_{1^-}} \ar[dr]^-{\bval_1}\\
 \R_{>0} \times \R_{>0} \cup \{0\} \ar[rr]_-{(a,b) \mapsto a} & & \R_{>0} \cup \{0\}.  
}
\end{equation}
Here, the horizontal projector $(a,b) \mapsto a$ is a morphism of totally ordered abelian groups. So, along with continuity of $\bval_{1^-}$, we deduce that if $g,s \in A$, then
\begin{equation*}
\val{g}_{1^-} \leq \val{s}_{1^-} \neq 0 \implies \val{g}_{1} \leq \val{s}_1 \neq 0.
\end{equation*} 
In terms of the topology on $\Cont(A)$, we have
\begin{equation*}
\bval_{1^-} \in U(\frac{g}{s}) \implies \bval_1 \in U(\frac{g}{s}).
\end{equation*} 
Therefore, each open in $\Cont(A)$ that contains $x_{1^-}$ also contains $x_1$. Said another way, $x_{1^-}$ lies in the closure of $x_1$ within $\Cont(A)$. This gives a new sense to the intuition that $x_{1^-}$ is infinitesimally close to $x_1$.

Next, we analyze supports. It is clear that $\supp(x_\alpha) = \lr{w - \alpha}$. As Exercise \ref{exercise:fiber-closed-point}, the reader can check that $x = x_\alpha$ is the {\em only} point of $\Cont(A)$ with this property. Since $\lr{w-\alpha} \in \Spec(A)$ is a closed point, so is $x_\alpha = \supp^{-1}(\lr{w-\alpha}) \in \Cont(A)$. The closedness is one way the $x_\alpha$ are distinguished from the Gauss point.

The remaining points of $\Cont(A)$ have generic support $\{0\}$. We analyze these points more closely in Section \ref{section:some-points}. The point $x_{1^-}$ is closed, while $x_r$ is non-closed whenever $r = \norm{\alpha}_p$ lies in the value group of $\Cp$. Since $x_{1^-}$ and $x_1$ have the same support {\em and} $x_{1^-}$ lies in the closure of $x_1$, we call $x_{1^-}$ a \term{vertical specialization} of $x_1$. See Exercise \ref{exercise:vertical-specialization} for details on this terminology. A cartoon is drawn in Figure \ref{fig:support-map-drawing}.

\begin{figure}[htp]
\centering
\includegraphics{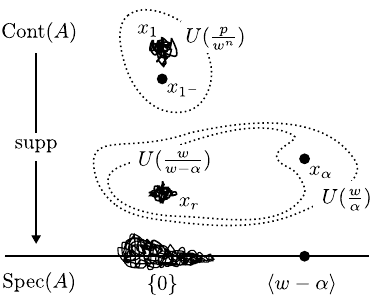}
\caption{A visualization of the support map for $A = \Cp\lr w$. The element $\alpha \in \OCp$ is meant to have $\norm{\alpha}_p = r$, and $n$ is chosen so large that $\frac{1}{p} > r^n$.}
\label{fig:support-map-drawing}
\end{figure}

We also examine the continuity criterion. We may choose $\OCp\lr w$ as a ring of definition and $p$ as a (principal) generator of an ideal of definition. Let $\bval$ be a valuation on $A$. Then, $\val{p} \neq 0$ since $p$ is a unit. The continuity criterion Proposition \ref{proposition:contiunity-criteria} implies that $\bval$ is continuous if and only if $\val{p}$ is co-final in $\Gamma_{\bval}$ and
\begin{equation}\label{eqn:f-less-pinv}
\val{f} < \val{p}^{-1}
\end{equation}
for all $f \in \OCp\lr w$. This gives a second argument for the continuity of $\bval_{1^-}$. Indeed, comparing with Figure \ref{fig:exotic-valuation-order}, we see the following.
\begin{enumerate}[label=(\alph*)]
\item The element $\val{p}_{1^-} = (\frac{1}{p},1)$ is co-final in $\R_{>0} \times \R_{>0}$.
\item If $f \in \OCp\lr w$, then $\val{f}_{1^-} \leq \val{f}_1 \leq 1 < \val{p^{-1}}_{1^-}$.
\end{enumerate}
Note, our continuity check here shows that $\val{f}_{1^-} \leq 1$ for $f \in \OCp\lr w$. This is strictly stronger than the requirement \eqref{eqn:f-less-pinv}. Indeed, it suggests a modification of $x_{1^-}$ that makes $w$ infinitesimally larger than 1 rather than smaller. We achieve this by defining
\begin{equation*}
\val{a_0 + a_1w + a_2w^2 + \dotsb}_{1^+} = \max_{i \geq 0} (\norm{a_i}_p,\varepsilon^{-i}).
\end{equation*}
(Still $0 < \varepsilon < 1$.) The value $\val{f}_{1^+}$ measures the Gauss norm in the first coordinate and the {\em largest} index coefficient realizing the Gauss norm in the second coordinate. It is still a valuation and still continuous because
\begin{equation}\label{eqn:w1plus}
1 < \val{w^n}_{1^+} = (1,\varepsilon^{-n}) < (p,1) = \val{p^{-1}}_{1^+}.
\end{equation}
We therefore have a new point $x_{1^+} \in \Cont(A)$. It lies in the closure of $x_1$ as before. 

Finally, Huber's model for the \term{closed unit disc} is given by
\begin{equation*}
D = \{x \in \Cont(\Cp\lr w) \mid \val{w(x)} \leq 1\}.
\end{equation*}
From the definitions, we see $D$ contains $x_\alpha$ for each $\alpha \in \OCp$, it contains $x_r$ for all $0 < r\leq 1$, and it contains $x_{1^-}$. There is a strict containment $D \subsetneq \Cont(\Cp\lr w)$ because $x_{1^+}\not\in D$ by \eqref{eqn:w1plus}. One of the primary results in Huber's \cite{Huber-ContinuousValuations} is the following theorem. Here we specialize to the context of $D$. See Theorem \ref{thm:hubers-main-theorem} later, as well.

\begin{theorem}[Huber]\label{thm:hubers-main-theorem-D}
The topological space $D$ is quasi-compact and quasi-separated. Moreover, if $g_1,\dotsc,g_r,s \in \Cp\lr w$ generate the unit ideal, then the rational subset $U(\frac{g_1,\dotsc,g_r}{s})$ is also quasi-compact.
\end{theorem}

Note that the hypothesis in Theorem \ref{thm:hubers-main-theorem-D} is the same that appears in the rational localization Theorem \ref{theorem:rational-localization}. Compare with Section \ref{subsubsec:rational-localization-huber-pairs}.

To illustrate Huber's theorem, let us return to the issue of disconnection raised on page \pageref{pageref:tate-decomposition}. The closed unit disc $D$ certainly separates into $D = V_1 \cup V_{<1}$ where
\begin{equation*}
V_1 = \{x \in D \mid \val{w(x)} = 1\} \;\;\;\; \text{and} \;\;\;\; V_{<1} = \{x \in D \mid \val{w(x)} < 1\}.
\end{equation*}
Since $\val{p(x)} \neq 0$ and $\val{w(x)} < 1$ on $V_{<1}$, it would {\em seem} that $V_{<1}$ can be written as a union
\begin{equation}\label{eqn:fake-cover}
V_{<1} = \bigcup_{0 < n} U(\frac{w^n}{p}) = \{x \in D \mid \val{w(x)}^n < \val{p(x)} \text{ for $n\gg 0$}\}.
\end{equation}
However, the point $x_{1^-} \in D$ lies in a gap between $V_{<1}$ and the union. This is good! If there were not a gap, then Huber's theorem would be contradicted by $V_1$ together with the ``cover'' of $V_{<1}$ alleged in \eqref{eqn:fake-cover}.

\subsection*{Section \thesection~ Exercises}

\begin{exercise}\label{exercise:r-norm-polynomials}
Let $\Cp[w]$ be given the topology induced from the Gauss norm. For $0\leq r < \infty$, define
\begin{equation*}
\val{a_0+a_1w + \dotsb + a_mw^m}_r = \max_{n} \norm{a_n}r^n.
\end{equation*}
\begin{enumerate}[label=(\alph*)]
\item Show that $\bval_r$ is a valuation on $\Cp[w]$.
\item Show that $\bval_r$ is continuous if and only if $r \leq 1$.
\end{enumerate}
\end{exercise}
\begin{solution}
\begin{enumerate}[label=(\alph*)]
\item Gau\ss\; lemma.
\item Fill in.
\end{enumerate}
\end{solution}

\begin{exercise}\label{exercise:1minus-motivation}
Let $f \in \Cp\lr w$. Assume $a = \val{f}_1 = \norm{\alpha}_p$ for some $\alpha \in \Cp$. Therefore, $\alpha^{-1}f \in \OCp\lr w^\times$ has a non-zero reduction $\overline f := \alpha^{-1}f \bmod{\frakm_{\OCp}} \in \Fpbar[w]$. Show that $\val{f}_{1^-} = (a,\varepsilon^n)$ where $n$ is the order of vanishing of $\overline f$ at $w=0$.
\end{exercise}

\begin{exercise}\label{exercise:continuity-closed-conditions}
Let $\bval: A \rightarrow \Gamma \cup \{0\}$ be a valuation on a topological ring $A$. For $\gamma \in \Gamma_{\bval}$, define
\begin{equation*}
\overline{U}_\gamma = \{f \in A \mid \val f \leq \gamma \}.
\end{equation*}
\begin{enumerate}[label=(\alph*)]
\item Show that if $\bval$ is continuous, then each $\overline U_\gamma$ is open.
\item Show that if $\bval$ is non-trivial and each $\overline U_{\gamma}$ is open, then $\bval$ is continuous.\label{exercise-part:continuity-closed-condition-non-trivial}
\end{enumerate}
\end{exercise}
\begin{solution}
Part (a).  $\overline U_\gamma$ always contains $U_\gamma$. Therefore if each $U_\gamma$ is open, then so is each $\overline U_\gamma$.

Part (b). if $\bval$ is non-trivial, then there is $\gamma_0 < 1$ in the value group. In particular, for all $\gamma$ there is some $\gamma' = \gamma \gamma_0 < \gamma$. Thus, for all $\gamma$, we have:
\begin{equation*}
U_\gamma = \bigcup_{\gamma' < \gamma} \overline U_\gamma.
\end{equation*}

Part (c). Here is the issue. Suppose that there exists some $\gamma$ such that $\gamma < \delta$ for all $\delta \in \Gamma_{\bval}$. Then $U_\gamma = \supp(\bval)$ is just the support of the valuation. This is open only if $A$ is discrete.
\end{solution}

\begin{exercise}\label{exercise:trivial-valuation}
Let $A$ be a topological ring. Show that $\bval_{\triv}$ is continuous on $A$ if and only if $A$ is discrete.
\end{exercise}
\begin{solution}
If $A$ is discrete, anything is open. On the other hand, it $\bval_{\triv}$ is continuous then $\{0\} =  \supp(\bval_{\triv}) = \{a \in A \mid \val{a}_{\triv} < 1$ is open.
\end{solution}

\begin{exercise}\label{exercise:valuation-equivalence}
Let $A$ be a ring and $\bval_1, \bval_2$ two valuations on $A$. Show that the following conditions are equivalent:
\begin{enumerate}[label=(\roman*)]
\item $\bval_1$ is equivalent to $\bval_2$.
\item $\supp(\bval_1) = \frakp = \supp(\bval_2)$ and the induced valuations on $\Frac(A/\frakp)$ have the same valuation rings.\label{enum-part:valuation-equivalence-valuation-rings}
\item There exists an isomorphism $\varphi: \Gamma_{\bval_1} \xrightarrow{\simeq} \Gamma_{\bval_2}$ of totally ordered abelian groups making the following diagram commute
\begin{equation*}
\xymatrix{
 & A   \ar[dl]_-{\bval_1} \ar[dr]^-{\bval_2}\\
\Gamma_{\bval_1} \ar[rr]_-{\varphi} & & \Gamma_{\bval_2}
 }
\end{equation*}
\end{enumerate}
\end{exercise}
\begin{solution}
See \cite[Proposition/Definition 1.27]{Wedhorn-AdicSpaces}.
\end{solution}

\begin{exercise}\label{exercise:Qp-valuation-check}
Consider the $p$-adic numbers $\Qp$.
\begin{enumerate}[label=(\alph*)]
\item Show that if $\bval: \Qp \rightarrow \Gamma \cup \{0\}$ with $\Gamma$ cyclic, then $\bval$ is equivalent to either $\bval_\triv$ or $\bnorm_p$.
\item Show that any continuous valuation on $\Qp$ is equivalent to $\bnorm_p$.\label{exercise-part:Qp-valuation-check-continuity}
\item Show that \ref{exercise-part:Qp-valuation-check-continuity} also holds for $\Cp$ instead of $\Qp$.
\label{exercise-part:Qp-valuation-check-continuity-Cp}
\end{enumerate}
\end{exercise}

\begin{solution}
Part (a). Write $\Gamma = \gamma^{\Z}$ with $\gamma > 1$. We claim that $u \in \Zp^\times$ must have $\val{u} = 1$. First assume that $u \equiv 1 \pmod{p}$. Then, $x^n - u \equiv 0 \pmod{p}$ has simple roots for all $n$ such that $p\nmid n$. Therefore, by Hensel's lemma, each $u \in 1+p\Zp$ has an $n$-th root $u^{1/n}$. It follows that for each $p\nmid n$, the value $\val{u}$ also has $n$-th roots for all $p \nmid n$. That forces $\val{u} = 1$. This is only $1+p\Zp$. For something like $\val{2}$ you can instead look at $x^d - 2$ where $(d,(p-1)p) = 1$ and see that $x^d \equiv 2 \pmod{p}$ has a  unique solution to lift to a root $2^{1/d}$ and then run the same argument again. Therefore $\val{u} = 1$ on $\Zp^\times$. Finally, $\val{p}\leq 1$ and therefore is equivalent to $\bnorm{}_p$ when $\val{p}<1$ or $\val{}$ is equivalent to $\bval_{\triv}$ when $\val{p} = 1$.

Part (b). Let $R$ be the valuation ring of $\Qp$. If $\bval$ is continuous, then $R$ is open, since $R$ contains $\{x \in \Qp \mid \val{x} < 1\}$. Therefore, $R$ contains $p^n\Zp$ for some $n\gg 0$. But $R$ also contains $\Z$ and therefore $R$ contains $\Zp$. Once $R$ contains $\Zp$ you know that $R$ is either $\Zp$ or $\Qp$, corresponding to the continuous $p$-adic norm on the only hand and the discontinuous trivial norm. The only option is that $\bval$ is equivalent to $\bnorm_p$.

Part (c). Fix $\Qpbar$. If $K/\Qp$ is a finite extension then the $p$-adic norm is up to equivalence the unique continuous norm, by the same argument as in \ref{exercise-part:Qp-valuation-check-continuity}. If $\bnorm$ is continuous on $\Cp$ it must be the $p$-adic norm on all these subfields compatibly. I think that does it.
\end{solution}

\begin{exercise}\label{exercise:continuity-equivalent-valuations}
Let $\bval_1$ and $\bval_2$ be equivalent valuations on a topological ring. Show that $\bval_1$ is continuous if and only if $\bval_2$ is continuous.
\end{exercise}

\begin{solution}
This is easier with Exercise \ref{exercise:valuation-equivalence} clarifying the role of the value group. However, it can also be done directly from the definitions.

Suppose that $\bval_j$ is non-trivial (the trivial valuation's equivalence class consists only of the trivial valuation) and $\bval_1$ is continuous. For $\gamma = \frac{\val{g}_2}{\val{h}_2} \in \Gamma_{\bval_2}$ we define
\begin{equation*}
\overline U_{\gamma_2} = \{f \in A \mid \val{f}_2 \leq \frac{\val{g}_2}{\val{h}_2}\}.
\end{equation*}
By equivalence this is the same as
\begin{equation*}
\overline U_{\gamma_1} = \{f \in A \mid \val{f}_1 \leq \frac{\val{g}_1}{\val{h}_1}\}.
\end{equation*}
where $\gamma_1 = \frac{\val{g}_1}{\val{h}_1}$. Therefore, if $\bval_1$ is continuous, then $\overline U_{\gamma_1} = \overline U_{\gamma_2}$ is open. Then $\bval_2$ is continuous by Exercise \ref{exercise:continuity-closed-conditions}\ref{exercise-part:continuity-closed-condition-non-trivial}.
\end{solution}

\begin{exercise}\label{exercise:functorial-Cont(A)}
Let $\varphi: A \rightarrow B$ be a map of rings.
\begin{enumerate}[label=(\alph*)]
\item Show that the induced map $\Spv(B) \rightarrow \Spv(A)$ is continuous.
\item Show that $A$ and $B$ are topological rings and $\varphi$ is continuous, then the induced map $\Cont(B) \rightarrow \Cont(A)$ is well-defined (and continuous).
\end{enumerate}
\end{exercise}
\begin{solution}
\begin{enumerate}[label=(\alph*)]
\item Let $\Spv(B) \xrightarrow{\psi} \Spv(A)$ be the induced map. Thus for $f \in A$ and $y \in \Spv(B)$ we have
\begin{equation*}
\val{f(\psi(y))} := \val{\varphi(f)(y)}.
\end{equation*}
In particular, for $g,s \in A$ we have
\begin{equation*}
\psi^{-1}(U(\frac{g}{s}) = U(\frac{\varphi(g)}{\varphi(s)}).
\end{equation*}
\item Suppose $\bval: B \rightarrow \Gamma\cup \{0\}$ is a continuous valuation. Define $\bval' : A \rightarrow \Gamma \cup \{0\}$ by
\begin{equation*}
\val{f}' := \val{\varphi(f)}.
\end{equation*}
By definition we have an inclusion of value groups $\Gamma_{\bval'} \subseteq \Gamma_{\bval}$. Suppose $\gamma \in \Gamma_{\bval'}$. Then,
\begin{equation*}
U_\gamma^B = \{g \in B \mid \val{g} < \gamma\}
\end{equation*}
is open in $B$, and
\begin{equation*}
U_\gamma^A = \{f \in A \mid \val{f}' < \gamma\} = \varphi^{-1}(U_\gamma^B).
\end{equation*}
Therefore $U_\gamma^A$ is open in $A$, showing that $\bval'$ s continuous.
\end{enumerate}
\end{solution}

\begin{exercise}\label{exercise:adic-maps}
Let $A$ and $B$ be Huber rings and $\varphi : A \rightarrow B$ a ring homomorphism. We call $\varphi$ an \term{adic morphism} if there exists a pair of definition $(A_0,I)$ for $A$ and a ring of definition $B_0$ for $B$ such that $\varphi(A_0) \subseteq B_0$ and $(B_0,\varphi(I)B_0)$ is a pair of definition for $B$.
\begin{enumerate}[label=(\alph*)]
\item Show that if $\varphi$ is an adic morphism, then $\varphi$ is continuous.
\item Show that if $A$ is a Tate ring and $\varphi$ is continuous, then $B$ is a Tate ring and $\varphi$ is adic.\label{exercise-part:adic-maps-Tate-rings}
\item Suppose $g_1,\dotsc,g_r,s \in A$ generate an open $A$-ideal. Show that the localization map $A \rightarrow A(\frac{g_1,\dotsc,g_r}{s})$ is adic.
\item Suppose $\varphi$ is continuous. Show that $\varphi$ is adic if and only if for {\em any} rings of definitions $A_0 \subseteq A$ and $B_0 \subseteq B$, if $\varphi(A_0) \subseteq B_0$ and $I \subseteq A_0$ is an ideal of definition, then $\varphi(I)B_0 \subseteq B_0$ is an ideal of definition.
\end{enumerate}
\end{exercise}

\begin{solution}
\begin{enumerate}[label=(\alph*)]
\item 
\item 
\item 
\end{enumerate}
\end{solution}

\begin{exercise}\label{exercise:adic-maps-functoriality}
Let $A$ and $B$ be Huber rings and $\varphi: A\rightarrow B$ an adic morphism. Let $\psi: \Cont(B) \rightarrow \Cont(A)$ be the induced map. Show that if $U \subseteq \Cont(A)$ is a rational subset, then $\psi^{-1}(U) \subseteq \Cont(B)$ is also a rational subset.
\end{exercise}

\begin{solution}
As we saw Exercise \ref{exercise:continuity-equivalent-valuations}, for $g,s \in A$ we have
\begin{equation*}
\psi^{-1}(U(\frac{g}{s}) = U(\frac{\varphi(g)}{\varphi(s)}).
\end{equation*}
To solve the exercise, we therefore need to show that if $g_1,\dotsc,g_r,s \in A$ generate an open $A$-ideal then $\varphi(g_1),\dotsc,\varphi(g_r),\varphi(s) \in B$ generate an open $B$-ideal.

Let $B_0$ be any ring of definition in $B$ and choose $A_0 \subseteq A$ such that $\varphi(A_0) \subseteq B_0$. Let $\fraka$ be the ideal generated by $g_1,\dotsc,g_r,s$ in $A$. By assumption, there exists an ideal of definition $I \subseteq A_0$ such that $I \subseteq \fraka$. By part (c) of Exercise \ref{exercise:adic-maps}, the ideal $\varphi(I)B_0 \subseteq B_0$ is an open in $B$, and clearly $\varphi(I)B_0 \subseteq \varphi(\fraka)$.
\end{solution}

\begin{exercise}\label{exercise:supports-closed}
Let $A$ be a ring.
\begin{enumerate}[label=(\alph*)]
\item Show that if $x \in \Spv(A)$, then $\supp(x)$ is a prime ideal in $A$.
\item Show that if $x,y \in \Spv(A)$ and $y \in \overline{\{x\}}$, then $\supp(x) \subseteq \supp(y)$.\label{exercise-part:supports-closed-specialization}
\item Assume $A$ is a topological ring and $x \in \Cont(A)$. Show that $\supp(x)$ is closed in $A$.
\end{enumerate}
\end{exercise}
\begin{solution}
\begin{enumerate}[label=(\alph*)]
\item Suppose $f,g \in A$ and $\val{f(x)g(x)} = 0$. The only solutions to $\gamma_1\gamma_2 = 0$ require one of the $\gamma_i$ to be zero. So, if $\val{f(x)} \neq 0$, then $\val{g(x)} = 0$.
\item Let $U_\gamma = \{f \in A \mid \val{f(x)} < \gamma\}$. If $x \in \Cont(A)$, then $U_\gamma$ is open for all $\gamma \in \Gamma$. The support is given by
\begin{equation*}
\supp(x) = A - \bigcup_{\gamma \in \Gamma} U_{\gamma}.
\end{equation*}
Therefore, the support is closed.
\end{enumerate}
\end{solution}

\begin{exercise}\label{exercise:valuation-to-valuation-ring}
Let $K$ be a field. A \term{valuation subring} of $K$ is a subring $A \subseteq K$ such that if $\alpha \neq 0$ in $K$ then either $\alpha \in A$ or $\alpha^{-1} \in A$. Let $x \in \Spv(K)$. Define
\begin{equation*}
A_x = \{\alpha \in K \mid \val{\alpha(x)} \leq 1\}.
\end{equation*}
\begin{enumerate}[label=(\alph*)]
\item Show that $A_x$ is a valuation subring of $K$.
\item Show directly that $\frakm_x = \{\alpha \in A_x \mid \val{\alpha(x)} < 1\}$ is an ideal in $A_x$, and it consists of all the non-units in $A_x$. Conclude that $A_x$ is a local ring.
\item Show that $y \in \overline{\{x\}}$ within $\Spv(K)$ if and only if $A_y \subseteq A_x$.\label{exercise-part:valuation-to-valuation-ring-closure}
\end{enumerate}
\end{exercise}
\begin{solution}
Part (a). If $\alpha \in K^\times$ then $\val{\alpha(x)} \neq 0$. If $\val{\alpha(x)} > 1$ then the fact that the value group $\Gamma_x$ is totally ordered implies that $\val{\alpha^{-1}(x)} = \val{\alpha(x)}^{-1} < 1$. Therefore either $\alpha$ or $\alpha^{-1}$ lies n $A_x$.

Note that a sub-basis of rational subsets of $\Spv(K)$ can be defined by $U_\alpha = \{x \in \Spv(K) \mid \val{\alpha(x)}\leq 1\}$. Therefore we have
\begin{align*}
x \in \overline{\{y\}} &\iff \text{for all $\alpha$}, (x \in U_\alpha \implies y \in U_\alpha)\\
&\iff \text{for all $\alpha$)} (\alpha \in A_x \implies \alpha \in A_y)\\
&\iff A_x \subseteq A_y.
\end{align*}
\end{solution}

\begin{exercise}\label{exercise:valuation-ring-to-valuation}
Let $K$ be a field and $A \subseteq K$ a valuation subring. Set $
\Gamma_A = K^\times/A^\times$.
\begin{enumerate}[label=(\alph*)]
\item Show that $\Gamma_A $ is a totally ordered abelian group under the ordering
\begin{equation}\label{eqn:valuation-ring-value-group}
\alpha A^\times \leq \beta A^\times \iff \alpha A \subseteq \beta A.
\end{equation}
\item Show that the function
\begin{equation*}
\val{\alpha(x_A)} = \begin{cases}
 \alpha A^\times & \text{if $\alpha \neq 0$;}\\
 0 & \text{if $\alpha = 0$,}
 \end{cases}
\end{equation*}
defines a valuation $x_A \in \Spv(K)$.
\end{enumerate}
\end{exercise}

\begin{solution}
Part (a). It is clear that $\alpha A^\times = \beta A^\times$ if and only if $\alpha A^\times \leq \beta A^\times$ and $\beta A^\times \leq \alpha A^\times$. Suppose that $\alpha A^\times \neq \beta A^\times$. Let $\gamma = \alpha/\beta$. Since $A$ is a valuation ring, either $\gamma \in A$ or $\gamma^{-1} \in A$. Suppose $\gamma \in A$. Then $\alpha A \subseteq \beta A$. Otherwise, if $\gamma^{-1} \in A$, then $\beta A \subseteq \alpha A$. This shows the order is a total order.

Part (b). It is clear that $0\mapsto 0$ and $1 \mapsto 1$. Suppose $\alpha$ and $\beta$ are given. Without loss of generality, since $A$ is a valuation ring, we suppose that $\alpha A \subseteq \beta A$. Then, 
\begin{equation*}
(\alpha + \beta)A \subseteq \beta A.
\end{equation*}
This proves the triangle inequality. The multiplicative property is clear.
\end{solution}

\begin{exercise}\label{exercise:valuation-valuation-ring-bijection}
Show that the maps $x \mapsto A_x$ and $A \mapsto x_A$ from Exercises \ref{exercise:valuation-to-valuation-ring} and \ref{exercise:valuation-ring-to-valuation} are inverse bijections between $\Spv(K)$ and the set of valuation subrings $A \subseteq K$.
\end{exercise}

\begin{solution}
First we have
\begin{equation*}
A_{x_A} = \{\alpha \in K \mid \alpha A \subseteq A\} = A.
\end{equation*}
(If $\alpha A \subseteq A$ then $\alpha \cdot 1 \in A$. Conversely, $A \cdot A \subseteq A$.)

Second, we have
\begin{equation*}
|\alpha(x_{A_x})| = \alpha A_x
\end{equation*}
Therefore,
\begin{align*}
|\alpha(x_{A_x})| \leq |\beta(x_{A_x})| &\iff \alpha A_x \subseteq \beta A_x \\
&\iff \frac{\alpha}{\beta} \in A_x\\
&\iff \val{\alpha(x)} \leq \val{\beta(x)}.
\end{align*}
We have shown that $x$ is equivalent to $x_{A_x}$.
\end{solution}

\begin{exercise}\label{exercise:fiber-closed-point}
Show that if $x \in \Cont(\Cp\lr w)$ and $\supp(x) = \lr{w-\alpha}$ for some $\alpha \in \OCp$, then $x = x_\alpha$.
\end{exercise}
\begin{hint}
Use the criterion \ref{enum-part:valuation-equivalence-valuation-rings} in Exercise \ref{exercise:valuation-equivalence} and Exercise \ref{exercise:Qp-valuation-check}\ref{exercise-part:Qp-valuation-check-continuity-Cp}.
\end{hint}

\begin{solution}
Clearly $\supp(x) = \supp(x_y)$. By Exercise \ref{exercise:valuation-equivalence}, it suffices to show the valuation rings of $x$ and $x_y$ are the same. But these valuation rings are valuation rings within $\Cp$ for some continuous valuation. There is only one --- it is $\OCp$.
\end{solution}

\begin{exercise}\label{exercise:convex-subgroups}
Let $\Gamma$ be a totally ordered abelian group. A subgroup $\Delta$ is \term{convex} in $\Gamma$ if, for all $\delta_1,\delta_2 \in \Delta$ and $\gamma \in \Gamma$ such that $\delta_1 \leq \gamma \leq \delta_2$, then $\gamma \in \Delta$. 

Assume now that $\Delta$ is convex in $\Gamma$. Define $\leq$ on $\Gamma/\Delta$ by
\begin{equation*}
\gamma \Delta \leq \gamma' \Delta \iff \gamma \leq \gamma' \delta \text{ for some $\delta \in \Delta$.}
\end{equation*}
\begin{enumerate}[label=(\alph*)]
\item Show that $\leq$ defines a total order on $\Gamma/\Delta$ and $\Gamma \rightarrow \Gamma/\Delta$ is a morphism of totally ordered abelian groups.\label{exercise-part:convex-subgroups-total-order}
\item Suppose that the exact sequence
\begin{equation*}
1 \rightarrow \Delta \xrightarrow{\iota} \Gamma \xrightarrow{\pi} \Gamma/\Delta \rightarrow 1
\end{equation*}
is split as totally ordered abelian groups. Let $s : \Gamma/\Delta \rightarrow \Gamma$ be a splitting. Show that the group isomorphism
\begin{equation*}
\Gamma/\Delta \times \Delta \xrightarrow{(s,\iota)} \Gamma
\end{equation*}
is an isomorphism of totally ordered abelian groups when the product group is given the {\em lexicographic} order.\label{exercise-part:convex-subgroups-splitting}
\end{enumerate}
\end{exercise}
\begin{hint}
In \ref{exercise-part:convex-subgroups-total-order}, convexity is only used to show the ``strongly symmetric'' property of an order relation --- in this case to prove that $\gamma\Delta = \gamma'\Delta$ if and only if both $\gamma\Delta \leq \gamma'\Delta$ and $\gamma'\Delta \leq \gamma\Delta$.
\end{hint}
\begin{solution}
\begin{enumerate}[label=(\alph*)]
\item The reflexive property ($\gamma\Delta \leq \gamma\Delta$), the transitive property ($\gamma\Delta \leq \gamma'\Delta$ and $\gamma'\Delta \leq \gamma''\Delta$ implies $\gamma\Delta \leq \gamma''\Delta$) are both clear. We now check the strongly symmetric property. Suppose $\gamma\Delta \leq \gamma'\Delta$ and $\gamma'\Delta \leq \gamma\Delta$. Thus $\gamma \leq \gamma'\delta$ for some $\delta \in \Delta$ and $\gamma'\leq \gamma'\eta$ for some $\eta \in \Delta$. It follows that
\begin{equation*}
\eta^{-1} \leq \frac{\gamma}{\gamma'} \leq \delta.
\end{equation*}
Since $\Delta$ in convex, we conclude $\gamma/\gamma' \in \Delta$. We have shown that $\leq$ is an order relation on $\Gamma/\Delta$. Finally we show it is a total order. Suppose $\gamma,\gamma' \in \Gamma$. Then if $\gamma \leq \gamma'$ we have $\gamma\Delta \leq \gamma'\Delta$. Otherwise $\gamma' \leq \gamma$ and in that case $\gamma'\Delta \leq \gamma\Delta$.
\item This is clear.
\item Suppose $\gamma\Delta < 1$ in $\Gamma/\Delta$. Thus $\gamma < \delta$ for some $\delta \in \Delta$. Replacing $\gamma$ by $\gamma' = \gamma\delta^{-1}$ we have $\gamma\Delta = \gamma'\Delta$ and $\gamma' < 1$.
\end{enumerate}
\end{solution}

\begin{exercise}
Let $\varphi: \Gamma \rightarrow \Gamma'$ be a morphism of totally ordered abelian groups.
\begin{enumerate}[label=(\alph*)]
\item Show that $\ker(\varphi) \subseteq \Gamma$ is a convex subgroup.
\item Show that $\im(\varphi) \subseteq \Gamma'$ is a totally ordered abelian group.
\item Show that $\varphi: \Gamma/\ker(\varphi) \xrightarrow{\simeq} \im(\varphi)$ as totally ordered abelian groups.
\end{enumerate}
\end{exercise}

\begin{exercise}\label{exercise:vertical-specialization}
Let $A$ be a ring. For $x,y \in \Spv(A)$, we say $y$ is a \term{vertical specialization} of $x$ if
\begin{enumerate}[label=(\roman*)]
\item $y \in \overline{\{x\}}$, and
\item $\supp(y) = \supp(x)$.
\end{enumerate}
We also say that $x$ is a \term{vertical generization} of $y$.
\begin{enumerate}[label=(\alph*)]
\item Show that if $x$ is a vertical generization of $y$, then there is a natural quotient map $\Gamma_{y} \twoheadrightarrow \Gamma_x$ of totally ordered abelian groups.\label{exercise-part:vertical-specialization-quotient}
\item If $y \in \Spv(A)$, show that
\begin{align*}
\{\text{vertical generizations of $y$}\} &\rightarrow \{\Delta \subseteq \Gamma_y \text{ convex subgroups}\}\\
x &\mapsto \ker(\Gamma_y \twoheadrightarrow \Gamma_x)
\end{align*}
is a bijection.\label{exercise-part:vertical-specialization-bijection}
\item Suppose $A$ is a topological ring and $y \in \Cont(A)$. Show that if $x$ is a vertical generization then either $x \in \Cont(A)$ or $x$ is the trivial valuation modulo $\supp(y)$.\label{exercise-part:vertical-specialization-continuous}
\end{enumerate}
\end{exercise}
\begin{hint}
If $\supp(x) = \supp(y)$, then $x$ and $y$ can be viewed as valuations on the same field. Then, use Exercises \ref{exercise:valuation-to-valuation-ring}-\ref{exercise:valuation-valuation-ring-bijection} for \ref{exercise-part:vertical-specialization-quotient}. For \ref{exercise-part:vertical-specialization-continuous}, use Exercise \ref{exercise:continuity-closed-conditions}.
\end{hint}

\begin{solution}
See \cite[Proposition 4.2.4]{Conrad-PerfectoidSeminarNotes} and \cite[Theorem 8.2.1]{Conrad-PerfectoidSeminarNotes}.
\end{solution}

\section{Constructions with adic spectra}\label{section:Aplus}

Let $A$ be a Huber ring. This lecture focuses on the adic spectrum
\begin{equation}\label{eqn:Spa-A-Aplus}
\Spa(A,A^+) = \{x \in \Cont(A) \mid \val{f(x)} \leq 1 \text{ for all $f \in A^+$}\},
\end{equation}
which is the topic of H\"ubner's initial lecture \cite{Hubner-AdicSpaces}.

The ring $A^+$ is a \term{ring of integral elements}, which means $A^+ \subseteq A^{\circ}$ and $A^+$ is open and integrally closed in $A$. Proposition \ref{proposition:power-bounded-union-rings-of-definition} shows $A^{\circ}$ is always a ring of integral elements. A pair $(A,A^+)$ is called a \term{Huber pair}. We make $X=\Spa(A,A^+)$ a topological space via the inclusion $X \subseteq \Cont(A)$. 

Our initial goal in Section \ref{subsec:bounds} is to discuss the bounds imposed on adic spectra. In Sections \ref{subsubsec:rational-localization-huber-pairs}-\ref{subsubsec:completion-huber-pairs}, we describe how localization, tensor products, and completions impact Huber pairs. Then, in Section \ref{subsec:adic-nullstellensatz}, we prove that
\begin{equation}\label{eqn:A+-equality}
A^+ = \{f \in A \mid \val{f(x)}\leq 1 \text{ for all $x \in X$}\}
\end{equation}
when $X = \Spa(A,A^+)$.

The bulk of our energy is spent on \eqref{eqn:A+-equality}. It reminds one of the statement $``I(V(I)) = I$ if $I$ is a radical ideal'', which one encounters as the Nullstellensatz in algebraic geometry. For this reason, Conrad even refers to \eqref{eqn:A+-equality} as an {\em adic Nullstellensatz} in \cite[Theorem 2.25]{Conrad-AdicMSRILectures}. We observe here, in further support of this name, that the first step in the proof we present involves localization in a way reminiscent of  the ``Rabinowitsch trick'' used in proofs of the Nullstellensatz.

Presenting \eqref{eqn:A+-equality} seems near optimal, in terms of satisfaction, for a proof using only what we explained in Sections \ref{section:huber} and \ref{section:valuation}. It provides a chance to use rational localization of Huber rings and we also get to introduce a  technique called horizontal specialization, which  complements the vertical specializations described in Section \ref{subsec:valuation-final-example} and Exercise \ref{exercise:vertical-specialization}.

One topic we will not address is the role of the containment $A^+ \subseteq A^{\circ}$. It is not required for \eqref{eqn:A+-equality}. The containment is crucial in Huber's study \cite[Proposition 3.6]{Huber-ContinuousValuations} of whether or not adic spectra are empty, where he proves that $\Spa(A,A^+) = \emptyset$ if and only if $\{0\}$ is dense in $A$. (Recall $\Spec(A) = \emptyset$ if and only if $A = \{0\}$.)\label{pageref:emptiness}

\subsection{Bounds on adic spectra}\label{subsec:bounds}

Our initial goal is clarifying how adic spectra generalize the closed unit disc
\begin{equation*}
D = \{x \in \Cont(\Cp\lr w) \mid \val{w(x)} \leq 1\}.
\end{equation*}
In defining $D$, we impose a bound $\val{w(x)}\leq 1$ on only $w \in \Cp\lr w$, while \eqref{eqn:Spa-A-Aplus} imposes bounds on all of $A^+$. There {\em is} a difference in the style of definition.  Examining the definition of rational subsets, imposing bounds on single functions, or a short list of them, is natural. Instead of $\Spa(A,A^+)$, one might consider the more basic object
\begin{equation}\label{eqn:Spa-A-Sigma}
\Spa(A,\Sigma) = \{x \in \Cont(A) \mid \val{f(x)} \leq 1 \text{ for all $f \in \Sigma$}\},
\end{equation}
where $\Sigma \subseteq A$ is any subset.

There are two basic remarks. First, sets defined as \eqref{eqn:Spa-A-Sigma} present technical challenges. Writing proofs would require learning how to track $\Sigma$ while applying algebro-topological constructions to $A$. It would be similar to tracking functions that define a projective algebraic variety rather than an ideal sheaf. Second, there is no loss of generality in focusing only on $\Sigma = A^+$ where $A^+$ is open and integrally closed in $A$. Indeed, for all $\Sigma$, there exists an open and integrally closed ring $A^+ \supseteq \Sigma$ such that if $\val{f(x)} \leq 1$ for all $f \in \Sigma$, then $\val{f(x)} \leq 1$ for all $f \in A^+$. See Exercises \ref{exercise:top-nil-always-in-A+}-\ref{exercise:pass-to-integral-closure}. Therefore, if $\Sigma \subseteq A^{\circ}$, then
\begin{equation*}
\Spa(A,\Sigma) = \Spa(A,A^+)
\end{equation*} 
where $A^+$ is a ring of integral elements. In the case $A = \Cp\lr w$ and $\Sigma = \{w\}$, the relevant ring is even $A^+ = A^{\circ} = \OCp\lr w$. Therefore, 
\begin{equation}\label{eqn:D=Spa}
D = \{x \in \Cont(\Cp\lr w) \mid \val{w(x)} \leq 1\} = \Spa(\Cp\lr w, \OCp\lr w).
\end{equation}
See Exercise \ref{exercise:unit-disc-closure}.

For the record, a general version of Huber's theorem Theorem \ref{thm:hubers-main-theorem-D} is:
\begin{theorem}[{Huber, \cite[Theorem 3.5]{Huber-ContinuousValuations}}]\label{thm:hubers-main-theorem}
Let $(A,A^+)$ be a Huber pair. Then, $\Spa(A,A^+)$ is quasi-compact and quasi-separated. If $g_1,\dotsc,g_r,s \in A$ generates an open $A$-ideal, then the rational subset $U(\frac{g_1,\dotsc,g_r}{s})$ is quasi-compact as well.
\end{theorem}

Technically, if you look at Huber's theorem you will find that $\Spa(A,A^+)$ is a \term{spectral space} and that each rational subset is constructible. Spectral space were defined by Hochster \cite{Hochster-Spectral}. The theorem we stated is just a portion of Huber's theorem, since spectral spaces are, in particular, quasi-compact and quasi-separated and, in addition, constructible subsets of quasi-compact spaces are quasi-compact.

Instead of proving this theorem, we move on to explain how constructions with Huber rings extend to constructions with Huber pairs.

\subsection{Construction:\ rational localization}\label{subsubsec:rational-localization-huber-pairs}

Suppose $(A,A^+)$ is a Huber pair and $g_1,\dotsc,g_r ,s \in A$ generate an open $A$-ideal. In Section \ref{subsection:rational-localization} we defined the rational localization
\begin{equation*}
B = A(\frac{g_1,\dotsc,g_r}{s}).
\end{equation*}
The underlying ring is $B = A[\frac{1}{s}]$. Suppose $(A_0,I)$ is a pair of definition. A ring of definition for $B$ is equal to $B_0 = A_0[\frac{g_1}{s},\dotsc,\frac{g_r}{s}]$, with ideal of definition $J= IB_0$. The Huber ring $B$ is independent of the choice of $(A_0,I)$. Considering $A^+$, we assume without loss of generality that $A_0 \subseteq A^+$. (See Exercise \ref{exercise:ring-of-definitions-contained-in-any-open}.) Then, we define
\begin{equation*}
B^+ = \text{integral closure of } A^+[\frac{g_1}{s},\dotsc,\frac{g_r}{s}] \text{ within $B$}.
\end{equation*}

We claim $B^+$ is a ring of integral elements for $B$. First, $A_0 \subseteq A^+$ and so $B_0 \subseteq B^+$. Therefore, $B^+$ is open in $B$. Second, $B^+$ is integrally closed in $B$ by construction. Finally, we claim that $B^+\subseteq B^{\circ}$. To see this, start by noting that for each $i$, we have $\frac{g_i}{s} \in B_0$, and $B_0 \subseteq  B^{\circ}$ by Proposition \ref{proposition:power-bounded-union-rings-of-definition}. By construction of $B$, if $A_0' \subseteq A$ is {\em any} ring of definition, then the image of $A_0'$ in $B$ is contained in some ring of definition $B_0'$. Therefore, {\em loc.\ cit.}\ implies the image of $A^{\circ}$ in $B$ is contained in $B^{\circ}$. Since $A^+ \subseteq A^{\circ}$, we have shown $B^{\circ}$ contains $A^+[\frac{g_1}{s},\dotsc,\frac{g_r}{s}]$. Finally, $B^+\subseteq B^{\circ}$ because $B^{\circ}$ is itself integrally closed by Proposition \ref{proposition:power-bounded-union-rings-of-definition}, once again. (The integral closure step is generally required. See Exercise \ref{exercise:localization-must-integrally-close}.)

Rational localizations are related to rational subsets, as we now explain. Suppose that $(A,A^+) \rightarrow (B,B^+)$ is the natural morphism of Huber pairs implicit in the prior paragraph. In the middle of \cite[Lemma 1.5(ii)]{Huber-Generalization}, it is proven that we in fact have a natural commuting diagram
\begin{equation}\label{eqn:rational-factorization}
\xymatrix{
\Spa(B,B^+) \ar@{.>}[dr]_-{\cong} \ar[r] & \Spa(A,A^+)\\
 & U(\frac{g_1,\dotsc,g_r}{s}) \ar[u]
}
\end{equation}
where the diagonal arrow is a homeomorphism. More precisely, the rational subsets in $\Spa(B,B^+)$ correspond bijectively with the rational subsets contained in $U$, via the diagonal arrow.  How difficult is the proof? Based on Sections \ref{section:huber} and \ref{section:valuation}, we {\em could} prove that \eqref{eqn:rational-factorization} exists, the diagonal arrow is a bijection, and that the pre-image of a rational subset in $\Spa(A,A^+)$ is rational in $\Spa(B,B^+)$. See Exercise \ref{exercise:rational-domain-factorization}-\ref{exercise:rational-domain-factorization-false}. It is more difficult to show a rational subset in $\Spa(B,B^+)$ maps onto a rational subset in $U$. The issue is that rational subsets of $\Spa(B,B^+)$ are built from elements generating an open $B$-ideal. After clearing denominators, there is no reason for them to generate an open $A$-ideal. The proofs we know rely on knowing {\em a priori} that $U$ is quasi-compact, which is part of Theorem \ref{thm:hubers-main-theorem}. The idea is explained in Exercise \ref{exercise:minimal-modulus-principal}.

\subsection{Construction:\ tensor products (of Tate rings)}\label{subsubsec:tensor-product}
In this section, we explain tensor products $B \otimes_A C$ when $A$ is a Tate ring. We extend the construction to \term{Tate--Huber pairs} (Huber pairs with first entry a Tate ring).

Suppose that $A$ is a Tate ring and $B$ and $C$ are Huber rings with continuous ring morphisms $A \xrightarrow{\varphi} B$ and $A \xrightarrow{\psi} C$. Then, $B$ and $C$ are also Tate rings because the image of a pseudo-uniformizer for $A$ under a continuous ring morphism is a pseudo-uniformizer in the target. See Exercise \ref{exercise:adic-maps}.

We form $R := B \otimes_A C$ algebraically, and now we make it a topological ring. Choose rings of definitions $B_0 \subseteq B$ and $C_0 \subseteq C$. Since $\varphi^{-1}(B_0) \cap \psi^{-1}(C_0)$ is an open subring, it contains a ring of definition $A_0$ (Exercise \ref{exercise:ring-of-definitions-contained-in-any-open}). We define
\begin{equation*}
R_0 := \img(B_0\otimes_{A_0} C_0\rightarrow B\otimes_A C).
\end{equation*}
Now choose $\varpi \in A_0$, a pseudo-uniformizer for $A$. We equip $R_0$ with the $\varpi R_0$-adic topology, making $R_0$ into a topological ring. We make $R$ a ring with a topology by declaring $R_0$ is open in $R$. We leave it as Exercise \ref{exercise:tate-ring-tensors} that this makes $R$ into a topological ring, which is indeed a Tate ring. We dealt with a similar situation in Section \ref{subsection:rational-localization} with rational localizations. A universal property confirms that the definition of $R$ does not depend on the choices made. Namely, the $A$-algebra maps
\begin{align*}
\id \otimes 1 : B &\rightarrow R\\
1 \otimes \id : C &\rightarrow R
\end{align*}
are {\em continuous}, and they are initial with respect to pairs of continuous $A$-algebras maps $B \rightarrow S$ and $C \rightarrow S$.

Extending to Tate--Huber pairs goes like this. If we start with $(A,A^+) \xrightarrow{\varphi} (B,B^+)$ and $(A,A^+) \xrightarrow{\psi} (C,C^+)$, we may define
\begin{equation*}
R^+ = \text{ the integral closure of } \img(B^+ \otimes_{A^+} C^+\rightarrow B \otimes_A C) \subseteq R.
\end{equation*}
This makes $R^+$ integrally closed in $R$. In the construction, we could have assumed $B_0 \subseteq B^+$ and $C_0 \subseteq C^+$ from the start, and so $A_0 \subseteq A^+$. Thus $R^+$ is open. It is left as Exercise \ref{exercise:tensor-power-bounded} that the image of $B^{\circ}\otimes_{A^{\circ}} C^{\circ}$ is contained in $R^{\circ}$, from which the containment $R^+ \subseteq R^{\circ}$ follows.

The difficulty when $A$ is not a Tate ring is defining the topological ring structure on $R$. To do this in general one imposes the condition that $A\rightarrow B$ and $A \rightarrow C$ are adic morphisms of Huber rings, as in Exercise \ref{exercise:adic-maps}. In another direction, an anonymous referee points out that future students may learn how to apply Clausen and Scholze's theory of condensed mathematics and analytic rings to streamline the construction of tensor products.

\subsection{Construction:\ completions}\label{subsubsec:completion-huber-pairs}

The goal here is defining the completion of a Huber ring. The delicate point is that completions are topologically defined with respect to the underlying abelian group, so the algebraic structure of ring needs to be constructed by hand.

Let $A$ be a Huber ring and $(A_0,I)$ a pair of definition. The additive subgroups $I,I^2,I^3,\dotsc$ are a neighborhood basis of zero in $A$. They are also ideals in $A_0$. We can form the ring-theoretic completion 
\begin{equation*}
\widehat{A_0} = \varprojlim_{n} A_0/I^n,
\end{equation*}
which becomes a complete topological ring. Recall, topologically, the quotient ring $A_0/I^n$ is discrete and then $\widehat{A_0}$ is given the subspace topology via the inclusion
\begin{equation*}
\widehat{A_0} \hookrightarrow \prod_{n=1}^\infty A_0/I^n.
\end{equation*}
Since the ideal $I$ is {\em finitely generated}, it is a theorem (see \cite[Tag 05GG]{stacks-project}) that this topology on $\widehat{A_0}$ coincides with the topology defined by the ideal $I\widehat{A_0} \subseteq \widehat{A_0}$, and
\begin{equation*}
\widehat{A_0} \cong \varprojlim_{n} \widehat{A_0}/I^n\widehat{A_0}.
\end{equation*}
Replacing $A_0$ by $A$, the only choice we have for forming the completion is
\begin{equation}\label{eqn:A-completion-topological}
\widehat A = \varprojlim_n A/I^n \hookrightarrow \prod_{n=1}^{\infty} A/I^n.
\end{equation}
This is a complete topological group. The individual factors $A/I^n$ are not rings, but $\widehat{A}$ can be given the structure of a topological ring in three steps.
\begin{enumerate}[label=(\roman*)]
\item It is clear that $\widehat{A}$ is an $\widehat{A_0}$-module. But, one can also show $\widehat{A_0} \subseteq \widehat{A}$ is open and $\widehat{A_0}$ acts by continuous module operations on $\widehat{A}$.
\item There is a natural $A$-module structure on $\widehat{A}$. For $f \in A$, the continuity of the multiplication by $f$ on $A$ makes $f : A/I^{n+r} \rightarrow A/I^{r+1}$ well-defined for some $n \geq 1$ depending on $f$, uniform in $r$. A module structure is thus induced on the projective limit over $r$.
\item Since $A$ is dense in $\widehat{A}$ and $\widehat{A_0}$ is open, there is an additive decomposition $\widehat{A} = A + \widehat{A_0}$. One defines the structure of a ring on $\widehat{A}$ by using the previous module structures and then forcing the distributive law to hold.
\end{enumerate}
The details are outlined as Exercise \ref{exercise:completion}.

Since $\widehat{A_0}$ is open in $\widehat{A}$ and the topology on $\widehat{A_0}$ is the $I\widehat{A_0}$-adic topology, we conclude that $\widehat{A}$ is a Huber ring. The continuous morphism $A \rightarrow \widehat{A}$ is initial for maps $A \rightarrow B$ with $B$ complete. So, $\widehat{A}$ is independent of the initial choice of pair of definition $(A_0,I)$.

 If $(A,A^+)$ is a Huber pair, we can go back to the start and assume that $A_0 \subseteq A^+$. Then, we form the completion
\begin{equation*}
\widehat{A^+} = \varprojlim_{n} A^+/I^n.
\end{equation*}
This defines an open subring of $\widehat{A}$. Unlike the constructions in Sections \ref{subsubsec:rational-localization-huber-pairs}-\ref{subsubsec:tensor-product}, the ring $\widehat{A^+}$ is already integrally closed and so we get a Huber pair $(\widehat{A},\widehat{A^+})$. In addition one can confirm directly that  $\widehat{(A^{\circ})} = (\widehat A)^{\circ}$, so the property ``$A^+ = A^{\circ}$'' is preserved by completions. See Exercise \ref{exercise:completion-integrally-closed}.

Finally, since $\widehat{A}$ is the completion of $A$, the natural map $\Cont(\widehat{A}) \rightarrow \Cont(A)$ is a bijection. It induces, when $(A,A^+)$ is a Huber pair, a canonical map
\begin{equation}\label{eqn:completion-adic-spectrum}
\Spa(\widehat{A},\widehat{A^+}) \rightarrow \Spa(A,A^+)
\end{equation}
that is also bijective. Huber proves \eqref{eqn:completion-adic-spectrum} is a homeomorphism \cite[Proposition 3.9]{Huber-ContinuousValuations}. The difficulties are similar to those we discussed with rational localization already.

\subsection{Identifying $A^+$}\label{subsec:adic-nullstellensatz}
The remaining goal in this lecture is showing that rings of integral elements are intrinsic to adic spectra, just as radical ideals are intrinsic to closed subsets of affine schemes.

\begin{theorem}[The adic Nullstellensatz]\label{theorem:Aplus-recovery}
Let $A$ be a Huber ring. Suppose that $A^+ \subseteq A$ is open and integrally closed. Then,
\begin{equation}\label{eqn:A+-equality-in-theorem}
A^+ = \{f \in A \mid \val{f(x)} \leq 1 \text{ for all $x \in \Spa(A,A^+)$}\}.
\end{equation}
\end{theorem}

By definition, $A^+$ is contained in the right-hand side of \eqref{eqn:A+-equality-in-theorem}. To prove the theorem, we need to prove that if $f \in A$ but $f \not\in A^+$, then there exists a {\em continuous} valuation $x$ on $A$ such that $\val{f(x)} > 1$ and $\val{g(x)}\leq 1$ for all $g \in A^+$. 

The argument occurs in three steps. In the first step, we reduce to the case where $f$ is a unit in $A$. This is where we use that $A^+$ is integrally closed. In the second step, we construct a candidate valuation $x_0$, without imposing a continuity condition. The construction is pure algebra, including a brutal extension of a valuation from one field to a larger field. The extension is so uncontrolled that arguing directly for continuity is hopeless. Therefore, in the third step, we replace $x_0$ by a continuous valuation $x$, while preserving the bounds imposed on $f$ and $g \in A^+$. This is where the openness of $A^+$ is used. The replacement step is based on a process called \term{horizontal specialization}. Properties of horizontal specialization will be given as exercises, but note that it is a fundamental technique in Huber's papers. Seeing the proof of Theorem \ref{theorem:Aplus-recovery} may inspire the reader to study original sources more carefully.

The proof we give of Theorem \ref{theorem:Aplus-recovery} is essentially the same as in \cite[Lemma 3.3(i)]{Huber-ContinuousValuations} and \cite[Theorem 10.3.6]{Conrad-PerfectoidSeminarNotes}. The main difference is that, in the third step, we argue for continuity directly from Proposition \ref{proposition:contiunity-criteria}, whereas other proofs refer to a result \cite[Theorem 3.1]{Huber-ContinuousValuations} that recognizes $\Cont(A)$ within $\Spv(A)$.

For Sections \ref{subsubsec:reduction-Aplus-theorem}, \ref{subsubsec:algebra}, and \ref{subsubsec:specialization}, we reserve $A$ for a fixed Huber ring, $A^+$ for an open and integrally closed subring, and $f$ an element of $A$ such that $f \not\in A^+$.

\subsection{Nullstellensatz:\ The reduction step}\label{subsubsec:reduction-Aplus-theorem}

We seek $x \in X = \Spa(A,A^+)$ with $\val{f(x)} > 1$. In principle, we can limit our search to the open subset $U(\frac{1}{f}) = \{x \in X \mid 1 \leq \val{f(x)}\}$. In terms of rings, we focus on $B = A(\frac{1}{f})$ and $B^+$, which we define to be the integral closure of $A^+[\frac{1}{f}]$ in $A(\frac{1}{f})$.  Note $B$ is a Huber ring, since $\{1,f\}$ generate the unit ideal in $A$. By the argument in Section \ref{subsubsec:rational-localization-huber-pairs}, $B^+\subseteq B$ is open and integrally closed. 

We claim that $f\not\in B^+$. Indeed, if $f \in B^+$, then $f$ is integral over $A^+[\frac{1}{f}]$. Clearing denominators in $A[\frac{1}{f}]$, we find a polynomial relation $f^m + gf^{m-1} + \dotsb = 0$ in $A$, with coefficients $g \in A^+$. This is impossible because $A^+$ is integrally closed and $f \not\in A^+$. So, $f \not\in B^+$. 

Finally, if $x \in \Spa(B,B^+)$ and $\val{f(x)} > 1$, then its image in $\Spa(A,A^+)$ satisfies the same inequality. Replacing $(A,A^+)$ with $(B,B^+)$, we will now assume that
\begin{enumerate}[label=(\roman*)]
\item $f$ is a unit in $A$, and\label{enum-part:Aplus-theorem-reduction-unit}
\item $\frac{1}{f} \in A^+$ but $f \not\in A^+$.\label{enum-part:Aplus-theorem-Aplus-maintain}
\end{enumerate}

\subsection{Nullstellensatz:\ The algebraic argument}\label{subsubsec:algebra}

This part of the argument is pure algebra. We assume \ref{enum-part:Aplus-theorem-reduction-unit} and \ref{enum-part:Aplus-theorem-Aplus-maintain}. We will construct $x_0 \in \Spv(A)$ such that $\val{g(x_0)}\leq 1$ for all $g\in A^+$ while $\val{f(x_0)} > 1$.

By \ref{enum-part:Aplus-theorem-Aplus-maintain}, the element $\frac{1}{f} \in A^+$ is not a unit. Choose a prime $\frakp \subseteq A^+$ with $\frac{1}{f} \in \frakp$. By \ref{enum-part:Aplus-theorem-reduction-unit}, $f$ {\em is} a unit  in $A$. Therefore $\frac{1}{f}$ is definitely not nilpotent in $A^+_{\frakp}$. So, choose a minimal prime $\frakq$ in $A^+$ such that $\frakq \subseteq \frakp$ and $\frac{1}{f} \not\in \frakq$. We consider then the localizations $A^+_{\frakq} \subseteq A_{\frakq}$. A prime ideal of (the non-zero ring) $A_{\frakq}$ contracts to a prime $\frakQ$ of $A$ such that $\frakQ \cap A^+ \subseteq \frakq$. Equality holds since $\frakq$ is minimal among primes in $A^+$. We now have a ring extension $A^+/\frakq \subseteq A/\frakQ$ that gives rise to a field extension
\begin{equation}\label{eqn:Kplus-K-extension}
K^+ = \Frac(A^+/\frakq) \subseteq \Frac(A/\frakQ) = K.
\end{equation}

We now reference commutative algebra and valuation theory. Focusing first just on $K^+$, \cite[Theorem 10.2]{Matsumura-CommutativeRingTheory} implies that we may construct a valuation subring $R^+ \subseteq K^+$ such that 
\begin{equation*}
A^+/\frakq \subseteq R^+ \subseteq K^+ \;\; \text{and}\;\; \frakm_{R^+} \cap A^+/\frakq = \frakp/\frakq.
\end{equation*}
As explained in Section \ref{subsec:support-fibers} and Exercises \ref{exercise:valuation-to-valuation-ring}-\ref{exercise:valuation-valuation-ring-bijection}, there is a unique $x^+ \in \Spv(K^+)$ with $A_{x^+} = R^+$. We view $x^+ \in \Spv(K^+) \cong \supp^{-1}(\frakq) \subseteq \Spv(A^+)$. We then observe:
\begin{enumerate}[label=(\alph*)]
\item Since $A^+/\frakq \subseteq A_x^+$, we have $\val{g(x^+)} \leq 1$ for all $g \in A^+$.\label{enum-part:gAplus-bound}
\item Since $\frac{1}{f} \not\in \frakq$, we have $0 < \val{\frac{1}{f}(x^+)}$. Yet, $\frac{1}{f} \bmod \frakq \in  \frakm_{R^+}$ and so $\val{\frac{1}{f}(x^+)} < 1$.\label{enum-part:fAplus-bound}
\end{enumerate}
Bringing $K$ into the discussion, Chevalley's theorem \cite[Chapter VI, \S 3, no. 3, Proposition 5]{Bourbaki-CommutativeAlgebra} says the inclusion $K^+ \subseteq K$ induces a {\em surjection} $\Spv(K) \twoheadrightarrow \Spv(K^+)$. Therefore, we choose 
\begin{equation*}
x_0 \in \Spv(K) \cong \supp^{-1}(\mathfrak Q) \subseteq \Spv(A)
\end{equation*}
that lifts $x^+$. If $g \in A^+$, then $\val{g(x_0)} = \val{g(x^+)} \leq 1$ by \ref{enum-part:gAplus-bound}. Since $f \in A^\times$, we have $\val{f(x_0)} > 1$ by \ref{enum-part:fAplus-bound}. This completes the construction of $x_0$.

\subsection{Nullstellensatz:\ The specialization maneuver}\label{subsubsec:specialization}
So far, we have $x_0 \in \Spv(A)$ such that $\val{f(x_0)} > 1$, while $\val{g(x_0)} \leq 1$ for all $g \in A^+$. Now we replace $x_0 \in \Spv(A)$ by $x \in \Cont(A)$ without altering the constraints. The rest of the argument relies on  $A^+$ being open in $A$.

As preparation, we examine how close $x_0$ is to being continuous. Let $\Gamma_{0}$ be the value group of $x_0$. By Proposition \ref{proposition:contiunity-criteria}, the continuity of $x_0$ depends on whether or not $\val{h(x_0)}$ is co-final in $\Gamma_0$ for $h \in A^{\circ\circ}$. Consider $s \in A$ with $\val{s(x_0)} \neq 0$ and $h \in A^{\circ\circ}$. Since $A^+$ is open in $A$ and $h$ is topologically nilpotent we have $h^n sf \in A^+$ as $n\rightarrow \infty$. So, $\val{h^nsf(x_0)}\leq 1$ as $n\rightarrow \infty$. Since $\val{s(x_0)}\neq 0$ and $\val{f(x_0)}>1$ we see 
\begin{equation}\label{eqn:Aplus-proof-almost-co-final}
\val{h(x_0)}^n \leq \frac{1}{\val{s(x_0)}\val{f(x_0)}} < \frac{1}{\val{s(x_0)}} \;\;\;\;\; (n \gg 0).
\end{equation}
Strictly speaking, this {\em does not} show $\val{h(x_0)}$ is co-final in $\Gamma_0$, but it is close and we will end up using the estimate \eqref{eqn:Aplus-proof-almost-co-final}. We now adjust $x_0$ in three steps.

\begin{enumerate}[label=(\Roman*)]
\item Let $\Gamma_1 \subseteq \Gamma_0$ be the subgroup generated by all $\val{s(x_0)} \geq 1$ for $s \in A$. The general element of $\Gamma_1$ is
\begin{equation*}
\frac{\val{t(x_0)}}{\val{s(x_0)}}
\end{equation*}
where that $s,t \in A$ and $\val{s(x_0)},\val{t(x_0)} \geq 1$.\label{enum-part:Gamma1}
\item Let $\overline{\Gamma}_1 \subseteq \Gamma_0$ be the convex closure of $\Gamma_1$. This is the subgroup of elements in $\Gamma_0$ that lie between two elements of $\Gamma_1$. See Exercise \ref{exercise:convex-closures}. If $\delta \in \overline{\Gamma}_1$, then \ref{enum-part:Gamma1}  implies there exists $s,t \in A$ with $\val{s(x_0)}$, $\val{t(x_0)} \geq 1$ such that
\begin{equation}\label{eqn:h2-versus-gamma}
\frac{1}{\val{s(x_0)}} \leq  \frac{\val{t(x_0)}}{\val{s(x_0)}} \leq \delta.
\end{equation}
\label{enum-part:Gamma1-closure}
\item We now define $x \in \Spv(A)$. For $s \in A$, set
\begin{equation}\label{eqn:horizontal-specialization-definition}
\val{s(x)} = \begin{cases}
\val{s(x_0)} & \text{if $\val{s(x_0)} \in \overline{\Gamma}_1$;}\\
0 & \text{otherwise}.
\end{cases}
\end{equation}
We leave as  Exercise \ref{exercise:horizontal-specializations} that this is a valuation on $A$. In fact, it is the most extreme case of a process called \term{horizontal specialization}. The same formula defines a valuation if $\overline{\Gamma}_1$ is replaced by any convex subgroup $\Delta \subseteq \Gamma_0$ containing $\Gamma_1$. 
\end{enumerate}

We now argue that $\val{g(x)}\leq 1$ for $g \in A^+$ and  $\val{f(x)} > 1$  and that $x$ is continuous.

\begin{enumerate}[label=\arabic*.]
\item In \eqref{eqn:horizontal-specialization-definition}, we see $\val{s(x)}\leq \val{s(x_0)}$ for all $s \in A$. Given $\val{g(x_0)}\leq 1$ for $g \in A^+$, we therefore also have $\val{g(x)}\leq 1$ for $g \in A^+$.
\item On the other hand, $\val{f(x_0)} \in \Gamma_1 \subseteq \overline{\Gamma}_1$. So, $\val{f(x)} = \val{f(x_0)} > 1$.
\item Finally, suppose $h \in A^{\circ\circ}$ and $\delta \in \overline{\Gamma}_1$ is arbitrary. By \eqref{eqn:Aplus-proof-almost-co-final} and \eqref{eqn:h2-versus-gamma} we may choose $s \in A$ such that $\val{s(x_0)} \neq 0$ and
\begin{equation*}
\val{h(x)}^n \leq \val{h(x_0)}^n < \frac{1}{\val{s(x_0)}} \leq \delta \;\;\;\;\; (n \gg 0).
\end{equation*}
So, $\val{h(x)}$ is co-final in $\overline{\Gamma}_1$, and  $x$ is continuous by Proposition \ref{proposition:contiunity-criteria}.
\end{enumerate}

\subsection*{Section \thesection~ Exercises}

\begin{exercise}\label{exercise:top-nil-always-in-A+}
Let $A$ be a Huber ring.
\begin{enumerate}[label=(\alph*)]
\item Show that if $x \in \Cont(A)$ and $f \in A^{\circ\circ}$, then $\val{f(x)} < 1$.\label{enum-part:top-nil-Acirccirc}
\item Show that if $A^+ $ is a ring of integral elements, then $A^{\circ\circ} \subseteq A^+$.\label{enum-part:Acirccirc-integral-elements}
\end{enumerate}
\end{exercise}
\begin{solution}
\begin{enumerate}[label=(\alph*)]
\item  This follows from the continuity criterion Proposition \ref{proposition:contiunity-criteria}. 
\item Start with $A^+$ being open. If $f \in A^{\circ\circ}$, then $f^n \in A^+$ for $n \gg 0$. Therefore $f$ is integral over $A^+$. Since $A^+$ is integrally closed, we conclude $f \in A^+$.
\end{enumerate}
\end{solution}

\begin{exercise}\label{exercise:pass-to-integral-closure}
Let $A$ be a Huber ring and $\Sigma \subseteq A$ any subset. Define 
\begin{equation*}
\Spa(A,\Sigma) = \{x \in \Cont(A) \mid \val{f(x)} \leq 1 \text{ for all $f \in \Sigma$}\}.
\end{equation*}
Let $A^+$ be the integral closure in $A$ of the subring generated by $\Sigma$ and $A^{\circ\circ}$.
\begin{enumerate}[label=(\alph*)]
\item Show that $A^+$ is open and integrally closed.
\item Show that $\Spa(A,\Sigma) = \Spa(A,A^+)$.\label{exercise-part:pass-to-integral-closure-no-Spa-change}
\end{enumerate}
\end{exercise}
\begin{solution}
\begin{enumerate}[label=(\alph*)]
\item It is clear that $A^+$ is integrally closed. It is open because it contains $A^{\circ\circ}$, which is itself open.
\item Since $\Sigma \subseteq A^+$ we have a containment from left to right. Suppose that $x \in \Spa(A,\Sigma)$. Since $x$ is continuous, $\val{f(x)} < 1$ for $f \in A^{\circ\circ}$ by part \ref{enum-part:top-nil-Acirccirc} of Exercise \ref{exercise:top-nil-always-in-A+}. Therefore $\val{f(x)} \leq 1$ on the subring $B$ generated by $\Sigma$ and $A^{\circ\circ}$. Since $A^+$ is the integral closure of $B$, we need to show that if $f$ is integral over $B$ then $\val{f(x)}\leq 1$. So suppose $f^n + b_{n-1}f^{n-1} + \dotsb + b_0 = 0$ with $b_j \in B$ at least one of which is non-zero. (If they are all zero then $f \in A^{\circ\circ}$ already.) Then
\begin{equation*}
\val{f(x)}^n \leq \max\{ \val{b_j(x)}\val{f(x)}^j \}.
\end{equation*}
Since $\val{b_j(x)} \leq 1$ already we see $\val{f(x)}^n \leq \val{f(x)}^j$ for some $0\leq j < n$. Thus $\val{f(x)}^{n-j}\leq 1$ and from that it follows $\val{f(x)}\leq 1$ as well.
\end{enumerate}
\end{solution}

\begin{exercise}\label{exercise:unit-disc-closure}
Show that
\begin{equation*}
\Spa(\Cp\lr w,\OCp\lr w) = \{x \in \Cont(\Cp\lr w) \mid \val{w(x)} \leq 1\}.
\end{equation*}
\end{exercise}

\begin{exercise}\label{exercise:ring-of-definitions-contained-in-any-open}
Let $A$ be a Huber ring and $B \subseteq A$ an open subring. Show that there exists a ring of definition $A_0$ of $A$ that is contained in $B$.
\end{exercise}
\begin{solution}
Suppose $A_0'$ is any ring of definition. Let $A_0 = A_0' \cap B$. Then $A_0$ is open since $B$ is open and $A_0$ is bounded since it is contained in the bounded subring $A_0'$. (See Exercise \ref{exercise:bounded-exercises}.)
\end{solution}

\begin{exercise}\label{exercise:localization-must-integrally-close}
Let $(A,A^+)$ be a Huber pair and $(B,B^+)$ its rational localization with respect to $g_1,\dotsc,g_r,s$. Show by example that $A^+[\frac{g_1}{s},\dotsc,\frac{g_d}{s}]\neq B^+$, possibly.
\end{exercise}
\begin{solution}
Why would you expect equality? Suppose $A = \Qp[w]$ with the $p$-adic topology and we will localize with respect to $\{g_1,g_2\} = \{p^2,w^2\}$ and $s = \{p^2\}$. Choose $A^+ = A_0 = \Zp[w]$. Then, $B=\Qp[w]$ and
\begin{equation*}
A^+[\frac{g_1}{s},\frac{g_2}{s}] = \Zp[(\frac{w}{p})^2] \subseteq B
\end{equation*}
is not integrally closed. The integral closure is $B^+ = \Zp[w/p]$.
\end{solution}

\begin{exercise}\label{exercise:rational-domain-factorization}
Let $A$ be a Huber ring and assume $g_1,\dotsc,g_r,s \in A$ generate an open $A$-ideal. Let $B = A(\frac{g_1,\dotsc,g_r}{s})$. Given a ring of integral elements $A^+$, define the corresponding Huber pair $(B,B^+)$ as in Section \ref{subsubsec:rational-localization-huber-pairs}.
\begin{enumerate}[label=(\alph*)]
\item Show that the natural map $\Spa(B,B^+) \rightarrow \Spa(A,A^+)$ factors through the rational subset $U(\frac{g_1,\dotsc,g_r}{s})\subseteq \Spa(A,A^+)$.
\item Show that the natural map $\Spa(B,B^+) \rightarrow U(\frac{g_1,\dotsc,g_r}{s})$ is a bijection.
\item Show that the preimage of a rational subset in $\Spa(A,A^+)$ is a rational subset in $\Spa(B,B^+)$.
\end{enumerate}
\end{exercise}
\begin{solution}
\begin{enumerate}[label=(\alph*)]
\item Let $\iota: A \rightarrow B$ be the localization map. Given $x \in \Spa(B,B^+)$ we define $y \in \Spa(A,A^+)$ by $\val{f(y)} = \val{\iota(f)(x)}$. Recall that $s \in B^\times$. So $\val{s(y)} = \val{\iota(s)(x)} \neq 0$. Second, by construction $\frac{\iota(g_j)}{\iota(s)} \in B^+$ for each $j=1,\dotsc,r$. Therefore if $x \in \Spa(B,B^+)$, then $\val{\frac{\iota(g_j)}{\iota(s)}(x)} \leq 1$. Multiplying through we get $\val{\iota(g_j)(x)} \leq \val{\iota(s)(x)} \neq 0$.  This shows that $y \in U(\frac{g_1}{s},\dotsc,\frac{g_r}{s})$.
\item It is clear that $x \mapsto y$ is injective, since valuations are multiplicative. So, we have to show that $\Spa(B,B^+) \rightarrow  U(\frac{g_1}{s},\dotsc,\frac{g_r}{s})$ is onto. Suppose that $y \in U(\frac{g_1}{s},\dotsc,\frac{g_r}{s})$, so $\val{s(y)} \neq 0$. Then $B = A[\frac{1}{s}]$ and we define $x$ as a valuation on $B$ via
\begin{equation*}
\val{\frac{\iota(f)}{\iota(s^n)}(x)} = \frac{\val{f(y)}}{\val{s(y)}^n}.
\end{equation*}
It is not hard to check that this defines a valuation $x$ on $B$. We have $\val{F(x)}\leq 1$ for $F \in A^+[\frac{g_1}{s},\dotsc,\frac{g_r}{s}]$, since $\val{f(y)} \leq 1$ for $f \in A^+$ and $\val{g_j(y)}\leq \val{s(y)}$, since $y \in U(\frac{g_1}{s},\dotsc,\frac{g_r}{s})$. By Exercise \ref{exercise:pass-to-integral-closure}\ref{exercise-part:pass-to-integral-closure-no-Spa-change}, this shows $\val{F(x)}\leq 1$ for $F \in B^+$.
\item For this part, what we want to use is Exercise \ref{exercise:adic-maps} and \ref{exercise:adic-maps-functoriality}. Indeed, what we learn in those exercises is that $A \rightarrow B$ is an adic morphism of Huber rings, and that this means the preimage of a rational subset is a rational subset.
\end{enumerate}
\end{solution}

\begin{exercise}\label{exercise:rational-domain-factorization-false}
Let $A$ be a Huber ring and assume that $g_1,\dotsc,g_r,s \in A$ generate an open $A$-ideal. Let $B=A(\frac{g_1,\dotsc,g_r}{s})$. Define
\begin{equation*}
U = U(\frac{g_1,\dotsc,g_r}{s}) = \{x \in \Cont(A) \mid \val{g_j(x)}\leq \val{s(x)} \neq 0 \text{ for all $j$}\}.
\end{equation*}
Show that the natural map $\Cont(B) \rightarrow \Cont(A)$ does not always factor through $U$.
\end{exercise}

\begin{solution}
Consider $A = \Cp\lr w$. Choose $g_1 = w$ and $s = 1$. So, $B = \Cp\lr w(\frac{w}{1})$. It is clear that $U = D$ is the adic unit disc defined by $x \in \Cont(A)$ such that $\val{w(x)}\leq 1$. But what is $B$? Well, $B = A$ as rings since $B = A[\frac{1}{s}] = A[\frac{1}{1}] = A$. Second, the topology on $B$ is the one where the open sets are the $p^n\OCp\lr w[w] = p^n\OCp\lr w$. Thus, it is the same topology, i.e.\ $A \cong B$ as Huber rings. The map $\Cont(B)\rightarrow \Cont(A)$ is the identity map. Since $x_{1^+} \in \Cont(A) \backslash U$, we conclude the image does not factor through $U$.
\end{solution}

\begin{exercise}\label{exercise:minimal-modulus-principal}
Let $A$ be a Huber ring  and $I$ an ideal of definition. Assume $U$ is quasi-compact in $\Spa(A,A^+)$ and $s \in A$ such that $\val{s(x)}\neq 0$ on $U$.
\begin{enumerate}[label=(\alph*)]
\item Show that if $x \in U$, then there exists $n$ such that $\val{f(x)} \leq \val{s(x)}$ for all $f \in I^n$.
\item Suppose that $g_1,\dotsc,g_r$ is any list of elements of $A$. Show that there exists elements $f_1,\dotsc,f_d \in I$ such that the ideal generated by $g_1,\dotsc,g_r,f_1,\dotsc,f_d$ is open in $A$ and 
\begin{equation*}
U(\frac{g_1,\dotsc,g_r}{s}) =U(\frac{g_1,\dotsc,g_r,f_1,\dotsc,f_d}{s}).
\end{equation*}
\item Show that the map in part (a) of Exercise \ref{exercise:rational-domain-factorization} maps rational subsets to rational subsets.
\end{enumerate}
\end{exercise}
\begin{solution}
See \cite[Lemma 3.11]{Huber-ContinuousValuations} or \cite[Lemma III.3.3]{Morel-AdicSpaceNotes}.
\end{solution}

\begin{exercise}\label{exercise:tate-ring-tensors}
Let $A$, $B$, and $C$ be Tate rings and assume there are continuous ring morphisms $A \rightarrow B$ and $A \rightarrow C$.
\begin{enumerate}[label=(\alph*)]
\item Show that $R = B\otimes_A C$ with the topology defined in Section \ref{subsubsec:tensor-product} is a topological ring.\label{exercise:tate-ring-tensors-topological}
\item Show that $R$ is a Tate ring.\label{exercise:tate-ring-tensors-still-tate}
\item Verify the universal property of the tensor product $R$ with respect to pairs of continuous map $B \rightarrow S$ and $C \rightarrow S$.
\end{enumerate}
\end{exercise}

\begin{hint}
See Exercise \ref{exercise:subring-Iadic} for part \ref{exercise:tate-ring-tensors-topological}.
\end{hint}

\begin{solution}
\begin{enumerate}[label=(\alph*)]
\item Recall the construction of the topology on $R$ is as follows. Fix $A \xrightarrow{\varphi} B$ and $A \xrightarrow{\psi} C$. We assume we have rings of definition $\varphi(A_0) \subseteq B_0$ and $\psi(A_0) \subseteq C_0$. We fix a pseudo-uniformizer $\varpi \in A_0$, so $\varphi(\varpi) \in B_0$ and $\psi(\varpi) \in C_0$ are also pseudo-uniformizers. The topology on $R$ is given by making $R_0 = B_0 \otimes_{A_0} C_0$ an open subring with, and giving it the $\varpi R_0$-adic topology. 

By Exercise \ref{exercise:subring-Iadic}, to show $R$ is a topological ring, it suffices to show that for all simple tensors $b\otimes c$, we have $(b\otimes c)\varpi^n \in R_0$ for some $n \gg 0$. But this is clear, since we can choose $m \gg 0$ such that $b\varphi(\varpi)^m \in B_0$ and $\ell \gg 0$ such that $c \psi(\varpi)^\ell \in C_0$. Then for $n \geq m+\ell$ we have
\begin{equation*}
(b\otimes c)\varpi^{n} \in R_0.
\end{equation*}
\item This is clear since the topology on $R_0$ is given $\varpi R_0$-adically and $\varpi$ is invertible (in either $B_0$ or $C_0$, it does not matter).
\end{enumerate}
\end{solution}

\begin{exercise}\label{exercise:tensor-power-bounded}
Let $A \rightarrow B$ be a continuous morphism of Tate rings.
\begin{enumerate}[label=(\alph*)]
\item Show that the natural map $A^{\circ} \rightarrow B$ factors through $B^{\circ}$.
\item Suppose in addition that $A \rightarrow C$ is a continuous morphism of Tate rings. Show that the natural map $B^{\circ}\otimes_{A^\circ} C^\circ \rightarrow B \otimes_A C$ factors through $(B \otimes_A C)^{\circ}$.
\end{enumerate}
\end{exercise}

\begin{exercise}\label{exercise:completion}
Let $A$ be a Huber ring and $(A_0,I)$ a pair of definition. Define 
\begin{equation*}
\widehat{A} = \varprojlim_n A/I^n \supseteq \widehat{A_0} = \varprojlim_n A_0/I^n.
\end{equation*}
\begin{enumerate}[label=(\alph*)]
\item Show that $\widehat{A_0} \subseteq \widehat{A}$ is open.\label{exercise-part:completion-open-subgroup}
\item Show that if $\widetilde f \in \widehat{A_0}$, then multiplication by $\widetilde f$ is continuous on $\widehat{A}$.
\item If $f \in A$ and $\widetilde g = (g_j) \in \widehat{A}$, show there exists $n$ such that $(fg_n,fg_{n+1},\dotsc)$ has well-defined image in $\widehat{A}$. Show that $\widetilde g \mapsto f\widetilde g$ is continuous on $\widehat{A}$. \label{exercise-part:completion-A-module-structure}
\item Show that if $\widetilde g \in \widehat{A}$ then $\widetilde g = \widetilde g_0 + f$ for some $\widetilde g_0 \in \widehat{A_0}$ and $f \in A$.\label{exercise-part:completion-decompostion}
\item Given $\widetilde g = \widetilde g_0 + f \in \widehat{A}$ and $\widetilde h = \widetilde h_0 + k \in \widehat{A}$ as in part \ref{exercise-part:completion-decompostion}, show 
\begin{equation*}
\widetilde g \widetilde h = \widetilde g_0\widetilde h_0  + f\widetilde h_0 + \widetilde g_0k + fk
\end{equation*}
is well-defined in $\widehat{A}$, and it makes $\widehat{A}$ a topological ring.
\end{enumerate} 
\end{exercise}
\begin{solution}
\begin{enumerate}[label=(\alph*)]
\item Suppose $\widetilde f = (f_1,f_2,\dotsc) \in \widehat{A}$. Thus $f_n - f_1 \in I \subseteq A_0$ for all $n \geq 2$. Therefore, $\widetilde f \in \widehat{A_0}$ if and only if $f_1 \in A_0$. That is, $\widehat{A_0}$ is the pre-image of $A_0/I \subseteq A/I$ under the natural continuous projection $\widehat A \rightarrow A/I$. That makes $\widehat{A_0}$ open.
\item Since $\widehat{A}$ is a topological group and multiplication by $\widetilde{f}$ is an additive group homomorphism, it is enough to check the multiplication map is continuous near zero. But a neighborhood basis of zero is already contained in $\widehat{A_0}$, since $\widehat{A_0}$ is open by part \ref{exercise-part:completion-open-subgroup}. Thus, the continuity follows from the fact that $\widehat{A_0}$ is already a topological ring.
\item Multiplication by $f$ on $A$ is continuous and so $fI^n \subseteq A_0$ for some $n \geq 1$. Therefore, $fg_n \bmod I \in A/I$ {\em is} well-defined, as is $fg_{n+1}\bmod I^2 \in A/I^2$, and so on. Then,
\begin{equation*}
(fg_n,fg_{n+1},\dotsc) \in \widehat A
\end{equation*}
is well-defined. It is also independent of the choice of $n$ (since it is lands in the projective limit). This shows the first part sentence holds. The proof shows the second part as well. Indeed, it shows that $f I^n\widehat{A_0} \subseteq \widehat{A_0}$. Since the topology on $\widehat{A_0}$ is the $I\widehat{A_0}$-adic topology, this shows $f$ the action of is continuous in a neighborhood of zero, which is all that is needed.
\item 
\end{enumerate}

\end{solution}

\begin{exercise}\label{exercise:completion-integrally-closed}
Let $(A,A^+)$ be a Huber pair.

\begin{enumerate}[label=(\alph*)]
\item Show that the completion $\widehat{A^+}$ is integrally closed in $\widehat{A}$.\label{exercise-part:completion-integrally-closed-Aplus}
\item Let $B = \widehat{A}$. Show that $\widehat{A^{\circ}} = B^{\circ}$.\label{exercise-part:completion-integrally-closed-Acirc}
\end{enumerate}
\end{exercise}
\begin{solution}
Part (a). Explain in \cite[Remark 11.5.2]{Conrad-PerfectoidSeminarNotes}.

Part (b). Claimed in \cite[Remark 11.5.2]{Conrad-PerfectoidSeminarNotes} but I do not know a reference and I have not written an argument down. But Morel to the rescue \cite[Proposition II.3.1.12]{Morel-AdicSpaceNotes}.
\end{solution}

\begin{exercise}\label{exercise:convex-closures}
 Suppose that $\Gamma$ is a totally ordered abelian group. Convex subgroups of $\Gamma$ were defined in Exercise \ref{exercise:convex-subgroups}. This exercise shows that if $\Delta \subseteq \Gamma$, then $\Delta$ is always contained in a smallest convex subgroup $\overline \Delta$ called the \term{convex closure} of $\Delta$ within $\Gamma$. Namely, define
\begin{equation*}
\overline{\Delta} = \{\gamma \in \Gamma \mid \delta_1 \leq \gamma \leq \delta_2 \text{ for some } \delta_1,\delta_2 \in \Delta\}.
\end{equation*}
\begin{enumerate}[label=(\alph*)]
\item Show that $\overline{\Delta}$ is a subgroup of $\Gamma$.
\item Show that $\overline{\Delta}$ is a convex.
\item Show that 
\begin{equation*}
\overline{\Delta} = \bigcap_{\substack{\Gamma \supseteq \Delta_0 \supseteq \Delta\\\text{$\Delta_0$ convex}}} \Delta_0.
\end{equation*}
\end{enumerate}
\end{exercise}

\begin{exercise}\label{exercise:horizontal-specializations}
Let $A$ be a ring and $\bval: A \rightarrow \Gamma \cup \{0\}$ a valuation. Assume that $\Delta \subseteq \Gamma_{\bval}$ is a {\em convex} subgroup that contains $\val{f}$ for all $f \in A$ such that $\val{f} \geq 1$.
\begin{enumerate}[label=(\alph*)]
\item Show that
\begin{equation*}
\val{f}_{\Delta} = \begin{cases}
\val{f} & \text{ if $\val{f} \in \Delta$;}\\
0 & \text{otherwise,}
\end{cases}
\end{equation*}
is a valuation on $A$.
\item Show that $\bval_{\Delta}$ lies in the closure of $\bval$ within $\Spv(A)$.
\end{enumerate}
\end{exercise}
\begin{solution}
\begin{enumerate}[label=(\alph*)]
\item See the proof of \cite[Proposition 4.3.6]{Conrad-PerfectoidSeminarNotes} for the proof $\bval_{\Delta}$ is multiplicative. The key point is showing that if $\val{f},\val{g} \not\in \Delta$, then $\val{fg} \not\in\Delta$.

That source leaves the ultra-metric inequality for the reader. Let us fill that gap. We let $f,g \in A$ and without loss of generality we assume that $\val{f}\leq \val{g}$.
\begin{itemize}
\item Suppose $\val{g} \not\in \Delta$. We claim $\val{f+g} \not\in\Delta$ as well. By assumption $\val{g} < 1$ and so by the ultrametric inequality
\begin{equation*}
\val{f+g} \leq \val{g} < 1.
\end{equation*}
Therefore $\val{f+g} \not\in \Delta$ because $\Delta$ is convex. We conclude $\val{f+g}_{\Delta} = 0$. But also $\val{g}_{\Delta} = 0$ by assumption and $\val{f}\leq \val{g} < 1$ implies $\val{f}_{\Delta} = 0$ as well.
\item Now suppose $\val{g} \in \Delta$. Then $\val{f+g}_{\Delta} \leq \val{f+g} \leq \val{g} = \val{g}_{\Delta}$. Since $\val{f}_{\Delta} \leq \val{f} \leq \val{g}$, this case follows as well.
\end{itemize}
\item We need to show that if $g,s \in A$ then $\bval_{\Delta} \in U(\frac{g}{s})$ implies $\bval \in U(\frac{g}{s})$. That is, we need to show
\begin{equation*}
\val{g}_{\Delta} \leq \val{s}_{\Delta} \neq 0 \implies  \val{g} \leq \val{s} \neq 0.
\end{equation*}
Given $\val{s}_{\Delta} \neq 0$ we necessarily have $\val{s} = \val{s}_{\Delta}$ is non-zero. In particular, $\val{s} \in \Delta$. If $\val{g} \in \Delta$ as well, then the inequalities are the same. Suppose $\val{g} \not\in \Delta$. Note that this implies $\val{g} < 1$. Therefore $\val{s} < \val{g} < 1$ is impossible. We conclude $\val{g} \leq \val{s}$.
\end{enumerate}
\end{solution}

\section{The closed unit disc}\label{section:some-points}
In this lecture, we analyze the closed unit disc 
\begin{equation*}
D = \Spa(\Cp\lr w, \OCp\lr w).
\end{equation*}
The reader looking for a complete analysis can consult \cite[Section 11]{Conrad-PerfectoidSeminarNotes}. Our focus will be more narrow and perhaps prepare the learner for a more detailed treatment. 

First, we broadly discuss how to distinguish points in $D$. We have studied $D$ in Example \ref{example:tate-algebra} and Sections \ref{subsec:exotic-Gamma} and \ref{subsec:valuation-final-example}. We revisit those discussions. Second, we focus on the Gauss point $x_1 \in D$.  We saw in Section \ref{subsec:valuation-final-example} that $x_{1^-}$ lies in the closure, but we will systematically produce many similar points. The main theorem (Theorem \ref{theorem:closure-gauss-point}) exactly describes the closure of $x_1$ within $D$.

\subsection{The Huber ring $\Cp\lr w$}\label{subsection:Cp<w>-reminder}

We begin by reviewing $\Cp\lr w$ as a Huber ring. First, recall $\Cp$ is complete and algebraically closed. Its ring of integers is
\begin{equation*}
\OCp = \{\alpha \in \Cp \mid \norm{\alpha}_p \leq 1\}.
\end{equation*}
The maximal ideal $\frakm_{\OCp}$ consists of those $\alpha$ with $\norm{\alpha}_p < 1$. The residue field is 
\begin{equation*}
\Fpbar \cong \Zpbar/\frakm_{\Zpbar} \cong \OCp/\frakm_{\OCp}.
\end{equation*}
Second, the one-variable Tate algebra $\Cp\lr w$ is a principal ideal domain. Up to scalar, the irreducible elements are the linear polynomials $w-\alpha$ with $\alpha \in \OCp$. Thus,
\begin{equation*}
\Spec(\Cp\lr w) = \{\{0\}\} \cup \{\lr{w-\alpha} \mid \alpha \in \Cp\}.
\end{equation*}
This is a consequence of the \term{Weierstrass preparation theorem}, which says that any $f \in \Cp\lr w$ factors uniquely as $f = p^{\mu}Pu$ where $\mu$ is an integer and $P \in \Cp[w]$ is a monic polynomial and $u \in \OCp\lr w^\times$.

The topology on $\Cp\lr w$ is the one induced by the Gauss norm. Thus $\Cp\lr w$ is a Tate ring with pseudo-uniformizer $p$. The following is a specific instance of the continuity criterion for valuations on $\Cp\lr w$.

\begin{proposition}\label{proposition:continuity-criterion-tate-algebra}
Let $x \in \Spv(\Cp \lr w)$ with value group $\Gamma_x$. The following are equivalent:
\begin{enumerate}[label=(\roman*)]
\item The point $x$ belongs to $D$.\label{enum-part:prop:continuity-criterion-tate-algebra-disc-belong}
\item We have $\val{p(x)}$ is co-final in $\Gamma_x$ and $\val{f(x)} \leq 1$ for all $f \in \OCp\lr w$.\label{enum-part:prop:continuity-criterion-tate-algebra-boundOcp}
\end{enumerate}
\end{proposition}
\begin{proof}
By Proposition \ref{proposition:contiunity-criteria}, if $x \in \Cont(\Cp\lr w)$, then $\val{p(x)}$ is co-final in $\Gamma_x$. If $x$ also lies in $D$, then of course $\val{f(x)}\leq 1$ for all $f \in \OCp\lr w$. Therefore \ref{enum-part:prop:continuity-criterion-tate-algebra-disc-belong} implies \ref{enum-part:prop:continuity-criterion-tate-algebra-boundOcp}.

Now suppose $x$ is a valuation and \ref{enum-part:prop:continuity-criterion-tate-algebra-boundOcp} holds. First, since $\val{p(x)}$ is co-final in $\Gamma_x$, we have $\val{p(x)} < 1$. Second, if $f \in \OCp\lr w$, then $\val{f(x)} \leq 1$ by assumption and so
\begin{equation*}
\val{f(x)} \leq 1 < \val{p(x)}^{-1}.
\end{equation*}
Thus $x \in \Cont(\Cp\lr w)$ by Proposition \ref{proposition:contiunity-criteria}. Of course, once $x$ is continuous, it lies in $D$ by assumption. Therefore \ref{enum-part:prop:continuity-criterion-tate-algebra-boundOcp} implies \ref{enum-part:prop:continuity-criterion-tate-algebra-disc-belong}.
\end{proof}

\subsection{Classical and ``disc'' points}\label{subsec:classical-and-disc-points}
Proposition \ref{proposition:continuity-criterion-tate-algebra} simplifies checking whether valuations on $\Cp\lr w$ lie in $D$. Or, at least, it practically reduces a supposition ``$x \in D$'' to simply checking $x$ defines a valuation. Let us tally several points of $D$, called classical and (nested, perhaps) disc points.

\subsubsection{Classical points}
Let $\alpha \in \OCp$. Then there is a point $x_\alpha \in D$ given by
\begin{equation*}
\val{f(x_\alpha)} = \norm{f(\alpha)}_{p}
\end{equation*}
for all $f \in \Cp\lr w$. This defines a valuation since the $p$-adic norm is a valuation. It is the {\em only} $x \in D$ with $\supp(x) = \lr{w-\alpha}$. See Exercise \ref{exercise:fiber-closed-point}.

\subsubsection{Disc points}
Suppose $0 < r \leq 1$ and $\alpha \in \OCp$. We define
\begin{equation*}
D_r(\alpha) = \{\alpha' \in \OCp \mid \norm{\alpha - \alpha'}_p \leq r\}.
\end{equation*}
This is the closed disc of radius $r$ centered at $\alpha$. If $f \in \Cp\lr w$, it has a series expansion
\begin{equation*}
f = b_0 + b_1(w-\alpha) + b_2(w-\alpha)^2 + \dotsb 
\end{equation*}
with $\displaystyle \lim_{i\rightarrow \infty} b_i = 0$. We define $x_{\alpha,r} \in D$ according to
\begin{equation*}
\val{f(x_{\alpha,r})} = \max_{i\geq 0} \norm{b_i}_p r^i \in \R_{\geq 0}.
\end{equation*}
Given $x_{\alpha,r}$ is a valuation, it lies in $D$ by Proposition \ref{proposition:continuity-criterion-tate-algebra}. Indeed $\val{p(x_{\alpha,r})} = \frac{1}{p}$ is co-final in $\R_{>0}$ and if $b_i \in \OCp$ for all $i$, then $\val{f(x_{\alpha,r})} \leq 1$, clearly. We leave as Exercise \ref{exercise:disc-valuation} to check that $x_{\alpha,r}$ is a valuation and
\begin{equation*}
\val{f(x_{\alpha,r})} = \sup_{\alpha' \in D_r(\alpha)} \norm{f(\alpha')}_p.
\end{equation*}
So, $x_{\alpha,r}$ depends only on $D_r(\alpha)$, rather than $\alpha$. We write $x_D = x_{\alpha,r}$ if $D = D_r(\alpha)$.

\subsubsection{Nested discs}
Suppose that $D_{\bullet} : D_0 \supseteq D_1 \supseteq D_2 \supseteq \dotsb$ is a sequence of discs in $\OCp$. Then, we define
\begin{equation*}
\val{f(x_{D_{\bullet}})} = \inf_{i\geq 0} \val{f(x_{D_i})}.
\end{equation*}
Admitting this defines a valuation, it is continuous by Proposition \ref{proposition:continuity-criterion-tate-algebra}. Discussing why this is a valuation is not a priority in this lecture. See \cite[Section 11.3]{Conrad-PerfectoidSeminarNotes}, instead. We at least point out there are three possible behaviors:
\begin{itemize}
\item The intersection $\bigcap_{i \geq 0} D_i$ may be a single point $\alpha \in \OCp$. Then, $x_{D_{\bullet}} = x_\alpha$.
\item The intersection $\bigcap_{i\geq 0} D_i$ may be another disc $D$. Then, $x_{D_{\bullet}} = x_D$.
\item The intersection $\bigcap_{i\geq 0} D_i$ may be empty! This option is available because $\Cp$ is a metric field that is  {\em not} \term{spherically complete}. In that case, $x_{D_{\bullet}}$ is a valuation we have not yet constructed, but on which we will not dwell.
\end{itemize}

\subsection{Review of valuation rings}\label{subsec:valuation-ring-review}
Points in $D$ are classified, in part, by their valuation rings and residue fields. So, we review the  constructions from Section \ref{subsec:support-fibers} and Exercises \ref{exercise:valuation-to-valuation-ring}-\ref{exercise:valuation-valuation-ring-bijection}.

Let $K$ be a field. A subring $A \subseteq K$ is a valuation ring if, given $\alpha \in K^\times$, either $\alpha$ or $\alpha^{-1}$ belong to $A$. If $A$ is a valuation ring, the non-zero principal fractional ideals $\{\alpha A \mid \alpha \in K^\times\}$ are totally ordered by inclusion. In fact, $\alpha \mapsto \alpha A$ defines a bijection
\begin{equation*}
K^\times/A^\times \leftrightarrow \{\alpha A \mid \alpha \in K^\times\}.
\end{equation*}
Therefore, the group $\Gamma_A = K^\times/A^\times$ is naturally a totally ordered abelian group. In terms of cosets, this order is $\alpha A^\times \leq \beta A^\times$ if and only if $\alpha\beta^{-1} \in A$.

The natural function $K \rightarrow \Gamma_A \cup \{0\}$ defines a valuation on $K$. We write $x_A$ for its equivalence class. One of the exercises mentioned is to show $x_A \leftrightarrow A$ defines a bijection between $\Spv(K)$ and the set of valuation subrings of $K$. The inverse is 
\begin{equation*}
x\mapsto A_x = \{\alpha \in K \mid \val{\alpha(x)} \leq 1\}.
\end{equation*}
The ring $A_x$ has a maximal ideal $\frakm_x$. Its residue field is $A_x/\frakm_x$.

\subsection{High-level classification of points in $D$}\label{subsec:D-points-classification}
Points in $D$ are often classified into ``types''  called Types 1-5. See \cite[Example 2.20]{Scholze-Perfectoid} or \cite[p.\ 7-8]{AWS-Perfectoid}. Here, we describe the classification without proof, augmented by listing auxiliary data. For each $x \in D$, we look at its support $\supp(x)$, its value group $\Gamma_x$, and its residue field $A_x/\frakm_x$.

\begin{warning} 
The residue fields $A_x/\frakm_x$ are all characteristic $p$ fields. They are  {\em different} than any kind of residue field gotten by viewing $x \in D$ as a point in a $\Cp$-rigid analytic variety. Those geometric residue fields arise from the structure sheaf over $D$. They are all characteristic zero fields. 
\end{warning}

We will additionally indicate whether $x \in D$ is closed, and then we will finally list the Type. The result is compiled in Table \ref{table:points-in-D}. All but the bottom row of the table has been explained in Section \ref{subsec:classical-and-disc-points}. 

\begin{table}[htp]
\renewcommand{\arraystretch}{1.3}
\setlength{\tabcolsep}{4pt}
\caption{Classification of points in the closed unit disc $D$.}
\begin{center}
\begin{tabular}{|L|L|L|L|L|L|}
\hline
Name & $\supp(x)$ & $\Gamma_x$ & $A_x/\frakm_x$ & Closed? & Type\\
\hline\hline
$x_\alpha$ & $\lr{w-\alpha}$ & $p^{\Q}$ & $\Fpbar$ & Closed & 1 \\
\hline
$x_{\alpha,r}$ ($r \in p^{\Q}$) &  $\{0\}$ & $p^{\Q}$ & $\Fpbar(t)$ & Non-closed & 2\\
\hline
$x_{\alpha,r}$  ($r\not\in p^{\Q}$)  & $\{0\}$ & $p^{\Q}r^{\Z} \subseteq \R^\times$ & $\Fpbar$ & Closed & 3\\
\hline
$x_{D_{\bullet}}$ ($\cap D_i = \emptyset$) & $\{0\}$ & $p^{\Q}$ & $\Fpbar$ & Closed & 4\\
\hline
$x_{\alpha,r}^{\lambda}$ ($r \in p^{\Q}$) & $\{0\}$ & $p^{\Q} \times (\frac{1}{2})^{\Z}$ & $\Fpbar$ & Closed & 5\\
\hline
\end{tabular}
\end{center}
\label{table:points-in-D}
\end{table}%

The classical points $x_\alpha$ are the Type 1 points. The second and third rows list the disc points $x_{\alpha,r}$, but the data is separated according to whether or not $r \in p^{\Q} = \norm{\Cp^\times}_p$. When $r$ lies in $p^{\Q}$, the point $x_{\alpha,r}$ is not closed in $D$. We have already seen this phenomenon at the Gauss point $x_{0,1} = x_1$. In this case, we also see the valuation ring residue field is a transcendental extension of $\Fpbar$. The field generator $t$ is, essentially, the reduction of the coordinate function on the boundary of the rational disc. These phenomena do not occur for disc points of  radius not in the value group of $\Cp$. For those points, instead, the value group is larger than that of $\Cp$. The fourth row shows the data for the nested disc points that are neither true disc points nor classical points. 

The final row indicates a type of point that we have not yet seen. These points are parametrized by $\lambda \in \Ps^1(\Fpbar)$. They comprise the non-trivial points in the closure of $x_{\alpha,r}$ whenever $r \in p^{\Q}$. Their value group is a product group with the lexicographic order. Two examples of such points are $x_{0,1}^{0} = x_{1^-}$ and $x_{0,1}^{\infty} = x_{1^+}$ from Section \ref{subsec:valuation-final-example}, with the caveat that if $(\alpha,r) = (0,1)$, then we do not allow $\lambda = \infty$ in Table \ref{table:points-in-D}, since $x_{1^+} \not\in D$. (See Exercise \ref{exercise:exotic-valuation-explicit} for these calculations.)

The goal of Sections \ref{subsection:valuation-ring-construction} and \ref{subsection:gauss-closure} is to explain the final row, to make precise the construction of $x_{0,1}^{\lambda}$ for $\lambda \in \Aone(\Fpbar)$. We will discover these points while simultaneously showing they form the non-trivial points in the closure of the Gauss point within $D$. Before that, we sketch a cartoon of $D$ that is meant to {\em suggest} the $x_{0,1}^{\lambda}$'s exist.

\subsection{A schematic drawing, focused on the Gauss point}\label{subsection:disc-schematic}

We pause to draw a cartoon. We will sketch what $D$ looks like near the Gauss point $x = x_{0,1}$. Versions of this picture are drawn in other places, for instance \cite[Section 11.3]{Conrad-PerfectoidSeminarNotes} or \cite[Example 2.20]{Scholze-Perfectoid}. We have no intention of giving mathematical meaning to these drawings. That is why we have abstract algebra!

\begin{figure}[htp]
\centering
\includegraphics{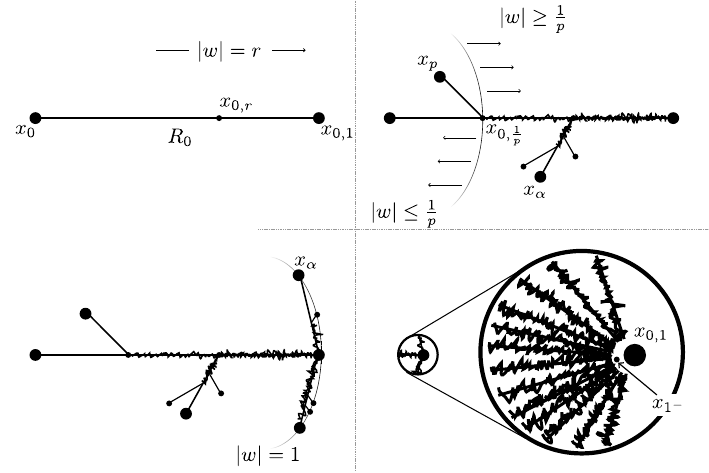}
\caption{Drawing of the closed unit disc $D$ from the perspective of the Gauss point.}
\label{fig:unit-disc}
\end{figure}

There are four steps to create our drawing. They are shown sequentially in Figure \ref{fig:unit-disc} and explained in writing now.

First, we plot the Gauss point $x_{0,1}$ and the classical point $x_0$, which we perceive as the origin of $D$. For now, both points are drawn as simple black dots, even though $x_{0,1}$ is non-closed in $D$. We connect these points with a ray $R_0$. The points of the ray are the disc points $x_{0,r}$ with $0 \leq r \leq 1$. The horizontal scale $r$ measures $\val{w}$ over $D$.

Second, we add the classical point $x_p$. As explained in Section \ref{subsec:classical-and-disc-points}, disc points depend only on physical discs. Since $D_1(0) = D_1(p)$, we see that $x_{0,1} = x_{p,1}$. There is no new Gauss point to consider from $x_p$'s perspective. Analogous to the first step, we draw the ray connecting $x_{p}$ to the Gauss point. The two rays overlap  because
\begin{equation*}
x_{0,r} = x_{p,r} \iff \frac{1}{p} \leq r.
\end{equation*} 
Overlap is drawn as a thicker, fuzzier, line. The process can be repeated for $x_\alpha$ with $\alpha \in \frakm_{\OCp}$. For each $x_\alpha$, its ray to the Gauss point intersects $R_0$ starting at $x_{0,\norm{\alpha}_p}$.

Third, if $\alpha \in \OCp^\times$, then $D_r(\alpha) \neq D_r(0)$ for all $r < 1$. Therefore, the ray from $x_\alpha$ to the Gauss point intersects $R_0$ {\em only} at $x_{0,1}$. The corresponding points $x_\alpha$ somehow lie on the boundary of $D$. Of course, different $x_\alpha$'s can have intersecting rays. Indeed, the fundamental principle we are repeatedly using is that if $\alpha,\alpha' \in \OCp$, then
\begin{align*}
\alpha \equiv \alpha' \bmod \frakm_{\OCp} &\iff D_r(\alpha) = D_r(\alpha') \text{ for some $r < 1$}\\
&\iff R_\alpha \cap R_{\alpha'} \supsetneq \{x_{0,1}\}.
\end{align*}
(Here $R_\alpha$ is the ray connecting $x_\alpha$ to the Gauss point.)

Finally, we zoom in on the Gauss point. We find rays parametrized by 
\begin{equation*}
\OCp/\frakm_{\OCp} \cong \Fpbar = \Aone(\Fpbar).
\end{equation*}
Zooming in with higher and higher magnification will eventually omit {\em any} given classical point and {\em any} given disc point not equal to the Gauss point. But it will never omit $x_{1^-}$, which lies in the closure of the Gauss point by Section \ref{subsec:valuation-final-example}. To turn this picture into mathematics, we are going to explain how to place $x_{1^-}$ in a $\Aone(\Fpbar)$-parametrized set of points, forming the non-trivial points in the closure of the Gauss point.

\subsection{The key valuation ring construction}\label{subsection:valuation-ring-construction}
This section complements the review of valuation rings in Section \ref{subsec:valuation-ring-review}. Let $K$ be a field, and let $A \subseteq K$ be a valuation subring with maximal ideal $\frakm$. To keep notations clear, let $L = A/\frakm$. Denote by $\pi: A \rightarrow L$ the natural quotient map.

Suppose $B \subseteq A$ is also a valuation subring of $K$. Then, $\frakm \subseteq B$. Indeed, if $x \in \frakm$ is non-zero then $x^{-1} \not\in A$ and therefore $x^{-1} \not\in B$. Since $B$ is a valuation ring, we conclude $x \in B$ as claimed. Thus, $\overline B = B/\frakm$ makes sense and is a valuation subring of $L$. We leave it as Exercise \ref{exercise:lifting-valuation-rings} to check the converse, i.e.\ that if $\overline B \subseteq L$ is a valuation subring of $L$ then $\pi^{-1}(B) \subseteq A$ is a valuation subring of $K$. Admitting that, we get inverse bijections
\begin{align}
\{\text{valuation rings $B \subseteq A$}\} &\overset{\simeq}{\longleftrightarrow} \{\text{valuation subrings $\overline B \subseteq L$}\}\label{eqn:valuation-ring-quotient}\\
B &\mapsto B/\frakm\nonumber\\
\pi^{-1}(\overline B) &\mapsfrom \overline B.\nonumber
\end{align}
Suppose we have a pair $B \leftrightarrow \overline B$ under \eqref{eqn:valuation-ring-quotient}. Write $x \in \Spv(K)$ for the valuation corresponding to $A$, write $y \in \Spv(K)$ for the one corresponding to $B$, and $\lambda \in \Spv(L)$ for the one determined by $\overline B$. As Exercise \ref{exercise:lifting-valuation-rings}, check the following statement as well:
\begin{enumerate}[label=(\alph*)]
\item The three value groups are arranged in a natural exact sequence
\begin{equation}\label{eqn:lifting-valuation-ring-ses}
1 \rightarrow \Gamma_{\lambda} \rightarrow \Gamma_y \rightarrow \Gamma_x \rightarrow 1.
\end{equation}
\item Within $\Spv(K)$ we have the specialization relation $y \in \overline{\{x\}}$.
\end{enumerate}

\begin{example}\label{example:valuations-function-field}
We will use \eqref{eqn:valuation-ring-quotient} to construct $y \in \Spv(K)$ from $\lambda \in \Spv(L)$. The relevant $L$ is $L \cong \Fpbar(w)$, so let us describe $\Spv(L)$ in that case.

Let $\lambda \in \Pone(\Fpbar)$. For $f \in \Fpbar(w)$, write $\ord_{\lambda} f$ for the order of vanishing of $f$ at $w=\lambda$.\footnote{Recall, $\ord_{\infty} f = \deg(f)$.} We thus get a valuation $\bval_\lambda: \Fpbar(w) \rightarrow \R_{\geq 0}$ by
\begin{equation*}
\val{f}_{\lambda} = (\frac{1}{2})^{\ord_{\lambda} f}.
\end{equation*}
The value group is $\Gamma_\lambda = (\frac{1}{2})^{\Z}$.  Ostrowski's theorem is that the non-trivial elements of $\Spv(\Fpbar(w))$ are represented by the $\bval_\lambda$'s. The proof is recalled as Exercise \ref{exercise:ostrowskis-theorem}.
\end{example}

\subsection{The closure of the Gauss point}\label{subsection:gauss-closure}

We can now describe the closure of the Gauss point in $D$. We begin with two lemmas.

\begin{lemma}\label{lemma:tate-ring-specialization}
Let $A$ be a Tate ring. For $x,y \in \Cont(A)$, the following are equivalent: 
\begin{enumerate}[label=(\roman*)]
\item The point $y$ lies in the closure $\overline{\{x\}}$.\label{lemma-part:tate-ring-specialization-closure}
\item We have $\supp(x) = \supp(y)$ and $A_y \subseteq A_x$.\label{lemma-part:tate-ring-specialization-supports}
\end{enumerate}
\end{lemma}
\begin{proof}
We will show that if $y \in \overline{\{x\}}$, then  $\supp(y) \subseteq \supp(x)$. This is where the Tate condition on $A$ is used. The rest of \ref{lemma-part:tate-ring-specialization-closure} implies \ref{lemma-part:tate-ring-specialization-supports}, and all of \ref{lemma-part:tate-ring-specialization-supports} implies \ref{lemma-part:tate-ring-specialization-closure} are left as Exercise \ref{exercise:supports-specialization}. (The Tate condition can also be weakened. See Exercise \ref{exercise:supports-analytic}.)

Since $A$ is a Tate ring, we may choose a pseudo-uniformizer $\varpi \in A$. Since $y$ is a continuous valuation, $\val{\varpi(y)}$ is co-final in the value group $\Gamma_y$. Therefore,
\begin{align*}
f \in \supp(y) &\iff \val{f(y)} = 0\\
&\iff \val{f(y)} \leq \val{\varpi^n(y)}  & \text{(for all $n$).}
\end{align*}
The same equivalences hold if $x$ replaces $y$. Yet, if $y \in \overline{\{x\}}$ then, by definition,
\begin{equation*}
\val{f(y)}\leq \val{\varpi^n(y)} \implies \val{f(x)} \leq \val{\varpi^n(x)},
\end{equation*}
for all $n$. Therefore, $\supp(y) \subseteq \supp(x)$. 
\end{proof}

For the remainder of this section, we let $x = x_{0,1} \in D$. Let $K = \Frac(\Cp\lr w)$ and let $A_x \subseteq K$ be the valuation ring of $x$, with maximal ideal $\frakm_x$ and residue field $L_x$. Our second lemma determines the field $L_x$. Note that 
\begin{equation*}
\OCp\lr w \subseteq A_x = \{ f \in K \mid \val{f(x)}\leq 1\}.
\end{equation*}
This containment is strict because $\frac{1}{w}$ belongs to $A_x$ but not $\OCp\lr w$. More generally, $A_x$ contains $(w-\alpha)^{-1}$ for any $\alpha \in \OCp$, and none of those elements lie in $\OCp\lr w$. These examples are the only essential difference at the level of residue fields.

\begin{lemma}\label{lemma:Gauss-point-residue-field}
The inclusion $\OCp\lr w \subseteq A_x$ induces an isomorphism $\Fpbar(w) \cong L_x$.
\end{lemma}
\begin{proof}
Note that $\frakm_x \cap \OCp\lr w = \frakm_{\OCp}\lr w$. Therefore,
\begin{equation*}
\Fpbar[w] \cong \OCp[w]/\frakm_{\OCp}[w] \cong \OCp\lr w/\frakm_{\OCp}\lr w \subseteq L_x,
\end{equation*}
which extends to $\Fpbar(w) \subseteq L_x$. We claim {\em this} inclusion is an equality.

To start, by Weierstrass preparation any fraction $f/g \in K$ can be expressed as
\begin{equation*}
\frac{f}{g} = \frac{P}{Q}u,
\end{equation*}
where $P,Q \in \Cp[w]$ and $u \in \OCp\lr w^\times$. We may multiply $Q$ by a non-zero scalar so that $Q \in \OCp[w]$ but not $\frakm_{\OCp}[w]$. Then, scale $P$ by the same factor.

Suppose now $f/g \in A_x$. Since $u \in \OCp\lr w^\times$, we have $\val{u(x)} = 1$. Therefore,
\begin{equation*}
\val{P(x)} \leq \val{Q(x)}.
\end{equation*}
Since $Q \in \OCp[w]$, we also have $P \in \OCp[w]$ by definition of $x$. Thus $P/Q \bmod \frakm_x \in \Fpbar(w)$. Since $u \bmod \frakm_x \in \Fpbar[w]$, we find $f/g \bmod \frakm_x \in \Fpbar(w)$, as claimed.
\end{proof}

It {\em seems} the discussion in this section and the previous one gives bijections
\begin{align*}
\Pone(\Fpbar) &\leftrightarrow \{\overline{B} \subseteq \Fpbar(w) \text{ a valuation subring}\} & \text{(Example \ref{example:valuations-function-field})}\\
&\leftrightarrow \{\overline{B} \subseteq L_x \text{ a valuation subring}\} & \text{(Lemma  \ref{lemma:Gauss-point-residue-field})}\\
&\leftrightarrow \{A_y \subseteq A_x \text{ a valuation subring of $K$}\} & \text{(Section \ref{subsection:valuation-ring-construction})}\\
&\leftrightarrow y \in \overline{\{x\}} & \text{(Lemma \ref{lemma:tate-ring-specialization})}.
\end{align*}
However, there is one caveat. For $\lambda \in \Pone(\Fpbar)$ we get a point $x_{0,1}^{\lambda} \in \Spv(\Cp\lr w)$ via the first three bijections, and $x_{0,1}^{\lambda} \in \overline{\{x\}}$ within the valuation spectrum, by the proof of Lemma \ref{lemma:tate-ring-specialization}. However, we need to show $x_{0,1}^{\lambda}$ is actually continuous!
\begin{theorem}\label{theorem:closure-gauss-point}
\leavevmode
\begin{enumerate}[label=(\alph*)]
\item Each $x_{0,1}^{\lambda}$ is a continuous valuation on $\Cp\lr w$.\label{theorem-part:closure-gauss-point-continuous}
\item If $\lambda \neq \infty$, then $x_{0,1}^{\lambda} \in D$. \label{theorem-part:closure-gauss-point-infinity}
\item Within $D$ we have $\overline{\{x_{0,1}\}} = \{x_{0,1}\} \cup \{x_{0,1}^{\lambda} \mid \lambda \in \Aone(\Fpbar)\}$.
\label{theorem-part:closure-gauss-point-final}
\end{enumerate}
\end{theorem}
\begin{proof}
The majority of the proof is an explicit analysis in support of \ref{theorem-part:closure-gauss-point-continuous}.

Let $\pi: A_x \rightarrow L_x$ be the natural projection map. Note $\Fpbar(w) \cong L_x$ by Lemma \ref{lemma:Gauss-point-residue-field} and, in this identification, $\pi(\OCp) \subseteq \Fpbar$. For $\lambda \in \Pone(\Fpbar)$ we let $\overline B_\lambda \subseteq L_x$ be the valuation ring arising from $\lambda$ in Example \ref{example:valuations-function-field}. By direct examination, $\Fpbar \subseteq \overline B_\lambda$ and therefore $\OCp\subseteq \pi^{-1}(\overline B_\lambda)$. 

The value group of $x$ is identified as $K^\times/A_x^\times$ by Section \ref{subsec:valuation-ring-review}. Let $y = x_{0,1}^{\lambda}$ and $A_y \subseteq A_x$ be its valuation ring. The prior paragraph shows that $\OCp \subseteq A_y$. Therefore, the value group $K^\times/A_y^\times$ fits into a diagram
\begin{equation*}
\xymatrix{
\Cp^\times \ar[r]^-{y} \ar@{>>}[d]  & K^\times/A_y^\times \ar@{>>}[d] \\
\Cp^\times/\OCp^\times \ar[r]^-{\simeq}_-{x} \ar@{.>}[ur] & K^\times/A_x^\times.
}
\end{equation*}
The diagonal arrow is injective because the bottom arrow is. Thus, the exact sequence
\begin{equation*}
1 \rightarrow \Gamma_{\lambda} \rightarrow K^\times/A_y^\times \rightarrow K^\times/A_x^\times \rightarrow 1
\end{equation*}
in \eqref{eqn:lifting-valuation-ring-ses} is split as  totally ordered abelian groups. From Example \ref{example:valuations-function-field}, the kernel is $\Gamma_{\lambda} \cong (\frac{1}{2})^{\Z}$. Moreover, part \ref{exercise-part:convex-subgroups-splitting} of Exercise \ref{exercise:convex-subgroups} implies the ordering on
\begin{equation*}
\Gamma_y \cong \Gamma_x \times (\frac{1}{2})^{\Z}
\end{equation*}
is the {\em lexicographic order}. The analysis  in Section \ref{subsec:valuation-final-example} now shows $y$ defines a continuous valuation. This completes the proof of \ref{theorem-part:closure-gauss-point-continuous}. 

For \ref{theorem-part:closure-gauss-point-infinity} we need to see that $\OCp\lr w \subseteq A_y$ if and only if $\lambda \neq \infty$. By the bijection in Section \ref{subsection:valuation-ring-construction}, it is equivalent to see that $\Fpbar[w] \subseteq \overline B_\lambda$ if and only if $\lambda \neq \infty$. But that is clear by definition of the $\lambda$-adic valuation on $\Fpbar(w)$. Finally, \ref{theorem-part:closure-gauss-point-final} follows from \ref{theorem-part:closure-gauss-point-continuous} and \ref{theorem-part:closure-gauss-point-infinity}, together with the overall discussion preceding the theorem statement.
\end{proof}

The reader can work as Exercise \ref{exercise:exotic-valuation-explicit} that the splitting used to prove \ref{theorem-part:closure-gauss-point-continuous} gives rise to identifications $x_{0,1}^{0} = x_{1^-}$ and $x_{0,1}^{\infty} = x_{1^+}$ as in Section \ref{subsec:valuation-final-example}. We also outline, as Exercise \ref{exercise:closure-adjustments}, the adjustments required to adapt the process to other points of $D$. Namely, for $r < 1$ and $\alpha \in \OCp$, we see in Section \ref{subsec:D-points-classification} that $x_{\alpha,r}$ is a non-closed point when $r$ lies in the value group of $\Cp$. In fact, the closure will be in bijection with $\Pone(\Fpbar)$. (The exception in Theorem \ref{theorem:closure-gauss-point}\ref{theorem-part:closure-gauss-point-infinity} disappears.)

The analysis of {\em all} the points in $D$ is incomplete, since we did not say anything at all about Type 3 or 4 points in Table \ref{table:points-in-D}. The reader who wants all the details can see \cite[Section 11]{Conrad-PerfectoidSeminarNotes}. Or, note that if $\Cp$ is replaced by a non-Archimedean field whose value group is $\R_{>0}$ and which is spherically complete, then the only points that exist are Types 1, 2, and 5. Type 3 points become Type 2 and Type 4 points become Type 1. From this perspective, the introduction of the Type 5 points is really the heart of the adic unit disc beyond the classical and rational disc points.

\subsection{Final comments}\label{subsection:final-comments}

What might the reader look at next? In Lemma \ref{lemma:tate-ring-specialization}, we realized every point in the closure of the Gauss point in $D$ has generic support $\{0\}$. This is a more general phenomenon, occurring at \term{analytic} points on adic spectra. We introduce this notion in Exercises \ref{exercise:analytic} and \ref{exercise:supports-analytic}. 

The reader would do well to understand the notion of analytic points, next. The Lemma \ref{lemma:tate-ring-specialization} we proved implicitly uses that every point on an adic specturm of a Tate ring is analytic. Analytic points are used more generally to analyze adic spectra. A good target theorem for a learner would be \cite[Proposition 3.6]{Huber-ContinuousValuations} on whether $\Spa(A,A^+)$ is empty or not. We discussed this on page \pageref{pageref:emptiness}. It is plausible to unwind the argument of that result using the tools outlined in these notes (cf.\ \cite[Section 11.6]{Conrad-PerfectoidSeminarNotes} and \cite[Section III.4.4]{Morel-AdicSpaceNotes}). In doing so, the reader will need to follow the construction of certain spaces $\Spv(A,I)$ of ``valuations with support conditions'' introduced by Huber. The benefit of doing so would be that these spaces with support conditions, and specialization arguments as in Section \ref{subsubsec:specialization}, are crucially applied in proving Huber's Theorem \ref{thm:hubers-main-theorem} on the geometric structure of $\Spa(A,A^+)$.

The other option, hopefully one the reader has already begun, is plowing ahead with the sheaf theory on adic spaces and perfectoid spaces outlined in the sibling lectures \cite{Hubner-AdicSpaces,Johansson-Sheaves,Heuer-PerfectoidLectures}. It is completely plausible for users of Huber's theory of adic spaces to never truly need to study the proof of Theorem \ref{thm:hubers-main-theorem}, as long as their intuition is guided by enough examples (as  Section \ref{section:some-points} tries to illustrate).

\subsection*{Section \thesection~ Exercises}

\begin{exercise}\label{exercise:disc-valuation}
Let $\alpha \in \OCp$ and $r$ a real number with $0 < r\leq 1$. For $f \in \Cp\lr w$ we write
\begin{equation*}
f = b_0 + b_1(w-\alpha) + b_2(w-\alpha)^2 + \dotsb
\end{equation*}
with $b_i \in \Cp$ converging to zero as $i \rightarrow \infty$. Let
\begin{equation*}
\val{f(x_{\alpha,r})} = \max_{i\geq 0} \norm{b_i}r^i.
\end{equation*}
\begin{enumerate}[label=(\alph*)]
\item Show that $x_{\alpha,r}$ defines a (continuous) valuation on $\Cp\lr w$.\label{exercise-part:disc-valuation-is-valuation}
\item Show that 
\begin{equation*}
f \mapsto \sup_{\alpha'\in D_r(\alpha)} \norm{f(\alpha')}_p
\end{equation*} 
is continuous on $\Cp\lr w$.\label{exercise-part:disc-valuation-sup-norm-cont}
\item Show that
\begin{equation*}
\val{f(x_{\alpha,r})} = \sup_{\alpha' \in D_{r}(\alpha)} \norm{f(\alpha')}_p.
\end{equation*}\label{exercise-part:disc-valuation-sup-norm}
\end{enumerate}
\end{exercise}
\begin{hint}
Once \ref{exercise-part:disc-valuation-is-valuation} and \ref{exercise-part:disc-valuation-sup-norm-cont} are shown, \ref{exercise-part:disc-valuation-sup-norm} can be checked directly on polynomials. It may be helpful to study the case $r \in \norm{\Cp^\times}_p = p^{\Q}$ first and do the general case as a limiting process.
\end{hint}

\begin{solution}
\begin{enumerate}[label=(\alph*)]
\item See \cite[Lemma 11.2.3]{Conrad-PerfectoidSeminarNotes}. (The reference only works on $\Cp[w]$ technically, but the proof itself works fine.)
\item If $f \in p^m \OCp\lr w$ then $f \in p^m\OCp\lr{w-\alpha}$. Write $f = \sum b_i (w-\alpha)^i$. Then, for all $\norm{\alpha-\alpha'}_p \leq r$, we have
\begin{equation*}
 \norm{f(\alpha')}_{p} \leq p^{-m}\cdot  \max_{i\geq 0} \norm{b_i}_p r^i
\end{equation*}
Thus if $f$ is small then so is $\norm{f(\alpha')}_p$, uniformly in $\alpha' \in D_r(\alpha)$. This shows the conclusion.
\item See hint.
\end{enumerate}
\end{solution}

\begin{exercise}\label{exercise:lifting-valuation-rings} 
Let $K$ be a field and $A \subseteq K$ a valuation ring with maximal ideal $\frakm$. Let $L = A/\frakm$ be the residue field of $A$ and $\pi : A \rightarrow L$ be the natural projection.
\begin{enumerate}[label=(\alph*)]
\item Show that if $\overline B \subseteq L$ is a valuation subring, then $B = \pi^{-1}(\overline B) \subseteq A$ is a valuation subring of $K$ as well.
\item By Exercise \ref{exercise:valuation-ring-to-valuation}, the valuation rings $A$, $B$, and $\overline B$ correspond to valuations on $K$, $K$, and $L$. Show that the corresponding value groups sit in a natural exact sequence
\begin{equation*}
1 \rightarrow \Gamma_{\overline B} \rightarrow \Gamma_B \rightarrow \Gamma_A \rightarrow 1.
\end{equation*}
\item Let $x_A,x_B \in \Spv(K)$ be the valuations corresponding to $A$ and $B$, respectively. Show that $x_B \in \overline{\{x_A\}}$ within $\Spv(K)$.
\end{enumerate}
\end{exercise}
\begin{solution}
\begin{enumerate}[label=(\alph*)]
\item Replacing $\alpha$ by $\alpha^{-1}$ if necessary, the fundamental claim is that if $\alpha \in A \subseteq K$, but $\pi(\alpha) \not\in B$, then $\pi(\alpha^{-1}) \in B$. Of course, we first must observe that if $\pi(\alpha) \not\in B$, then $\alpha \in A$, because $\pi(\alpha) \neq 0$. The fact that $\pi(\alpha^{-1}) \in B$ now follows from $B$ being assumed to be a valuation subring of $L$.
\item Recall that $\Gamma_A = K^\times/A^\times$. So, the containment $B \subseteq A$ gives the natural quotient map
\begin{equation*}
K^\times/B^\times \rightarrow K^\times/A^\times \rightarrow 1.
\end{equation*}
The kernel of this quotient map is
\begin{equation*}
A^\times/B^\times \overset{\pi}{\simeq} L^\times/\overline B^\times =: \Gamma_{\overline B}
\end{equation*}
\item This is essentially part \ref{exercise-part:valuation-to-valuation-ring-closure} from Exercise \ref{exercise:valuation-to-valuation-ring}.
\end{enumerate}
\end{solution}

\begin{exercise}\label{exercise:ostrowskis-theorem}
Suppose that $F$ is an algebraically closed field. For $\lambda \in \Pone(F)$ define $\bval_\lambda: F(w) \rightarrow \R_{\geq 0}$
by
\begin{equation*}
\val{f}_\lambda = (\frac{1}{2})^{\ord_{w=\lambda}(f)}.
\end{equation*}
(Note that $\val{f}_\infty = 2^{-\deg(f)}$, for $f \in F[w]$.)
\begin{enumerate}[label=(\alph*)]
\item Show that $\bval_\lambda$ is a valuation for all $\lambda$.
\end{enumerate}
Now suppose $F(w) \xrightarrow{\bval} \Gamma\cup \{0\}$ is any valuation and $\val{\alpha} \leq 1$ for all $\alpha \in F$. (This condition on the scalars is automatic if $F = \Fpbar$ because every non-zero element of $\Fpbar$ is a root of unity.)
\begin{enumerate}[label=(\alph*)]
\setcounter{enumi}{1}
\item If $\val{w} > 1$, show that  $\val{a_0 + a_1 w + \dotsb + a_n w^n} = \val{w}^n$ for all $a_0,\dotsc,a_n \in F$ with $a_n \neq 0$. Conclude that $\bval$ is equivalent to $\bval_\infty$.
\item Now suppose $\val{w} \leq 1$. 
\begin{itemize}
\item Show that $\frakp = \{f \in F[w] \mid \val{f(w)} < 1\}$ is a prime ideal in $F[w]$.
\item Show that if $\frakp = \{0\}$, then $\bval$ is equivalent to $\bval_{\triv}$.
\item Show that if $\frakp = \lr{w-\lambda}$, then $\bval$ is equivalent to $\bval_{\lambda}$.
\end{itemize}
\end{enumerate}
\end{exercise}
\begin{solution}
\end{solution}

\begin{exercise}\label{exercise:supports-specialization}
Let $A$ be any ring. Let $x,y \in \Spv(A)$.
\begin{enumerate}[label=(\alph*)]
\item Suppose that $y \in \overline{\{x\}}$. Show that $\supp(x) \subseteq \supp(y)$.
\item Show that if $\supp(x) = \supp(y)$ and $A_y \subseteq A_x$, then $y \in \overline{\{x\}}$.
\end{enumerate}
\end{exercise}
\begin{solution}
\end{solution}

\begin{exercise}\label{exercise:analytic}
Let $A$ be a Huber ring and $x \in \Cont(A)$. Recall, $\supp(x) \subseteq A$ is always a closed prime ideal. We call $x$ \term{analytic} if $\supp(x)$ is {\em not} open in $A$. Show that the following are equivalent for $x \in \Cont(A)$:
\begin{enumerate}[label=(\roman*)]
\item The point $x$ is analytic.\label{exercise-part:analytic-not-open}
\item There exists $f \in A^{\circ\circ}$ such that $\val{f(x)} \neq 0$.\label{exercise-part:analytic-topologically-nilpotent}
\item For any ideal of definition $I$ there exists $f \in I$ such that $\val{f(x)}\neq 0$.\label{exercise-part:analytic-any-ideal}
\end{enumerate}
\end{exercise}
\begin{solution}
If \ref{exercise-part:analytic-not-open} holds, then there exists some ideal of definition $I$ such that $I \not\subseteq \supp(x)$. Therefore, for some $f \in I$ we have $\val{f(x)} \neq 0$. Since $f \in A^{\circ\circ}$ we see that \ref{exercise-part:analytic-topologically-nilpotent} holds. So, \ref{exercise-part:analytic-not-open} implies \ref{exercise-part:analytic-topologically-nilpotent}.

We now show  \ref{exercise-part:analytic-topologically-nilpotent} implies \ref{exercise-part:analytic-any-ideal}. Suppose $f \in A^{\circ\circ}$ such that $\val{f(x)}\neq 0$ and $I$ is any ideal of definition. Then $f^d \in I$ for some $d \geq 0$ and $\val{f(x)^d} \neq 0$. This proves \ref{enum-part:analytic-any-ideal} holds.

Finally, \ref{exercise-part:analytic-any-ideal} implies \ref{exercise-part:analytic-not-open} since $\supp(x)$ being open would imply it contains some ideal of definition.
\end{solution}

\begin{exercise}\label{exercise:supports-analytic}
Let $A$ be a topological ring. Show that if $x \in \Cont(A)$ is analytic and $y \in \Cont(A)$ lies in $\overline{\{x\}}$, then $\supp(y) = \supp(x)$. Therefore, if $x$ is analytic, then its only specializations are vertical. (See Exercise \ref{exercise:vertical-specialization}.)
\end{exercise}

\begin{exercise}\label{exercise:exotic-valuation-explicit}
Let $\lambda \in \Pone(\Fpbar)$. Consider $x_{0,1}^{\lambda}$ constructed in Section \ref{subsection:gauss-closure}. Show that
\begin{equation*}
\Cp\lr w \xrightarrow{x_{0,1}^{\lambda}} \R_{>0} \times \R_{>0} \cup \{0\}
\end{equation*}
can be defined by the following recipe:
\begin{itemize}
\item First, $\val{p(x_{0,1}^{\lambda})} = (\frac{1}{p},1)$.
\item Second, if $f \in \OCp\lr w$ and $\bar f \in \Fpbar[w]$ is its reduction modulo $\frakm_{\OCp}$, then
\begin{equation*}
\val{f(x_{0,1}^{\lambda})} = (\val{f}_1,\val{\bar f}_{\lambda}).
\end{equation*}
\end{itemize}
Conclude that $x_{0,1}^{0} = x_{1^-}$ and $x_{0,1}^{\infty} = x_{1^+}$.
\end{exercise}

\begin{exercise}\label{exercise:closure-adjustments}
Let $\alpha \in \OCp$ and $\beta \in \frakm_{\OCp}$. Let $r = \norm{\beta}_p < 1$. Let $x = x_{\alpha,r}$ be the disc point given in Section \ref{subsec:classical-and-disc-points}.
\begin{enumerate}[label=(\alph*)]
\item Show that $t=\frac{w-\alpha}{\beta} \in A_x$ and there is a natural isomorphism $\Fpbar(t) \cong A_x/\frakm_x$.
\item Show that the containment $\OCp\lr w \subseteq A_x$ has image $\Fpbar$ in $A_x/\frakm_x$.
\item Show that inside $D$, the closure of $x$ is given by 
\begin{equation*}
\overline{\{x_{\alpha,r}\}} = \{x_{\alpha,r}\} \cup \{x_{\alpha,r}^{\lambda} \mid \lambda \in \Pone(\Fpbar)\},
\end{equation*}
for points $x_{\alpha,r}^{\lambda}$ constructed via the mechanism of Section \ref{subsection:valuation-ring-construction}.
\end{enumerate}
\end{exercise}
\begin{solution}
\begin{enumerate}[label=(\alph*)]
\item 
\item Let' assume $a$. Then, the point is that 
\begin{equation*}
w = \alpha + (w-\alpha) \in \frakm_{x}
\end{equation*}
since $r < 1$. It follows that the image of $\OCp\lr w$ in $A_{x_{\alpha,r}}/\frakm_{x_{\alpha,r}}$ is contained within $\Fpbar$. 
\item 
\end{enumerate}
\end{solution}

\bibliography{huber-bibliography.bib}
\bibliographystyle{alpha}

\end{document}